\documentclass{amsart}

\usepackage{amsmath, amsthm, mathrsfs, mathtools, amssymb}
\usepackage{graphicx, enumerate}
\usepackage{siunitx, booktabs, multirow} 
\usepackage{caption,subcaption} 
\usepackage{color}
\usepackage{bm} 

\usepackage{upgreek} 

\usepackage[hidelinks]{hyperref} 
\allowdisplaybreaks 

\newcommand{\N}{\mathbb{N}}

\newcommand{\Z}{\mathbb{Z}}
\newcommand{\R}{\mathbb{R}}
\newcommand{\T}{\mathbb{T}}
\newcommand{\Real}{\mathbb{R}}
\newcommand{\Torus}{\mathbb{T}}
\newcommand{\K}{\mathcal{K}}
\newcommand{\Ocal}{\mathcal{O}}
\newcommand{\Lcal}{\mathcal{L}}

\newcommand{\Fcal}{\mathcal{F}}

\newcommand{\tp}{\top} 

\newcommand*{\iu}{\ensuremath{\mathrm{i}}} 
\newcommand*{\itD}{\varDelta} 
\newcommand*{\dxi}{\itD\xi} 
\newcommand*{\dx}{\itD x} 
\newcommand*{\period}{L} 

\newcommand{\Rper}{\Real^n_{\text{per}}}
\newcommand{\Ltwo}{\mathbf{L}^{2}}
\newcommand{\Hone}{\mathbf{H}^1}

\DeclareMathOperator*{\id}{Id}

\newcommand{\sumZ}[2]{\sum_{#1\in\Z}#2}

\newcommand{\sumNxi}[2]{\dxi \sum_{#1 = 0}^{n-1}#2}
\newcommand{\intR}[2]{\int_{\R}#2\,d#1}

\newcommand{\abs}[1]{\left|#1\right|}

\newcommand{\normHone}[1]{\left\|#1\right\|_{\Hone}}

\newcommand{\fracpar}[2]{\frac{\partial #1}{\partial #2}}

\newcommand*{\A}[1]{\mathrm{A}[{#1}]}

\newcommand*{\D}{\mathrm{D}}
\newcommand*{\Dt}{\mathrm{D}^\tp}

\DeclareMathOperator{\tr}{tr} 

\DeclareMathOperator{\sech}{sech}
\DeclareMathOperator{\sgn}{sgn}

\newtheorem{thm}{Theorem}[section]

\theoremstyle{remark}
\newtheorem{rem}[thm]{Remark}

\theoremstyle{definition}

\newtheorem{prop}[thm]{Proposition}

\theoremstyle{definition}
\newtheorem{exmp}[thm]{Example}

\newcommand{\arxiv}[1]{\href{http://arxiv.org/abs/#1}{arXiv:#1}} 

\newcommand{\ds}{\displaystyle}

\begin{document}

\title[Numerics for  variational CH discretizations]{A numerical study of variational discretizations of the Camassa--Holm equation}

\author[S.\ T.\ Galtung]{Sondre Tesdal Galtung}
\address[S.\ T.\ Galtung]{Department of Mathematical Sciences, NTNU -- Norwegian University of Science and Technology, 7491 Trondheim, Norway}
\email{sondre.galtung@ntnu.no}

\author[K.\ Grunert]{Katrin Grunert}
\address[K.\ Grunert]{Department of Mathematical Sciences, NTNU -- Norwegian University of Science and Technology, 7491 Trondheim, Norway}
\email{katrin.grunert@ntnu.no}

\thanks{Research supported by the grants {\it Waves and Nonlinear Phenomena (WaNP)} and {\it Wave Phenomena and Stability --- a Shocking Combination (WaPheS)} from the Research Council of Norway.}

\begin{abstract}
	We present two semidiscretizations of the Camassa--Holm equation in periodic domains based on
	variational formulations and energy conservation. The first is a periodic version of an existing
	conservative multipeakon method on the real line, for which we propose efficient computation
	algorithms inspired by works of Camassa and collaborators.
	The second method, and of primary interest, is the periodic counterpart of a novel discretization of a
	two-component Camassa--Holm system based on variational principles in Lagrangian variables.
	Applying explicit ODE solvers to integrate in time, we compare the variational discretizations to
	existing methods over several numerical examples.
\end{abstract}
	
\maketitle

\section{Introduction}

The Camassa--Holm (CH) equation
\begin{equation}
u_{t}-u_{txx} + 3uu_{x} -2u_{x}u_{xx} - uu_{xxx} = 0,
\label{eq:CH}
\end{equation}
was presented in \cite{Camassa1993} as a model for shallow water waves, where $u = u(t,x)$ is the fluid velocity at position $x$
at time $t$, and the subscripts denote partial derivatives with respect to these variables.
Equation \eqref{eq:CH} can also be seen as a geodesic equation, see \cite{MR1674267,cons:01b,MR2016696}. This paper focuses on
numerical schemes that are inspired by this interpretation, and more specifically the flow map or Lagrangian point of view for
the equation.
We mention that the CH equation also turns up in models for hyperelastic rods \cite{MR2192287,MR1811323,MR2292515},
 and that it is known to have appeared first in \cite{Fokas1981} as a member of a family of completelyintegrable evolution equations.
Due to its rich mathematical structure and interesting properties, \eqref{eq:CH} has been widely
studied.
For instance it is bi-Hamiltonian \cite{Fokas1981}, has infinitely many conserved quantities, see, e.g., \cite{MR2125239}, and its solutions may
develop singularities in finite time even for smooth initial data, see, e.g., \cite{constantin1998global, constantin1998wave}.
Moreover, serveral extensions and generalisations  of the Camassa--Hom equation exist, but we will only consider one of them,
which is now commonly referred to as the two-component Camassa--Holm (2CH)
system. It was first introduced in \cite[Eq.\ (43)]{Olver1996}, and can be written as
\begin{equation}
\begin{cases}
u_{t} - u_{txx} + 3uu_{x} - 2u_{x}u_{xx} -uu_{xxx} + \rho \rho_{x} &= 0, \\
\rho_{t} + (\rho u)_{x} &= 0.
\end{cases}
\label{eq:2CH}
\end{equation}
That is \eqref{eq:CH} has been augmented with a term accounting for the contribution of the fluid density $\rho = \rho(t,x)$, and
paired with a conservation law for this density.

Since the paper \cite{Camassa1993} by Camassa and Holm there have been numerous works on \eqref{eq:CH}, and its extension
\eqref{eq:2CH} has also been widely studied.
Naturally, there has also been proposed a great variety of numerical methods with these equations in mind,
and here we will mention just a handful of them.
An adaptive finite volume method for peakons was introduced in \cite{NumSim_ArtSchr2006}.
In \cite{Holden2006,Coclite2008} finite difference schemes were proved to converge to dissipative solutions of
\eqref{eq:CH}, while invariant-preserving finite difference schemes for \eqref{eq:CH} and \eqref{eq:2CH} were
studied numerically in \cite{LiuPen2016}.
Pseudospectral, or Fourier collocation, methods for the CH equation were studied in \cite{Kalisch2005, Kalisch2006},
where in the latter paper the authors also proved a convergence result for the method.
In \cite{CamHuaLee2005,Camassa2008,Raynaud2006,CheLiuPen2012} the authors consider particle methods for \eqref{eq:CH} based on
its Hamiltonian formulation, which are shown to converge under suitable assumptions on the initial data.
On a related note, a numerical method based on the conservative multipeakon solution \cite{holden2007globalmp}
of \eqref{eq:CH} was presented in \cite{NumMP_HolRay2008}.
Furthermore, there have been proposed several Galerkin finite element methods for \eqref{eq:CH}:
an adaptive local discontinuous method was presented in \cite{XuShu2008}, a Hamiltonian-conserving scheme was
studied in \cite{Matsuo2010}, while \cite{AnDoMi2019} presented a Galerkin method with error estimates.
There have also been proposed more geometrically oriented methods, such as a geometric finite difference scheme
based on discrete gradient methods \cite{CohRay2011diff}, and multi-symplectic methods for both \eqref{eq:CH} and
\eqref{eq:2CH} in \cite{Cohen2008, Cohen2014}.
Moreover, \cite{CohRay2012num} presents a numerical method for \eqref{eq:CH} based on direct discretization of the equivalent
Lagrangian system of \cite{holden2007global}.
Such a list can never be exhaustive, and for more numerical schemes we refer to the most recent papers
mentioned above and the references therein.

In this paper however, we consider energy-preserving discretizations for \eqref{eq:CH} and \eqref{eq:2CH}, which are closely related to variational principles in \cite{vardisc} and \cite{holden2007globalmp}.
In particular, we are interested in studying how well the discretizations in \cite{vardisc} and \cite{holden2007globalmp} serve as numerical methods.
To this end, we will consider the initial value problem of \eqref{eq:2CH}, with periodic boundary conditions in
order to obtain a computationally viable numerical scheme, i.e.,
\begin{equation}
\begin{cases}
u_{t} - u_{txx} + 3uu_{x} - 2u_{x}u_{xx} -uu_{xxx} + \rho\rho_{x} = 0, & (t,x) \in \R_{+} \times \T, \\
\rho_{t} + (\rho u)_{x} = 0, & (t,x) \in \R_{+} \times \T, \\
u(0,x) = u_{0}(x), & x \in \T, \\
\rho(0,x) = \rho_{0}(x), & x \in \T.
\end{cases}
\label{eq:2CH_IVP}
\end{equation}
Here $\T$ denotes some one-dimensional torus, and we assume $u_0 \in \Hone(\T)$ and $\rho_0 \in \Ltwo(\T)$.
Observe that the choice $\rho_{0}(x) \equiv 0$ in \eqref{eq:2CH_IVP} yields the initial value problem
for \eqref{eq:CH}.

One of the hallmarks of the CH equation, and also the 2CH system,  is the fact that even for smooth initial data, its solutions
can develop singularities, also known as wave breaking.
Specifically, this means that the wave profile $u$ remains bounded, while the slope $u_x$ becomes unbounded from below. At the same time energy may concentrate on sets of measure zero. 
This scenario is now well understood and has been described in
\cite{constantin1998global,constantin1998wave,constantin2000global,Grunert2015}.
A fully analytical description of a solution which breaks is provided by the peakon-antipeakon example, see
\cite{holden2007globalmp}.
An important motivation for the discretizations derived in \cite{vardisc} and  \cite{holden2007globalmp} was for them to be able to handle
such singularity formation, and we will see examples of this in our final numerical simulations.

The variational derivation of the equation as a geodesic equation is based on Lagrangian variables,
and the Lagrangian framework is an essential ingredient in the construction of global conservative solutions, see
\cite{bressan2007global,holden2007global,Grunert2012}. The other essential ingredient is the addition of an extra energy
variable to the system of governing equations, which tracks the concentration of energy on sets of measure zero.
Later we will see that these ingredients have all been accounted for in our discretization. 

Next we will outline how the variational derivation of the CH equation is carried out in the periodic setting,
before we turn to our discrete methods.
In our setting, we take the period to be $\period > 0$ such that
\begin{equation*}
u(t,x+\period) = u(t,x), \qquad \rho(t,x+\period) = \rho(t,x)
\end{equation*}
for $t \ge 0$.
We introduce the characteristics $y(t,\xi)$ and the Lagrangian variables
\begin{align*}
y_t(t,\xi) &= u(t,y(t,\xi)) \eqqcolon U(t,\xi), \qquad r(t,\xi) \coloneqq \rho(t,y(t,\xi)) y_\xi(t,\xi).
\end{align*}
Furthermore, we require the periodic boundary conditions
\begin{equation}
y(t,\xi+\period) = y(t,\xi)+\period, \quad U(t,\xi+\period) = U(t,\xi), \quad r(t,\xi+\period) = r(t,\xi).
\label{eq:periodic_cont}
\end{equation}

Let us ignore $r$ for the moment by setting $r \equiv 0$, which corresponds to studying the CH equation.
A rather straight-forward discretization of the above variables comes from replacing the continuous parameter $\xi$
by a discrete parameter $\xi_i$ for $i$ in a set of indices.
The pairs $(y_i, U_i)$ can then be considered as position and velocity pairs for a set of discrete particles.
We want to derive the governing equations of the discrete system from an Euler--Lagrange principle.
The system of equations will thus be fully determined once we have a corresponding Lagrangian $\Lcal(y, U)$.

We base the construction of the discrete Lagrangian on the continuous case. For the CH equation, a Lagrangian
formulation is already available from the variational derivation of the equation. Let us briefly review this derivation. The
motion of a particle, labeled by the variable $\xi$, is described by the function $y(t, \xi)$. The velocity of the particle is
given by $y_t(t, \xi) = U(t, \xi)$.
The Eulerian velocity $u$ is given in the same reference frame through $u(t, y(t,\xi)) = U(t,\xi)$, and the
energy is given by a scalar product in the Eulerian frame.
For the CH equation, the scalar product $\left<\cdot,\cdot\right>$ is given by the $\Hone$-norm
\begin{equation}
\label{eq:contscalprod}
\left< u, u \right> = \int_\Torus (u^2 + u_x^2)\,dx  = \int_\Torus \left(y_t^2y_\xi + \frac{y_{t\xi}^2}{y_\xi}\right) d\xi.
\end{equation}
Other choices of the scalar product lead to the Burgers or Hunter--Saxton equations, see Table \ref{tab:eqsum}.
\begin{table}[h]
  \centering
  \begin{tabular}[t]{r|ccc}
    name & Burgers & Hunter--Saxton & Camassa--Holm\\
    \hline
    equation & $u_t + 3uu_x = 0$ & $(u_t + uu_x)_x = \frac12u_x^2$ & \eqref{eq:CH} \\
    (semi)-norm & $\int_\R u^2\,dx$ &$\int_\R u_x^2\,dx$ &$\int_\R u^2 + u_x^2\,dx$ \\
    momentum & $u$ & $u_{xx}$ & $u - u_{xx}$\\
    soliton-like solutions & not defined & piecewise linear & multi-peakons
  \end{tabular}
  \caption{Summary of norms and corresponding soliton-like solutions.}
  \label{tab:eqsum}
\end{table}
From the scalar product, we define the momentum $m$ as the function which satisfies
\begin{equation}
\label{eq:defmomcont}
\left< u, v \right> = \int_\Torus m(x)v(x)\,dx,
\end{equation}
for all $v$. Note that this scalar product is invariant with respect to relabeling of the particles,
or right invariant in the terminology of \cite{arnold1999topological}.
This means that for any diffeomorphism, also called relabeling function, $\phi(\xi)$, the transformation
$y\mapsto y\circ \phi$ and $U\mapsto U\circ\phi$ leaves the energy invariant:
\begin{equation*}
\left< u, u \right> = \left< U\circ\phi\circ (y\circ \phi)^{-1}, U\circ\phi\circ (y\circ \phi)^{-1} \right>
=  \left< U\circ y^{-1}, U\circ y^{-1} \right>.
\end{equation*}
By Noether's theorem, this invariance leads to the conserved quantity
\begin{equation*}
m_\text{c} = m\circ y y_\xi^2,
\end{equation*}
which is presented as the first Euler theorem in \cite{arnold1999topological}. We can recover the governing equation using the
conserved quantity $m_\text{c}$. We have
\begin{align*}
\fracpar{}{t}(m\circ y y_\xi^2)  = m_t\circ y y_\xi^2 + m_x\circ y y_ty_\xi^2 + m\circ yy_{\xi,t}y_{\xi} = 0.
\end{align*}
We use the definition of $u$ as $y_t = u\circ y$ and, after simplification, we obtain
\begin{equation*}
m_t + m_xu + 2mu_x = 0,
\end{equation*}
which is exactl´y \eqref{eq:CH}.

One method for discretizing the CH equation comes from its multipeakon solution, as studied in \cite{holden2007globalmp}.
This solution of \eqref{eq:CH} is a consequence of that the class of functions of the form
\begin{equation*}
m(t,x) = (u - u_{xx})(t,x) = \sum_{i = 1}^{n}U_i(t)\delta(x  - y_i(t))
\end{equation*}
is preserved by the equation.
By deriving an ODE system for $y_i$ and $U_i$ which define the position and height of the peaks, we can deduce their values at
any time $t$.
Then, given the points $(y_i,U_i)$ for $i \in \{1,\dots,n\}$, we can reconstruct the solution on the whole line by joining these
points with linear combinations of the exponentials $e^x$ and $e^{-x}$.
For the new scheme, we use instead a linear reconstruction, which is also the standard approach in finite difference methods.
In this case we approximate the energy in Lagrangian variables using finite differences for $y_i$ and $U_i$, and then
the corresponding Euler--Lagrange equation defines their time evolution.
Finally, we apply a piecewise linear reconstruction to interpolate $(y_i,U_i)$ for $i \in \{1,\dots,n\}$.

Comparing these two reconstruction methods, we face a trade-off in how we interpolate the points $(y_i,U_i)$.
Although the piecewise exponential reconstruction provides an exact solution of \eqref{eq:CH}, one may, in absence of additional
information on the initial data, consider it less natural to use these catenary curves to join the points instead of the
more standard linear interpolation.
On the other hand, linear reconstruction may approximate the initial data better, but an additional error is introduced
since piecewise linear functions are not preserved by the equation, see Figure \ref{fig:illus}.
\begin{figure}[h]
  \centering
  \includegraphics[width=0.6\textwidth]{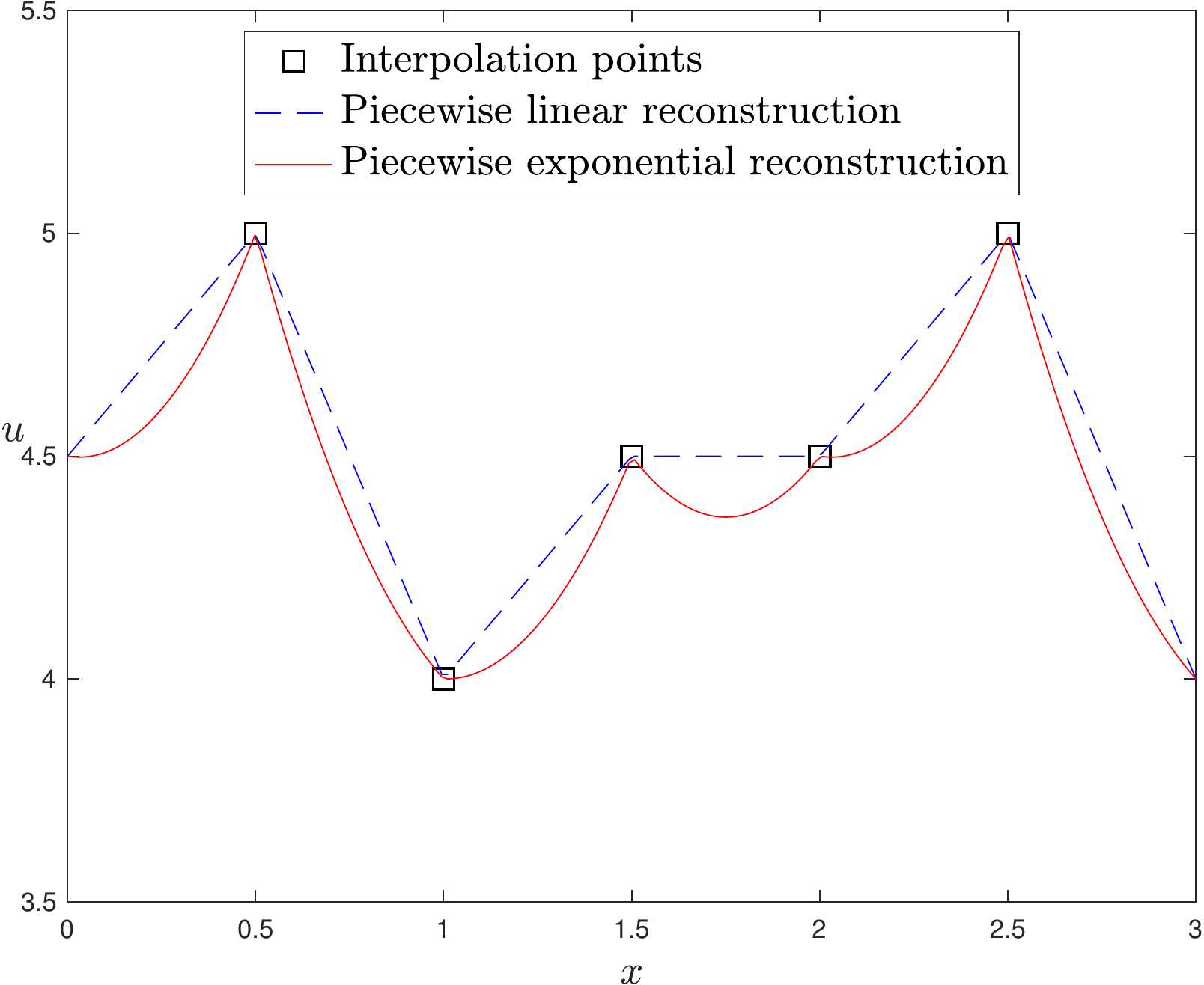}
  \caption{Piecewise exponential and linear reconstruction.}
  \label{fig:illus}
\end{figure}

Note that multipeakon solutions are not available in the case
$\rho\neq 0$, cf.\ \cite{ConsIvan2008}, and so the method based on linear reconstruction is the only scheme presented here
for the 2CH system which is based on variational principles in Lagrangian coordinates.
We remark that for the Hunter--Saxton equation, the soliton-like solutions
are piecewise linear, being solutions of $u_{xx} =\sum_{i \in \Z}U_i\delta(x - y_i)$.
Thus, the linear and the exact soliton reconstruction coincide for the Hunter--Saxton equation.
As a matter of fact, in \cite{grunert2020numerical} there has recently been developed a fully discrete numerical
method for conservative solutions of the Hunter--Saxton equation which is primarily set in Eulerian coordinates,
but employs characteristics to handle wave breaking.

The rest of this paper is organized as follows.
In Section \ref{s:MP} we briefly recall the conservative multipeakon method introduced in
\cite{holden2007globalmp}, where a finite set of peakons serve as the particles discretizing the CH equation,
and outline how the corresponding system is derived for the periodic case.
Moreover, we present efficient algorithms for computing the right-hand sides of their respective
ODEs, which are inspired by the fast summation algorithms of Camassa et al.\ for their particle methods
\cite{CamHuaLee2005,Camassa2008}.
Section \ref{s:VD} describes the new variational scheme in detail, and some emphasis is put on deriving
fundamental solutions for a discrete momentum operator, which in turn allows for collisions between
characteristics and hence wave breaking.
Finally, in Section \ref{s:experiments} we very briefly describe the methods we have chosen to compare with,
before turning to a series of numerical examples of both quantitative and qualitative nature.


\section{Conservative multipeakon scheme}\label{s:MP}

An interesting feature of \eqref{eq:CH} on the real line is that it admits so-called multipeakon solutions, that is, solutions of the form
\begin{equation} \label{eq:mp_u} u(t,x) = \sum_{i=1}^{n}p_i(t) e^{-\abs{x-q_i(t)}}
\end{equation}
defined by the ODE system
\begin{align}\label{eq:mp_qp}
\begin{aligned}
\dot{q}_i &= \sum_{j=1}^{n}p_j e^{-\abs{q_i -q_j}}, \\
\dot{p}_i &= p_i \sum_{j=1}^{n}\sgn(q_i-q_j) e^{-\abs{q_i -q_j}}
\end{aligned}
\end{align}
for $i \in \{1,\dots,n\}$, where $q_i$ and $p_i$ can respectively be seen as the position and momentum of
a particle labeled $i$.
In this sense, $q_i$ is analogous to the discrete characteristic $y_i$ in the previous section.
Several authors have studied the discrete system \eqref{eq:mp_qp}, in particular Camassa and collaborators who
named it an integrable particle method, see for
instance \cite{Camassa2003, CamHuaLee2005, Raynaud2006}. This system is Hamiltonian, and one of its hallmarks is
that for initial data satisfying $q_i \neq q_j$ for $i \neq j$ and all $p_i$ having the same sign, one can find an
explicit Lax pair, meaning the discrete system is in fact integrable.
The Lax pair also serves as a starting point for studying general conservative multipeakon solutions with the help
of spectral theory, see \cite{EckKos2014, EckKos2018}.

System \eqref{eq:mp_qp} is however not suited as a numerical method for extending solutions beyond the
collision of particles, which for instance occurs for the two peakon initial data with $q_1 < q_2$, and
$p_2 < 0 < p_1$.
Indeed, as $\abs{q_2 - q_1} \to 0$, the momenta blow up as $(p_1,p_2) \to (+\infty,-\infty)$, cf.~\cite{Wahlen2006,GruHol2016}.
Even though this happens at a rate such that the associated energy remains bounded, unbounded solution variables
are not well suited for numerical computations.
One alternative way of handling this is to include an algorithm which transfers momentum between particles which are
close enough according to some criterion, see for instance \cite{CheLiuPen2015}.
However, we prefer to use the method presented next, where a different choice of variables, which remain bounded at
collision-time, is introduced.

\subsection{Real line version}
In \cite{holden2007globalmp} the authors propose a method for computing conservative multipeakon solutions
of the CH equation \eqref{eq:CH}, based on the observation that between the peaks located at $q_i$ and $q_{i+1}$
in \eqref{eq:mp_u}, $u$ satisfies the boundary value problem $u-u_{xx} = 0$ with boundary
conditions $u(t,q_i) \eqqcolon u_i(t)$ and $u(t,q_{i+1}(t)) \eqqcolon u_{i+1}(t)$. 
Moreover, from the transport equation for the
energy density one can derive the time evolution of $H_i$ which denotes the cumulative energy up to the point $q_i$.  Using $y_i$
instead of $q_i$ to denote the $i$\textsuperscript{th} characteristic we then obtain the discrete system
\begin{align} \label{eq:cmp_line}
\begin{aligned}
\dot{y}_i &= u_i, \\
\dot{u}_i &= -Q_i, \\
\dot{H}_i &= u_i^3 - 2P_i u_i
\end{aligned}
\end{align}
for $i \in \{1,\dots,n\}$ with
\begin{align*}
P_i &= \frac12 \intR{x}{e^{-\abs{y_i-x}}\left(u^2 + \frac12 u_x^2\right)}, \\
Q_i &= -\frac12 \intR{x}{\sgn(y_i-x) e^{-\abs{y_i-x}}\left(u^2 + \frac12 u_x^2\right)}.
\end{align*}
We note that the solution $u$ is of the form $u(t,x) = A_i(t) e^x + B_i(t) e^{-x}$ between the peaks $y_i$ and
$y_{i+1}$ with coefficients
\begin{equation*}
A_i = \frac{e^{-\bar{y}_i}}{2} \left[ \frac{\bar{u}_i}{\cosh(\delta y_i)} + \frac{\delta u_i}{\sinh(\delta y_i)} \right], \qquad B_i = \frac{e^{\bar{y}_i}}{2} \left[ \frac{\bar{u}_i}{\cosh(\delta y_i)} - \frac{\delta u_i}{\sinh(\delta y_i)} \right],
\end{equation*}
and where we for any grid function $\{v_i\}_{i=0}^{n}$ have defined
\begin{equation}\label{eq:avg}
\bar{v}_i = \frac{v_{i+1}+v_i}{2}, \qquad \delta v_i = \frac{v_{i+1} - v_i}{2}.
\end{equation}
In order to compute the solution for $x < y_1$ and $x > y_n$, one also introduces the convention $(y_0,u_0) = (-\infty,0)$
and $(y_{n+1},u_{n+1}) = (\infty,0)$.
We also have the relation
\begin{equation*}
\delta H_i(t) = \frac{H_{i+1}(t) - H_i(t)}{2} = \frac12 \int_{y_i(t)}^{y_{i+1}(t)} \left(u^2(t,x) +
u_x^2(t,x)\right)dx,
\end{equation*}
which can be computed as
\begin{equation}\label{eq:delH}
\delta H_i = \bar{u}_i^2 \tanh(\delta y_i) + (\delta u_i)^2 \coth(\delta y_i).
\end{equation}
Here we emphasize that the total energy $H_{n+1}$ is then given by
\begin{equation*}
H_{n+1} = 2 \sum_{i=0}^{n}\delta H_i,
\end{equation*}
since $H_0 = 0$.
Due to the explicit form of $u$ we may compute $P$ and $Q$ as
\begin{equation}
P_i = \sum_{j=0}^{n} P_{ij}, \qquad Q_i = \sum_{j=0}^{n} Q_{ij}, \label{eq:cmp_PQ}
\end{equation}
with $Q_{ij} = -\sigma_{ij} P_{ij}$ and
\begin{equation} \label{eq:cmp_Pij}
P_{ij} = \begin{cases}
\displaystyle \frac14 u_1^2 e^{y_1 - y_i}, & j = 0, \\
\displaystyle \frac{e^{-\sigma_{ij}(y_i - \bar{y}_j)}}{2 \cosh(\delta y_j)} \left[ \delta H_j \cosh^2(\delta y_j) \right. \\
\quad\left.+ 2 \sigma_{ij} \bar{u}_j \delta u_j \sinh^2(\delta y_j) + \bar{u}_j^2 \tanh(\delta y_j) \right], & j \in \{1,\dots,n-1\}, \\
\displaystyle \frac14 u_n^2 e^{y_i - y_n}, & j = n,
\end{cases}
\end{equation}
where we have defined
\begin{equation*}
\sigma_{ij} = \begin{cases}
-1, & j \ge i, \\
1, & j < i.
\end{cases}
\end{equation*}
For details on how such multipeakons can be used to obtain a numerical scheme for \eqref{eq:CH} we refer to
\cite{NumMP_HolRay2008}.

\subsubsection{Fast summation algorithm}
We will present a periodic version of the above method to compare with our variational scheme.
Before that, we note that the above method can be computationally expensive if one naively computes
\eqref{eq:cmp_Pij} for each $i$ and $j$, amounting to a complexity of $\Ocal(n^2)$ for computing the right-hand
side of \eqref{eq:cmp_line}.
Inspired by \cite{CamHuaLee2005} and borrowing their terminology we shall propose a fast summation
algorithm for computing \eqref{eq:cmp_PQ} with complexity $\Ocal(n)$.
Indeed, this can be done by noticing that our $P_i$ and $Q_i$ share a similar structure with the right-hand sides of
\eqref{eq:mp_qp}.
To this end we make the splittings
\begin{equation*}
P_i = \sum_{j=0}^{i-1} P_{ij} + \sum_{j=i}^{n}P_{ij} \eqqcolon f^{\text{l}}_i + f^{\text{r}}_i, \qquad Q_i = -f^{\text{l}}_i + f^{\text{r}}_i,
\end{equation*}
and note that $f^{\text{l}}_i$ and $f^{\text{r}}_i$ satisfy the recursions
\begin{equation*}
f^{\text{l}}_{i+1} = e^{y_i-y_{i+1}} f^{\text{l}}_i + e^{y_i -\bar{y}_i} (a_i + b_i) = e^{-2 \delta y_i} f^{\text{l}}_i + e^{-\delta y_i} (a_i + b_i)
\end{equation*}
and
\begin{equation*}
f^{\text{r}}_i = e^{y_i-y_{i+1}} f^{\text{r}}_{i+1} + e^{y_i -\bar{y}_i} (a_i - b_i) = e^{-2 \delta y_i} f^{\text{l}}_{i+1} + e^{-\delta y_i} (a_i - b_i)
\end{equation*}
for $i \in \{1,\dots,n-1\}$ where we have defined
\begin{equation}\label{eq:fsa_ab}
a_j \coloneqq \frac{ \delta H_j \cosh^2(\delta y_j)
  + \bar{u}_j^2 \tanh(\delta y_j) }{2\cosh(\delta y_j)}, \qquad b_j \coloneqq \frac{ \bar{u}_j \delta u_j \sinh^2(\delta y_j) }{\cosh(\delta y_j)}.
\end{equation}
Moreover we have the starting points for the recursions given by
\begin{equation*}
f^{\text{l}}_1 = \frac14 u_1^2, \qquad f^{\text{r}}_n = \frac14 u_n^2.
\end{equation*}
Clearly, computing $f^{\text{l}}$ and $f^{\text{r}}$ recursively is of complexity $\Ocal(n)$, while adding and
subtracting them to produce $P$ and $Q$ is also of complexity $\Ocal(n)$, which yields the desired result.

\subsection{Periodic version}
Now for the periodic version of \eqref{eq:cmp_line} there are only a few modifications needed.  First of all we have to replace
the ``peakons at infinity'' given by $(y_0, u_0) = (-\infty,0)$ and $(y_{n+1},u_{n+1}) = (\infty,0)$ which in some sense define
the domain of definition for the solution.  The new domain will instead be located between the ``boundary peakons''
$(y_0,u_0) = (y_n-L,u_n)$ and $(y_n,u_n)$.  Thus, we are still free to choose $n$ peakons, but we impose periodicity by
introducing an extra peakon at $y_n-L$ with height $u_n$.  We also have to redefine $H_i$, which now will denote the energy
contained between $y_0$ and $y_i$.
Thus we have $H_0 = 0$, $H_n$ is the total energy of an interval of length $L$,
while each $H_i$ can be computed as
\begin{equation}\label{eq:cmp_H}
H_i = 2\sum_{j=0}^{i-1}\delta H_j, \qquad j \in \{1,\dots,n\}
\end{equation}
with $\delta H_i$ defined in \eqref{eq:delH}.
In addition, since the energy is now integrated over the interval $[y_0, y_i]$ we have to replace the evolution equation for $H_i$
with
\begin{equation*}
\dot{H}_i = u_i (u_i^2 -2P_i) - u_0(u_0^2 -2P_0) = u_i (u_i^2 -2P_i) - u_n(u_n^2 -2P_n),
\end{equation*}
where the last identity follows from the periodicity of $P_i$ by virtue of $\delta y_i$, $u_i$, and
$\delta H_i$ being $n$-periodic.

Moreover, we have to replace $e^{-\abs{y-x}}$ in $P_i$ and $Q_i$ with its periodic counterpart
\begin{equation}\label{eq:ker_per}
\sum_{m = -\infty}^{\infty}e^{-\abs{y-(x+m \period)}} = \frac{\cosh(\abs{y-x}-\frac{\period}{2})}{\sinh\left(\frac{\period}{2}\right)}, \qquad \abs{y-x} \le \period,
\end{equation}
and we now integrate over $[y_0,y_n]$ instead of $\R$.
This is analogous to the derivation of the periodic particle method in \cite{Camassa2008}, and the numerical
results of \cite{Raynaud2006}.
The courageous reader may verify that the calculations in \cite{holden2007globalmp} can be reused to a great extent.
In the end we find that the expressions for $P_i$ and $Q_i$ are essentially the same, we only need to replace each occurrence of $e^{-\sigma_{ij}(y_i - \bar{y}_j)}$
and its ``derivative'' with respect to $y_i$, $-\sigma_{ij} e^{-\sigma_{ij}(y_i - \bar{y}_j)}$, with
\begin{equation*}
\frac{ \cosh\left( \sigma_{ij} (y_i - \bar{y}_j) -\frac{L}{2}\right)}{\sinh\left(\frac{\period}{2}\right)}
\quad\text{and its ``derivative''}\quad
\sigma_{ij} \frac{ \sinh\left( \sigma_{ij} (y_i - \bar{y}_j) -\frac{L}{2}\right)}{\sinh\left(\frac{\period}{2}\right)},
\end{equation*}
respectively.
To be precise, $P_i$ and $Q_i$ are given by
\begin{equation}\label{eq:cmp_perPQ}
P_i = \sum_{j=0}^{n-1}P_{ij}, \qquad Q_i = \sum_{j=0}^{n-1}Q_{ij}
\end{equation}
for $i \in \{1,\dots,n\}$, with
\begin{align} \label{eq:cmp_perPij}
\begin{aligned}
P_{ij} &= \frac{\cosh\left( \sigma_{ij} (y_i - \bar{y}_j) -\frac{L}{2}\right)
  \left[ \delta H_j \cosh^2(\delta y_j) + \bar{u}_j^2 \tanh(\delta y_j) \right] }{2 \cosh(\delta y_j) \sinh\left(\frac{L}{2}\right)} \\
&\quad- \sigma_{ij} \frac{\sinh\left( \sigma_{ij} (y_i - \bar{y}_j) -\frac{L}{2}\right)
  \bar{u}_j \delta u_j \sinh^2(\delta y_j) }{\cosh(\delta y_j) \sinh\left(\frac{L}{2}\right)}
\end{aligned}
\end{align}
and
\begin{align} \label{eq:cmp_perQij}
\begin{aligned}
Q_{ij} &= \sigma_{ij} \frac{\sinh\left( \sigma_{ij} (y_i - \bar{y}_j) -\frac{L}{2}\right)
  \left[ \delta H_j \cosh^2(\delta y_j) + \bar{u}_j^2 \tanh(\delta y_j) \right] }{2 \cosh(\delta y_j) \sinh\left(\frac{L}{2}\right)} \\
&\quad- \frac{\cosh\left( \sigma_{ij} (y_i - \bar{y}_j) -\frac{L}{2}\right)
  \bar{u}_j \delta u_j \sinh^2(\delta y_j) }{\cosh(\delta y_j) \sinh\left(\frac{L}{2}\right)},
\end{aligned}
\end{align}
for $j \in \{1,\dots,n-1\}$.
To summarize, the periodic system reads
\begin{align} \label{eq:cmp_per}
\begin{aligned}
\dot{y}_i &= u_i, \\
\dot{u}_i &= -Q_i, \\
\dot{H}_i &= u_i (u_i^2 -2P_i) - u_n(u_n^2 -2P_n)
\end{aligned}
\end{align}
for $i \in \{1,\dots,n\}$, and $P_i$ and $Q_i$ defined by \eqref{eq:cmp_perPQ}, \eqref{eq:cmp_perPij},
and \eqref{eq:cmp_perQij}.

\begin{exmp}
  (i) Following \cite[Ex.\ 4.2]{holden2007globalmp} we set $n = 1$,
  and use the periodicity to find $\bar{y}_0 = y_1 - \period/2$, $\delta y_0 = \period/2$, $\bar{u}_0 = u_1$, $\delta u_1 = 0$,
  and $\delta H_0 = u_1^2 \tanh(\period/2)$. 
  Plugging into \eqref{eq:cmp_perPij} and \eqref{eq:cmp_perQij} we find
  \begin{equation*}
  P_{1} = \frac{u_1^2}{2} \left( 1 + \sech^2\left(\frac{\period}{2}\right) \right), \qquad Q_1 = 0.
  \end{equation*}
  Then \eqref{eq:cmp_per} yields $\dot{u}_1 = 0$ and $\dot{H}_1 = 0$, and setting $u_1(t) \equiv c$ we
  obtain $y_1(t) = y_1(0) + c t$. This shows that for $n = 1$, in complete analogy to the real line case, the
  evolution equation for $H_1$ decouples from the other equations, and we find that a periodic peakon travels with
  constant velocity $c$ equal to its height at the peak.
  
  (ii) With substantially more effort compared to (i) we could also consider $n = 2$ with antisymmetric initial
  datum for $u$ to recover the periodic peakon-antipeakon solution computed in
  \cite[pp.~ 5505--10]{Cohen2008}.
\end{exmp}

\subsubsection{Fast summation algorithm}
Drawing further inspiration from \cite{Camassa2008} we propose a fast summation algorithm for
the periodic scheme as well, and following their lead we use the infinite sum rather than the hyperbolic function
representation of the periodic kernel. Using geometric series we find
\begin{equation*}
\sum_{m = -\infty}^{\infty}e^{-\abs{y-(x+m \period)}} = \frac{e^{-\period}}{1-e^{-\period}}e^{x-y} + e^{-\abs{y-x}} + \frac{e^{-\period}}{1-e^{-\period}}e^{y-x},
\end{equation*}
valid for $\abs{x-y} \le \period$.
Then, replacing $e^{-\sigma_{ij}(y_i-\bar{y}_j)}$ in \eqref{eq:cmp_Pij} with
\begin{equation*}
\frac{e^{-\period}}{1-e^{-\period}}e^{\bar{y}_j-y_i} + e^{-\sigma_{ij}(y_i-\bar{y}_j)} + \frac{e^{-\period}}{1-e^{-\period}}e^{y_i-\bar{y}_j}
\end{equation*}
we find that the periodic $P_i$ and $Q_i$ can be written
\begin{equation*}
P_i = g^{-}_i + f^{\text{l}}_i + f^{\text{r}}_i + g^{+}_i, \qquad Q_i = f^{\text{r}}_i + g^{+}_i - g^{-}_i - f^{\text{l}}_i,
\end{equation*}
where in analogy to the full line case we have defined
\begin{equation*}
f^{\text{l}}_i  \coloneqq \sum_{j=0}^{i-1}e^{-y_i+\bar{y}_j} (a_j + b_j), \qquad f^{\text{r}}_i \coloneqq \sum_{j=i}^{n-1}e^{y_i-\bar{y}_j} (a_j - b_j),
\end{equation*}
in addition to
\begin{equation*}
g^{-}_i \coloneqq \frac{e^{-\period}}{1-e^{-\period}} \sum_{j=0}^{n-1} e^{-y_i+\bar{y}_j} (a_j + b_j), \qquad g^{+}_i \coloneqq \frac{e^{-\period}}{1-e^{-\period}} \sum_{j=0}^{n-1} e^{y_i-\bar{y}_j} (a_j - b_j),
\end{equation*}
with $a_j$ and $b_j$ defined in \eqref{eq:fsa_ab}.
Defining $g^{\text{l}}_i \coloneqq g^{-}_i + f^{\text{l}}_i$ and $g^{\text{r}}_i \coloneqq g^{+}_i + f^{\text{r}}_i$,
these functions satisfy the recursions
\begin{equation*}
g^{\text{l}}_{i+1} = e^{-2\delta y_i} g^{\text{l}}_i + e^{-\delta y_i} (a_i + b_i), \qquad g^{\text{l}}_1 = g^{-}_1 + e^{-\delta y_0} (a_0 + b_0),
\end{equation*}
and
\begin{equation*}
g^{\text{r}}_i = e^{-2\delta y_i} g^{\text{r}}_{i+1} + e^{-\delta y_i} (a_i - b_i), \qquad g^{\text{r}}_n = g^{-}_n
\end{equation*}
for $i \in \{1,\dots,n-1\}$.  Once more, the recursion allows us to compute $P_i$ and $Q_i$ with complexity $\Ocal(n)$ rather than
$\Ocal(n^2)$ for the naive computation of each distinct $P_{ij}$ and $Q_{ij}$ in \eqref{eq:cmp_perPij} and \eqref{eq:cmp_perQij}.

\section{Variational finite difference Lagrangian discretization}\label{s:VD}
Here we describe the method which is based on a finite difference discretization in Lagrangian coordinates, as derived in
\cite{vardisc}.

Denoting the number of grid cells $n \in \N$, we introduce the grid points $\xi_i = i \dxi$ for $i \in \{0,\dots,n-1\}$ and step size $\dxi > 0$ such that $n \dxi = \period$.
These will serve as ``labels'' for our discrete characteristics $y_i(t)$ which can be regarded as approximations of $y(t,\xi_i)$.  In a similar
spirit we introduce $U_i(t)$ and $r_i(t)$.
For our discrete variables, the periodicity in the continuous case \eqref{eq:periodic_cont} translates into
\begin{equation}
y_{i+n}(t) = y_{i}(t) + \period, \quad  U_{i+n}(t) = U_{i}(t), \quad r_{i+n}(t) = r_i(t).
\label{eq:periodic_disc}
\end{equation}
For a grid function $f = \{f_i\}_{i \in \Z}$ we introduce the forward difference operator $\D_+$ defined by
\begin{equation}\label{eq:diff}
\D_\pm f_i = \pm\frac{f_{i \pm 1} - f_i}{\dxi}, \quad \D_0 f_i = \frac{f_{i+1} - f_{i-1}}{2\dxi},
\end{equation}
where we also have included the backward and central differences for future reference.
We will use the standard Euclidean scalar product in $\R^n$ scaled by the grid cell size $\dxi$ to obtain
a Riemann sum approximation of the integral on $\T$.
Moreover, we introduce the space $\Rper$ of sequences $v = \{v_j\}_{j\in\Z}$ satisfying $v_{j+n} = v_{j}$,
and which is isomorphic to $\R^n$.
For $n$-periodic sequences, the adjoint (or transpose) $\D^\tp$ of the discrete difference operator $\D$
is defined by the relation
\begin{equation*}
\sumNxi{i}{(\Dt v_i)w_i } = \sumNxi{i}{v_i (\D w_i) }, \quad v, w \in \Rper.
\label{eq:adjoint}
\end{equation*}
For instance, summation by parts shows that the differences in \eqref{eq:diff} satisfy $\D_\pm^{\top} = -\D_{\mp}$
and $\D_0^{\top} = -\D_0$.

The variational derivation of the scheme for the CH equation \eqref{eq:CH} is based on an approximation of the energy given by
\begin{equation}\label{eq:disc_energy}
E \coloneqq \frac12 \sumNxi{i}{\left((\dot{y}_{i})^2 (\D_+ y_{i}) + \frac{(\D_+ \dot{y}_{i})^2}{\D_+ y_{i}}\right)},
\end{equation}
which corresponds to \eqref{eq:contscalprod} in the continuous case. Following \cite{vardisc} we obtain a semidiscrete system
which is valid also in the periodic case, namely
\begin{equation}\label{eq:semidisc_sys}
\begin{cases}
\dot{y}_i &= U_i, \\
(\D_+ y_{i}) \dot{U}_{i} - \D_- \left(\frac{\D_+ \dot{U}_{i}}{\D_+ y_{i}}\right) &= -U_{i} (\D_+ U_{i}) - \frac{1}{2} \D_-\left( U_{i}^2 + \left(\frac{\D_+ U_{i}}{\D_+ y_{i}}\right)^2  \right)
\end{cases}
\end{equation}
for initial data $y_i(0) = (y_0)_i$ and $U_i(0) = (U_0)_i$, and indices $i \in \{0,\dots,n-1\}$.
Observe that in solving \eqref{eq:semidisc_sys} we obtain approximations of the fluid velocity in Lagrangian variables since $y_i(t) \approx y(t,\xi_i)$ and
$U_{i}(t) \approx u(t,y(t,\xi_{i}))$.

For the 2CH system \eqref{eq:2CH} one has the identity
\begin{equation*}
r(t,\xi) = r(0,\xi) = \rho(0,y(0,\xi))y_\xi(0,\xi)
\end{equation*}
in the continuous setting.
Based on this we introduce the discrete identity
\begin{equation*}
\rho_i(t) \D_+ y_i(t) = \rho_i(0) \D_+ y_i(0),
\end{equation*}
which allows us to express the discrete density $\rho_i(t)$ as a function of $\D_+ y_i(t)$ and the initial data.
Accordingly, we have the approximate relation $\rho(t,y(t,\xi_i)) \approx (\rho_0)_i \D_+(y_0)_i / \D_+y_i(t)$.

Furthermore, the energy of the discrete 2CH system contains an additional term compared to \eqref{eq:disc_energy}, and reads
\begin{equation*}
E \coloneqq \frac12 \sumNxi{i}{\left((\dot{y}_{i})^2 (\D_+ y_{i}) + \frac{(\D_+ \dot{y}_{i})^2}{\D_+ y_{i}} + \frac{(\rho_0)_i \D_+(y_0)_i}{\D_+y_{i}}\right)^2 }.
\end{equation*}
As a consequence, the semidiscrete system for the 2CH system is the same as \eqref{eq:semidisc_sys}, except
that right-hand side of the second equation now becomes
\begin{equation*}
-U_{i} (\D_+ U_{i}) - \frac{1}{2} \D_-\left( U_{i}^2 + \left(\frac{\D_+ U_{i}}{\D_+ y_{i}}\right)^2 + \frac{(\rho_0)_i \D_+(y_0)_i}{\D_+y_{i}}\right)^2.
\end{equation*}

Note that \eqref{eq:semidisc_sys} does not give an explicit expression for the time derivative $\dot{U}$,
as a solution dependent operator has been applied to it.
For $\D_+ y_i \in \Rper$ and an arbitrary sequence $w = \{w_i\}_{i\in\Z} \in \Rper$, let us define
the discrete momentum operator $\A{\D_+ y} : \Rper \to \Rper$ by
\begin{equation}
\label{eq:defAop}
(\A{\D_+ y}w)_{i} \coloneqq (\D_+ y_i) w_i + \D_- \left(\frac{\D_+ w_{i}}{\D_+ y_{i}}\right).
\end{equation}
Note that when $\D_+ y_i = 1$, \eqref{eq:defAop} is a discrete version of the Sturm--Liouville operator
$\mathrm{Id}-\partial_{xx}$. The name \textit{momentum} operator comes from the fact that the discrete energy can
be written as the scalar product of $\A{\D_+ y} U$ and $U$,
\begin{equation*}
E = \frac12 \sumNxi{i}{ (\A{\D_+ y}U)_i U_i },
\end{equation*}
which corresponds to \eqref{eq:defmomcont}.
Moreover, as in \cite{vardisc} we find that \eqref{eq:semidisc_sys} preserves the total momentum
\begin{equation}\label{eq:disc_mom}
I \coloneqq \sumNxi{i}{(\A{\D_+y}U)_i} = \sumNxi{i}{U_i \D_+y_i},
\end{equation}
where the final identity comes from telescopic cancellations and periodicity.

\subsection{Presentation of the scheme for global in time solutions}
To follow \cite{vardisc} in obtaining a scheme which allows for global in time solutions, we have to invert the
discrete momentum operator \eqref{eq:defAop}, and in the aforementioned paper this is done by finding
a set of summation kernels, or fundamental solutions, $g_{i,j}$, $\gamma_{i,j}$, $k_{i,j}$, and $\kappa_{i,j}$
satisfying
\begin{equation*}
\begin{bmatrix}
(\D_+y_j) & -\D_{j-} \\
-\D_{j+} & (\D_+y_j)
\end{bmatrix}
\begin{bmatrix}
g_{i,j} & \kappa_{i,j} \\
\gamma_{i,j} & k_{i,j}
\end{bmatrix}
= \frac{1}{\dxi}
\begin{bmatrix}
\delta_{i,j} & 0 \\
0 & \delta_{i,j}
\end{bmatrix}, \quad i,j \in \Z,
\end{equation*}
where $\D_{j\pm}$ denotes differences with respect to the index $j$.
Let us for the moment assume that we have a corresponding set of kernels for the periodic case, namely
\begin{equation}\label{eq:fml_sys}
\begin{bmatrix}
(\D_+y_j) & -\D_{j-} \\
-\D_{j+} & (\D_+y_j)
\end{bmatrix}
\begin{bmatrix}
G_{i,j} & \K_{i,j} \\
\varGamma_{i,j} & K_{i,j}
\end{bmatrix}
= \frac{1}{\dxi}
\begin{bmatrix}
\delta_{i,j} & 0 \\
0 & \delta_{i,j}
\end{bmatrix}, \quad i,j \in \{0,\dots,n-1\},
\end{equation}
and which are $n$-periodic in their index $j$ for fixed $i$.
The existence of such kernels will be justified in the next subsection.

In the end we want to derive a system which is equivalent to \eqref{eq:semidisc_sys} for $\D_+y_j > 0$ and which
serves as a finite-dimensional analogue to \cite[Eq.\ (4.42)]{vardisc}.
Following the convention therein we decompose $y_j = \zeta_j + \xi_j$, which by \eqref{eq:periodic_disc} implies that $\zeta$ is $n$-periodic as well: $\zeta_{j+n} = \zeta_j$.
Then, with appropriate modifications of the approach in \cite{vardisc}, our system for $j \in \{0,\dots,{n-1}\}$
reads
\begin{subequations}
  \label{eq:VD_h}	
  \begin{align}
  \dot{\zeta}_j &= U_j, \label{eq:VD_h:z}\\
  \dot{U}_j &= - Q_j, \label{eq:VD_h:U}\\
  \dot{h}_j &= -U_j (\D_-R_j) - R_j (\D_+U_j) = -\D_+(U_i R_{i-1}), \label{eq:VD_h:h}
  \end{align}
\end{subequations}
where we have defined
\begin{align}\label{eq:RQ_per}
\begin{aligned}
R_j &\coloneqq \sumNxi{i}{ \varGamma_{i,j} U_i (\D_+U_i) } + \sumNxi{i}{K_{i,j} h_i }, \\
Q_j &\coloneqq \sumNxi{i}{ G_{i,j} U_i (\D_+U_i) } + \sumNxi{i}{\K_{i,j} h_i },
\end{aligned}
\end{align}
and $h_j$ is defined to satisfy
\begin{equation}\label{eq:h_id}
2 h_j (\D_+y_j) = U_j^2 (\D_+y_j)^2 + (\D_+U_j)^2 + r_j^2.
\end{equation} 
We note that we could have included
\begin{equation*}
\dot{r}_j = 0
\end{equation*}
in \eqref{eq:VD_h}, but since $r$ does not appear in any of the other equations, we choose to omit it.
Note that when considering the CH equation, $\rho$, and thus also $r$, vanishes identically.
In the current setting, this only affects the presence of $r$ in the identity \eqref{eq:h_id}.

Observe that $R_j$ and $Q_j$ in \eqref{eq:RQ_per} are $n$-periodic by virtue of the kernels being $n$-periodic
in $j$, and so it follows that \eqref{eq:VD_h} is of the form $\dot{X}_j(t) = F_j(X(t))$,
where $X_{j+n}(0) = X_j(0)$ and $F_{j+n}(X) = F_j(X)$.
Then the integral form of \eqref{eq:VD_h} shows that $X_{j+n}(t) = X_j(t)$, and so any solution of this equation must
be $n$-periodic.

We also note that we can equivalently formulate \eqref{eq:VD_h} more in the spirit of
\cite[Eq.\ (4.42)]{vardisc} by defining
\begin{equation}\label{eq:H}
H_j(t) = \dxi\sum_{i=0}^{j-1}h_i(t), \enskip j \in \{1,\dots,n\}, \qquad H_0(t) = 0,
\end{equation}
and replace \eqref{eq:VD_h:z} and \eqref{eq:VD_h:h} to obtain
\begin{subequations}
  \label{eq:VD_H}	
  \begin{align}
  \dot{y}_j &= U_j, \label{eq:VD_H:z} \\
  \dot{U}_j &= - Q_j, \label{eq:VD_H:U}\\
  \dot{H}_j &= U_0 R_{n-1} - U_j R_{j-1}, \label{eq:VD_H:H}
  \end{align}
\end{subequations}
where we have combined \eqref{eq:VD_h:h} and \eqref{eq:H} with the periodicity of $U$ and $R$ to get
\eqref{eq:VD_H:H}.
In this case we note that
$\D_+H_j = h_j$ for $j \in \{0,\dots,{n-1}\}$, $\dot{H}_{n}(t) = \dot{H}_0(t) \equiv 0$, and $H_n(t) = H_n(0)$ is
the total energy of the system.
The energy $H_n$ is a reformulation of \eqref{eq:disc_energy} in Lagrangian variables.
Equation \eqref{eq:VD_H} is in fact our preferred version of the scheme, as it more closely resembles
\eqref{eq:cmp_per} and preserves the discrete energy $H_n$ identically.

An important observation is that the sequences defined in \eqref{eq:RQ_per} solve
\begin{equation*}
\begin{bmatrix}
(\D_+y_j) & -\D_- \\
-\D_+ & (\D_+y_j)
\end{bmatrix}
\begin{bmatrix}
Q_j \\ R_j
\end{bmatrix}
=
\begin{bmatrix}
U_j (\D_+U_j) \\ h_j 
\end{bmatrix}, \qquad j \in \{0,\dots,{n-1}\},
\end{equation*}
which is equivalent to
\begin{equation}\label{eq:RQ_mat}
A[\D_+y] \begin{bmatrix}
Q_0 \\ R_0 \\ \vdots \\ Q_{n-1} \\ R_{n-1}		
\end{bmatrix} = \begin{bmatrix}
U_0 (\D_+U_0) \\ h_0 \\ \vdots \\ U_{n-1} (\D_+U_{n-1}) \\ h_{n-1}
\end{bmatrix},
\end{equation}
for the tridiagonal $2n \times 2n$-matrix
\begin{equation}\label{eq:Atot}
A[\D_+y] \coloneqq \frac{1}{\dxi}
\begin{bmatrix}
\dxi \D_+y_0 & -1 & 0 & \cdots & 0 & 0 & 1 \\
1 & \dxi \D_+y_0 & -1 & \cdots & 0 & 0 & 0 \\
\vdots & \vdots & \vdots & \ddots & \vdots & \vdots & \vdots  \\
0 & 0 & 0 & \cdots & 1 & \dxi \D_+y_{n-1} & -1 \\
-1 & 0 & 0 & \cdots & 0 & 1 & \dxi \D_+y_{n-1}
\end{bmatrix}
\end{equation}
with corners.
As shown in the next section, the matrix \eqref{eq:Atot} is invertible whenever $\D_+y_j \ge 0$ for
$j \in \{0,\dots,n-1\}$.
Thus, \eqref{eq:RQ_mat} provides a far more practical approach for computing the right hand side of
\eqref{eq:VD_h} than the identities \eqref{eq:RQ_per}, especially for numerical methods,
as there is no need to compute the kernels in \eqref{eq:fml_sys}.
Indeed, if one uses an explicit method to integrate in time, given $y$, $U$ and $h$ we can solve \eqref{eq:RQ_mat}
to obtain the corresponding $R$ and $Q$.

\subsection{Inversion of the discrete momentum operator}
The alert reader may wonder why we work with the $2n \times 2n$ matrix \eqref{eq:Atot} when inverting the operator
\eqref{eq:defAop} defined in only $n$ points.
This comes from the approach in \cite{vardisc} which enables the discretization to handle ``discrete''
wave breaking, i.e., $\D_+y_i = 0$.
By introducing a change of variables we rewrite the second order difference operator \eqref{eq:defAop} as the first order
matrix operator appearing in \eqref{eq:fml_sys}.
Thus we avoid $\D_+y$ in the denominator at the cost of increasing the size of the system.

When introducing the change of variables, we lose some desirable properties which would have made it easy to
establish the invertibility of the matrix corresponding to \eqref{eq:Atot} in the cases where $\D_+y_i \ge c$ for some
positive constant $c$.
This would for instance be the case for discretizations of \eqref{eq:2CH_IVP} where $\rho_{0}^2(x) \ge d$ for a
constant $d > 0$, since it is then known that wave breaking cannot occur, see
\cite[Thm.\ 4.5]{Grunert2013}.
In particular, we lose symmetry of the matrix which would have enabled us to use the standard argument
involving diagonal dominance,
as used for instance in \cite{LiuPen2016} for a discrete Helmholtz operator.
Our matrix \eqref{eq:Atot} is clearly not diagonally dominant, but it is still invertible,
as shown in the following proposition.

\begin{prop}
  \label{prop:Ainv}
  Assume $y_n - y_0 = \period$ and $\D_+y_i \ge 0$ for $i \in \{0,\dots,n-1\}$.
  Then $A[\D_+y]$ defined in \eqref{eq:Atot} is invertible for any $n \in \N$.
\end{prop}

\begin{proof}
  We will prove that the determinant of $A[\D_+y]$ is bounded from below by a strictly positive constant.
  Thus it is never singular.
  
  First we recall the matrix
  \begin{equation}\label{eq:Aj}
  A_j = \begin{bmatrix}
  1 + (\dxi \D_+y_j)^2 & \dxi\D_+y_j \\ \dxi\D_+y_j & 1
  \end{bmatrix}
  \end{equation}
  which played an essential part when inverting the discrete momentum operator on the full line in \cite{vardisc}.
  Below we will see that it plays a role in the periodic case as well, and we emphasize the property
  $\det{A_j} = 1$.
  
  Turning back to $A[\D_+y]$, we consider the rescaled matrix $\dxi A[\D_+y]$ in order to have the absolute
  values of the off-diagonal elements equal to one.
  We observe that this matrix is tridiagonal, with nonzero corners owing to the periodic boundary.
  Then, the clever argument in \cite[Lem.~ 1]{Molinari2008} gives an identity for the determinant of a general
  matrix of this form, which in our case reads
  \begin{equation}\label{eq:det_tr_id}
  \det(\dxi A[\D_+y]) = -\det\left( \varPi_0 - I \right) = \tr(\varPi_0)-2,
  \end{equation}
  with
  \begin{equation}\label{eq:Aprod0}
  \varPi_0 \coloneqq A_{n-1} A_{n-2} \cdots A_1 A_0,
  \end{equation}
  and where the last identity in \eqref{eq:det_tr_id} comes from $\det(\varPi_0) = 1$.
  Next, we note that each factor $A_j$ in $\varPi_0$ can be written as
  \begin{equation*}
  A_j = I + E_j, \qquad E_j = \dxi \D_+y_j \begin{bmatrix}
  \dxi \D_+y_j & 1 \\ 1 & 0
  \end{bmatrix}, 
  \end{equation*}
  for which we have
  \begin{equation*}
  E_j E_k = (\dxi \D_+y_j)(\dxi \D_+y_k) \begin{bmatrix}
  1 + (\dxi \D_+y_j)(\dxi \D_+y_k) & \dxi \D_+y_j \\ \dxi \D_+y_k & 1
  \end{bmatrix}
  \end{equation*}
  Then we may expand $\varPi_0$ as
  \begin{equation*}
  \varPi_0 = (I + E_{n-1})\cdots (I + E_0) = I + \sum_{j=0}^{n-1}E_j + \sum_{j=0}^{n-2}\sum_{k=j+1}^{n-1}E_j E_k + \cdots,
  \end{equation*}
  which means that its trace can be expanded as
  \begin{align*}
  \tr(\varPi_0) &= 2 + \sum_{j=0}^{n-1}(\dxi \D_+y_j)^2 \\
  &\quad+ \sum_{j=0}^{n-2}\sum_{k=j+1}^{n-1} (\dxi \D_+y_j)(\dxi \D_+y_k)[2 + (\dxi \D_+y_j)(\dxi \D_+y_k)] + \cdots.
  \end{align*}
  Since all factors are nonnegative, we throw away most terms to obtain
  \begin{align*}
  \tr(\varPi_0) - 2 &\ge \sum_{j=0}^{n-1}(\dxi \D_+y_j)^2 + 2\sum_{j=0}^{n-2}\sum_{k=j+1}^{n-1} (\dxi \D_+y_j)(\dxi \D_+y_k) \\
  &= \left( \sumNxi{j}{\D_+y_j} \right)^2 = \period^2,
  \end{align*}
  where the final identity follows from $y_n - y_0 = \period$.
  Hence, combining the above with $n \dxi = \period$ we obtain the lower bound
  \begin{equation*}
  \det(A[\D_+y]) = \frac{\tr(\varPi_0)-2}{\dxi^{2n}} \ge \frac{\period^2}{\dxi^{2n}} = \frac{n^{2n}}{\period^{2n-2} },
  \end{equation*}
  which clearly shows $A[\D_+y]$ to be nonsingular for any $n \in \N$.
\end{proof}

To prove the existence of global solutions to the governing equations \eqref{eq:VD_h} or \eqref{eq:VD_H} by a fixed point
argument, we have to establish Lipschitz continuity of the right-hand side.
However, Lipschitz bounds for the inverse operator of $\A{\D_+y}$ are difficult to obtain
directly. In particular, we see that the estimates in Proposition \ref{prop:Ainv} rely on the positivity of the sequence $\D_+y$,
which is difficult to impose in a fixed-point argument. Therefore, we will have to follow the approach developed in
\cite{vardisc} where we introduce the fundamental solutions for the operator $\A{\D_+y}$ and propagate those in time
together with the solution.
Proposition \ref{prop:Ainv} gives us the existence of the fundamental solutions in \eqref{eq:fml_sys}.
Indeed, comparing the equations \eqref{eq:fml_sys} with the matrix \eqref{eq:Atot} one can verify that each of
$G_{i,j}$, $\varGamma_{i,j}$, $K_{i,j}$, and $\K_{i,j}$ for $i,j \in \{0,\dots,n-1\}$, $4n^2$ in total,
appears as a distinct entry in the inverse of $\dxi A[\D_+y]$.

We do not detail here the argument developed in \cite{vardisc} which shows the existence of global solutions to the semi-discrete
system \eqref{eq:VD_h}. In fact, the periodic case is of finite dimension and therefore easier to treat than the case of the
real line.
Instead we will devote most of the remaining paper to numerical results.
Before that, we present nevertheless some interesting properties of the
fundamental solutions that can be derived in the periodic case, and which show the connection
to the fundamental solutions on the real line.
Readers more interested in numerical results may skip to Section \ref{s:experiments}.

\subsubsection{Properties of the fundamental solutions}
Here we present an alternative method for deriving the periodic fundamental solutions, more in line with the procedure in
\cite{vardisc}.
The construction is done in two steps, the first of which is to find the fundamental solutions on the infinite
grid $\dxi \Z$ as was done in \cite{vardisc}.  Then it turns out that we can periodize these solutions to find
fundamental solutions for the grid given by $i \dxi$ for $0 \le i < n$.
In this endeavor we only assume the periodicity $y_{i+n} = y_i + \period$ and $\D_+y_i \ge 0$, as was done in
Proposition \ref{prop:Ainv}.

Due to the periodicity, we can think of the sequences $\{\D_+y_j, U_j, h_j, r_j\}_{j \in \Z}$
being a repetition of $\{\D_+y_j, U_j, h_j, r_j\}_{j = 0}^{n-1}$,
such that $\D_+y_{j+k n} = \D_+y_j$ for $j,k \in \Z$, and similarly for the other entries.
Furthermore, it enforces the relation
\begin{equation*}
\sumNxi{i}{ \D_+y_i } = y_n - y_0 = \period,
\end{equation*}
which together with $n \dxi = \period$ yields
\begin{equation}\label{eq:lbd_Dy}
\frac{\period}{n}\sum_{i=0}^{n-1}\D_+y_i = \period \iff \frac{1}{n}\sum_{i=0}^{n-1} \D_+y_i = 1 \implies \max_{0 \le i \le n-1} \D_+y_i \ge 1.
\end{equation}
This also leads to the upper bounds
\begin{equation}\label{eq:ubd_Dy}
\max_{0 \le i \le n-1} \D_+y_i \le n \iff \max_{0 \le i \le n-1} \dxi \D_+y_i \le \period,
\end{equation}
but note that this bound can only be attained if $\D_+y_i = 0$ for every other index than the one achieving the maximum.

To find a fundamental solution $g_{i,j}$ for the operator \eqref{eq:defAop} defined on the real line,
that is $g_{i,j}$ which satisfies
\begin{equation*}
(\D_+y_j) g_{i,j} - \D_{j-}\left(\frac{\D_{j+}g_{i,j}}{\D_+y_j}\right) = \frac{\delta_{i,j}}{\dxi},
\end{equation*}
we consider the homogeneous operator equation
\begin{equation}\label{eq:hom_g}
(\D_+y_i) g_i - \D_-\left(\frac{\D_+g_i}{\D_+y_i}\right) = 0, \quad i \in \Z.
\end{equation}
By introducing the quantity
\begin{equation*}
\gamma_i = \frac{\D_+g_i}{\D_+y_i}
\end{equation*}
we can restate \eqref{eq:hom_g} as
\begin{equation*}
\begin{bmatrix}
g_{i+1} \\ \gamma_i
\end{bmatrix} = \begin{bmatrix}
1 + (\dxi \D_+y_i)^2 & \dxi \D_+y_i \\ \dxi \D_+y_i & 1
\end{bmatrix} \begin{bmatrix}
g_i \\ \gamma_{i-1}
\end{bmatrix} = A_i \begin{bmatrix}
g_i \\ \gamma_{i-1}
\end{bmatrix}, \quad i \in \Z.
\end{equation*}
Thus, if for any index $i$ we prescribe values for $g_i$ and $\gamma_{i-1}$, the corresponding solution of \eqref{eq:hom_g} in any other index can be found by repeated multiplication with the matrix $A_i$ from \eqref{eq:Aj} and its inverse.
The eigenvalues and eigenvectors of $A_i$ are found in \cite[Lem.~ 3.3]{vardisc}, and we briefly state its eigenvalues
\begin{equation*}
\lambda^\pm_i = 1 + \frac{(\dxi \D_+y_i)^2}{2} \pm \frac{\dxi \D_+y_i}{2}\sqrt{4 + (\dxi \D_+y_i)^2},
\end{equation*}
and underline that $\lambda_i^+ \lambda_i^- = 1$.
By \eqref{eq:ubd_Dy} we obtain the bound
\begin{equation*}
1 + \dxi \D_+y_i \le \lambda_i^+ \le 1 + \dxi (\D_+y_i) \left(\sqrt{1+\left(\frac{\period}{2}\right)^2} + \frac{\period}{2} \right).
\end{equation*}
Thus, using the inequality
\begin{equation*}
\frac{x}{1+\frac12 x} < \ln(1+x) < x, \quad x > 0
\end{equation*}
we find
\begin{equation*}
\frac{\dxi \D_+y_i}{1+\frac{\period}{2}} \le \ln{\lambda_i^+} \le \left(\sqrt{1+\left(\frac{\period}{2}\right)^2} + \frac{\period}{2} \right) \dxi \D_+y_i,
\end{equation*}
which means that we may write
\begin{equation}\label{eq:eigAnexp}
\lambda_i^\pm = e^{\pm c_i \dxi \D_+y_i}
\end{equation}
for some
\begin{equation*}
\frac{1}{1+\frac{\period}{2}} \le c_i \le \sqrt{1+\left(\frac{\period}{2}\right)^2} + \frac{\period}{2}.
\end{equation*}

To construct fundamental solutions for the operator we need to find the correct homogeneous solutions for our purpose,
namely those with exponential decay.
In \cite{vardisc} one used the asymptotic relation $\lim_{n\to\pm\infty} \D_+y_i = 1$ to deduce
the existence of limit matrices, and the correct values to prescribe for $g$ and $\gamma$ were given by the
eigenvectors of these matrices.
The periodicity of $\D_+y$ prevents us from applying the same procedure to the problem at hand, but fortunately it
turns out that a different argument can be applied in our case. In fact, we can draw much inspiration from
\cite[Chap.~ 7]{JacobiOperator} which treats Jacobi operators with periodic coefficients, since the operator
\eqref{eq:defAop} can be regarded as a particular case of such operators.
However, we make some modifications in this argument for our setting, such as introducing the variable
$\gamma_i$ from earlier, and using $[g_i, \gamma_{i-1}]^\top$ as the vector to be propagated instead of $[g_i, g_{i-1}]^\top$.
The reason for this is to ensure the nice properties of the transition matrix $A_i$, such as symmetry and
determinant equal to one, and to avoid problems with dividing by zero when $\D_+y_i = 0$.
See also the discussion leading up to \cite[Lem.~ 3.3]{vardisc}.

\begin{prop}\label{prop:floquet}
  The solutions of the homogeneous operator equation are of the form
  \begin{equation*}
  g^{\pm}_i = p_i e^{\pm i \dxi q}, \quad p_{i+n} = p_i, \quad q > 0,
  \end{equation*}
  which corresponds to the Floquet solutions in \cite[Thm.~ 7.3]{JacobiOperator}.
\end{prop}

\begin{proof}
  Let us follow \cite[Eq.\ (3.23)]{vardisc} in defining the transition matrix
  \begin{equation*}
  \varPhi_{j,i} \coloneqq \begin{cases} \ds
  A_{j-1} A_{j-2} \dots A_{i+1} A_{i}, & j > i, \\
  \ds I, & j = i, \\
  \ds (A_{j})^{-1} (A_{j-1})^{-1} \dots (A_{i-2})^{-1} (A_{i-1})^{-1}, & j < i
  \end{cases}
  \end{equation*}
  which satisfies
  \begin{equation*}
  \begin{bmatrix}
  g_j \\ \gamma_{j-1}
  \end{bmatrix} = \varPhi_{j,i} \begin{bmatrix}
  g_i \\ \gamma_{i-1}
  \end{bmatrix}, \quad (\varPhi_{j,i})^{-1} = \varPhi_{i,j}, \quad \varPhi_{j,i} = \varPhi_{j,k} \varPhi_{k,i}, \quad i,j,k \in \Z.
  \end{equation*}
  By the $n$-periodicity of $A_i$ we find $\varPhi_{j+n,i+n} = \varPhi_{j,i}$, and since $\det{A_i} = 1$ it follows that $\det{\varPhi_{j,i}} = 1$.
  
  The next step is to show that for any fixed $i_0$, we can write $\varPhi_{i,i_0}$ as the product of a matrix with $n$-periodic coefficients and a matrix exponential, as in \cite[p.\ 116]{JacobiOperator}.
  We define the generalization of \eqref{eq:Aprod0}
  \begin{equation}\label{eq:Aprod}
  \varPi_{i} \coloneqq A_{i+n-1} A_{i+n-2} \dots A_{i+1} A_i = \varPhi_{i+n,i},
  \end{equation}
  which is clearly $n$-periodic and contains every possible instance of $A_i$ as a factor.
  Note that for any $i, j \in \Z$, it follows from the properties of $\varPhi_{j,i}$ that we can write
  \begin{equation*}
  \varPi_j = \varPhi_{j+n,i+n} \varPi_i \varPhi_{i,j} = \varPhi_{j,i} \varPi_i (\varPhi_{j,i})^{-1},
  \end{equation*}
  so $\varPi_j$ and $\varPi_i$ are similar matrices.
  Thus, among other properties, they have the same eigenvalues.
  
  The matrix $\varPi_i$ in \eqref{eq:Aprod} will prove to be the key to the construction of fundamental solutions.
  Now we define $A$ to be the matrix given by \eqref{eq:Aj} for $\D_+y_j = 1$, that is
  \begin{equation*}
  A =
  \begin{bmatrix}
  1 + \dxi^2 & \dxi\\
  \dxi& 1
  \end{bmatrix}.
  \end{equation*}
  Since $A_i$ is the sum of the identity and a nonnegative matrix, we deduce
  the entrywise inequality $\varPi_{i_0} \ge A > 0$.
  Indeed, this follows from using \eqref{eq:lbd_Dy} to bound $A_i$ corresponding to the maximal $\D_+y_i$ from
  below by $A$, while bounding the remaining $A_i$ from below by $I$. 
  Then it follows from Wielandt's theorem \cite[p.\ 675]{MatrixApplied} that $\varpi^+ \ge \lambda^+ > 1$,
  where $\varpi^+$ and $\lambda^+$ are the largest eigenvalues of $\varPi_{i_0}$ and $A$ respectively.  Moreover, as
  $\det(\varPi_{i_0}) = 1$ this implies $1 > \lambda^- \ge \varpi^-$, where the superscript now indicates the smallest eigenvalue.
  In particular, this implies the existence of a matrix $Q$ with eigenvalues $\pm q$ for some $q > 0$ such that
  $\varPi_{i_0} = \exp(n \dxi Q)$.  This can be seen as an alternative factorization of $\varPi_{i_0}$,
  a sort of geometric mean raised to the power $n$.
  Note that due to \eqref{eq:eigAnexp} we find it natural to include $\dxi$ in the exponent to ensure that $q$ can
  be bounded from above and below by constants depending only on the period $\period$ instead of the
  grid parameter $\dxi$.
  
  The matrix exponential is always invertible and so we may write
  \begin{equation*}
  \varPhi_{i,i_0} = P_{i,i_0} \exp((i-i_0) \dxi Q)
  \end{equation*}
  for some matrix $P_{i,i_0}$ which necessarily satisfies $P_{i_0,i_0} = I$.
  We then verify that $P_{i,i_0}$ is $n$-periodic in $i$,
  \begin{align*}
  P_{i+n,i_0} &= \varPhi_{i+n,i_0} \exp(-(i+n-i_0) \dxi Q) \\
  &= \varPhi_{i+n,i_0} (\Pi_{i_0})^{-1} \exp(-(i-i_0) \dxi Q) \\
  &= \varPhi_{i+n,i_0} \varPhi_{i_0,i_0+n} \exp(-(i-i_0) \dxi Q) \\
  &= \varPhi_{i+n,i_0+n} \exp(-(i-i_0) \dxi Q) \\
  &= \varPhi_{i,i_0} \exp(-(i-i_0) \dxi Q) \\
  &= P_{i,i_0} \exp((i-i_0)Q) \exp(-(i-i_0) \dxi Q) \\
  &= P_{i,i_0}.
  \end{align*}
  From before we know that $\varPi_{i_0}$ has distinct eigenvalues $\varpi^{\pm} = e^{\pm n\dxi q} = e^{\pm \period q}$.  Let us
  then denote the corresponding eigenvectors of $\varPi_{i_0}$ by $v^\pm_{i_0}$, and define
  \begin{equation*}
  \begin{bmatrix}
  g^{\pm}_{i} \\ \gamma^{\pm}_{i-1}
  \end{bmatrix} = \varPhi_{i,i_0} v^\pm_{i_0},
  \end{equation*}
  where $g^\pm_i$ corresponds to the Floquet solution of \eqref{eq:hom_g}.
  Observe that
  \begin{equation*}
  \begin{bmatrix}
  g^{\pm}_{i+n} \\ \gamma^{\pm}_{i-1+n}
  \end{bmatrix} = \varPhi_{i+n,i_0} v^\pm_{i_0} = \varPhi_{i+n,i_0+n} \varPi_{i_0} v^\pm_{i_0} = \varpi^\pm \varPhi_{i,i_0} v^\pm_{i_0} = \varpi^\pm \begin{bmatrix}
  g^{\pm}_{i} \\ \gamma^{\pm}_{i-1}
  \end{bmatrix}.
  \end{equation*}
  Since we can find such eigenvectors for any $0 \le i_0 \le n-1$ it is clear that the homogeneous solutions can be
  written as
  \begin{equation}\label{eq:hom_exp}
  g^\pm_{i} = \hat{p}_i e^{\pm i \dxi q}, \quad \gamma^\pm_{i} = \tilde{p}_i e^{\pm i \dxi q}
  \end{equation}
  where $\hat{p}_{i+n} = \hat{p}_i$ and $\tilde{p}_{i+n} = \tilde{p}_i$.
  Note that we may use the same Floquet solutions for $i \to -\infty$ as well.
  Indeed,
  \begin{equation*}
  \begin{bmatrix}
  g^{\pm}_{i-n} \\ \gamma^{\pm}_{i-1-n}
  \end{bmatrix} = \varPhi_{i-n,i_0} v^\pm_{i_0} = \varPhi_{i-n,i_0-n} (\varPi_{i_0})^{-1} v^\pm_{i_0} = \varpi^\mp \varPhi_{i,i_0} v^\pm_{i_0} = \varpi^\mp \begin{bmatrix}
  g^{\pm}_{i} \\ \gamma^{\pm}_{i-1}
  \end{bmatrix}.
  \end{equation*}
  This concludes the proof.
\end{proof}

Hence, to obtain a fundamental solution centered at some $0 \le i \le n-1$,
we need only combine the Floquet solutions \eqref{eq:hom_exp} with decay in each direction in such a way as to satisfy the correct jump condition at $i$.
As shown in \cite{vardisc}, the solution is of the form
\begin{equation}\label{eq:g}
g_{i,j} = \frac{1}{W}\begin{cases}
g^-_j g^+_i, & j \ge i \\
g^+_j g^-_i, & j < i
\end{cases}, \quad \gamma_{i,j} = \frac{1}{W} \begin{cases}
\gamma^-_j g^+_i, & j \ge i \\
\gamma^+_j g^-_i, & j < i
\end{cases}
\end{equation}
where $W = W_j = g^-_j \gamma^+_j - g^+_j \gamma^-_j$ is the spatially constant Wronskian.
Furthermore, the use of $g_{i,j}$ and $\gamma_{i,j}$ in \cite{vardisc} to construct fundamental solutions
$k_{i,j}$ and $\kappa_{i,j}$ satisfying
\begin{equation*}
\kappa_{i,j} = \frac{\D_{j-}k_{i,j}}{\D_+y_j}, \quad (\D_+y_j) k_{i,j} - \D_+\left(\frac{\D_-k_{i,j}}{\D_+y_j}\right) = \frac{\delta_{i,j}}{\dxi}
\end{equation*}
carries over directly.

Using the fundamental solutions found before we introduce the periodized kernels
\begin{align}\label{eq:GK}
\begin{aligned}
G_{i,j} &\coloneqq \sumZ{m}{g_{i,j+m n}}, & \varGamma_{i,j} &\coloneqq \sumZ{m}{\gamma_{i,j+m n}}, \\
K_{i,j} &\coloneqq \sumZ{m}{k_{i,j+m n}}, & \K_{i,j} &\coloneqq \sumZ{m}{\kappa_{i,j+m n}}
\end{aligned}
\end{align}
which are defined for $i,j \in \{0,\dots,n-1\}$ and $n$-periodic in $j$, e.g., $G_{i,j+n} = G_{i,j}$.
By our previous analysis, the summands are exponentially decreasing in $|m|$, and so the series in \eqref{eq:GK} are well-defined.
For instance, using \eqref{eq:hom_exp} and \eqref{eq:g} we may compute $G_{i,j}$ as the sum of two geometric series,
\begin{equation*}
G_{i,j} = \frac{\hat{p}_i \hat{p}_j}{W} \frac{e^{-q\dxi|j-i|} + \varpi^- e^{q\dxi|j-i|}}{1-\varpi^-} = \frac{\hat{p}_i \hat{p}_j}{W} \frac{e^{-q\Delta\xi|j-i|} + e^{q\Delta\xi(|j-i|-n)}}{1-e^{-q\dxi n}}.
\end{equation*}
Compare this expression for $i=0$ to the definition of $g^p_j$ in \cite[p.\ 1658]{Holden2006},
and note that they coincide for $\D_+y_j \equiv 1$ with our $q\dxi$ and $\hat{p}_0 \hat{p}_j/W$ corresponding to their $\kappa$ and $c$ respectively.
Using the fact that $\D_+y_{j+m n} = \D_+y_j$ for $m \in \Z$ we observe that these functions satisfy
the fundamental solution identity \eqref{eq:fml_sys}.
Moreover, the identity \eqref{eq:fml_sys} imposes two symmetry conditions and an anti-symmetry condition on \eqref{eq:GK},
namely
\begin{equation}\label{eq:sym_GK}
G_{i,j} = G_{j,i}, \quad K_{i,j} = K_{j,i}, \quad \K_{i,j} = -\varGamma_{j,i}.
\end{equation}
These can be derived in complete analogy to the proof of \cite[Lem.~ 4.1]{vardisc}, replacing the decay at infinity by periodicity to carry out the summation by parts without any boundary terms.
Alternatively, one can use the structure of the matrix \eqref{eq:Atot} and its inverse to
show that \eqref{eq:sym_GK} holds.


\section{Numerical experiments}\label{s:experiments}

In this section we will test our numerical method presented in the previous section for both the CH equation
\eqref{eq:CH} and the 2CH system \eqref{eq:2CH}, and compare it to existing methods.
As these are only discretized in space, we just want to consider the error introduced by the spatial discretization.
To this end we have chosen to use explicit solvers from the \textsc{Matlab} ODE suite to integrate in time,
and in most cases this amounts to using \texttt{ode45}, the so-called go-to routine.
\textsc{Matlab}'s solvers estimate absolute and relative errors, and the user may set corresponding tolerances
for these errors, \texttt{AbsTol} and \texttt{RelTol}, to control the accuracy of the solution.
Our aim is to make the errors introduced by the temporal integration negligible compared to the errors stemming
from the spatial discretization, and thus be able to compare the spatial discretization error of our schemes to
those of existing methods.
All experiments were performed using \textsc{Matlab} R2018b on a 2015 Macbook Pro with a 3.1 GHz Dual-Core Intel Core i7 processor.

For the examples where we have an exact reference solution, we would like to compare convergence rates for some
fixed time $t$.
To compute the error we have then approximated the $\Hone$-norm by a Riemann sum
\begin{equation}\label{eq:H1_approx}
\normHone{u_n - u}^2 \approx \dx \sum_{i=0}^{2^{k_0}-1} \left[ (u_n(x_i) - u(x_i))^2 + ((u_n)_x(x_i) - u_x(x_i))^2 \right],
\end{equation}
where $u(x)$ is the reference solution, and $u_n(x)$ is the numerical solution for $n = 2^k$.
The norm in \eqref{eq:H1_approx} is interpolated on a reference grid $x_i = i \dx$ for $\dx = 2^{-k_0}\period$.
Here we ensure that $k_0$ is large enough compared to $k$ for the approximation to be sufficiently close
to the $\Hone$-norm, and in general we have found that taking $k_0 \ge 2 + \max_k k $ works well in our examples.
We omit the second term of the summand in \eqref{eq:H1_approx} to obtain the corresponding approximation of the
$\Ltwo$-norm.

Note that for the schemes \eqref{eq:cmp_per} and \eqref{eq:VD_H} set in Lagrangian coordinates, traveling waves
in an initial interval will move away from this interval along their characteristics.
To compare their solutions to schemes set in fixed Eulerian coordinates we consider only norm on the initial interval,
and for the Lagrangian solution we use the periodicity to identify $y(t,\xi)$ with a position on the initial
interval.
For instance, if the initial interval is $[0,\period]$, we identify the solution in the positions $y$ and
$y$ modulo $\period$.

\subsection{Review of the discretization methods}\label{ss:review}
Here we briefly review the discretization methods used in the coming examples, and in particular we specify how
the they have been interpolated on the reference grid.
The schemes we use to compare with \eqref{eq:VD_H} can of course be just a small sample of existing methods,
and we have chosen to compare with a subset of schemes which share some features with our variational scheme \eqref{eq:VD_H}.
As alluded to in the introduction, the conservative multipeakon scheme \eqref{eq:cmp_line} from \cite{holden2007globalmp} shares much structure
with \eqref{eq:VD_H}, and so we found it natural to define its periodic version \eqref{eq:cmp_per} for
comparison.
Furthermore, since the discrete energy \eqref{eq:disc_energy} is defined using finite differences, we decided to
implement some finite difference schemes, and here we included both conservative and
dissipative methods to illustrate their features.
Finally, we included a pseudospectral scheme, also known as Fourier collocation method, which has less in common with the other
schemes.
This is known to perform extremely well for smooth solutions, but we will see that it is less suited for solutions of
peakon type.

We underline that even though these numerical schemes may have been presented with specific methods for
integrating in time in their respective papers, for these examples we want to compare the error introduced by the
spatial discretization only, and to treat all methods equally we choose a common explicit method as described before.

\subsubsection{Conservative multipeakon scheme}
As mentioned in the introduction, when defining the interpolant $u_n(t,x)$ for the multipeakon scheme
\eqref{eq:cmp_per} between the peaks located at $y_i$, it is a piecewise combination of exponential functions.
This in turn makes its derivative $(u_n)_x$ piecewise smooth, but discontinuous at the peaks.
For the approximation of initial data, unless otherwise specified, we have chosen $y_i(0) = \xi_i$,
$U_i(0) = u_0(\xi_i)$, and computed $H_i(0)$ according to \eqref{eq:cmp_H} for $\xi_i = i \dxi$ and
$i \in \{0,\dots,n-1\}$.

\subsubsection{Variational finite difference Lagrangian scheme}
We mentioned in Section \ref{s:VD} that it is computationally advantageous to solve the matrix system
\eqref{eq:RQ_mat} when computing $R$ and $Q$ in the right-hand side of \eqref{eq:VD_H}.
Indeed, solving this nearly tridiagonal system should have a complexity close to $\mathcal{O}(n)$
when solved efficiently.
In practice, we find that the standard \textsc{Matlab} backslash operator, or \texttt{mldivide}
routine, is sufficient for our purposes, as it seems to scale approximately linearly with $n$ in our experiments. 

For the interpolant $u_n(t,x)$ we solve \eqref{eq:VD_H} for a given $n$ to find $y_i(t)$ and $U_i(t)$, and define
a piecewise linear interpolation.
This makes $(u_n)_x$ piecewise constant with value $\D_+U_i / \D_+y_i$ for $x \in [y_i, y_{i+1})$.
Note that there is no trouble with dividing by zero as the corresponding intervals are empty.
When applying the scheme to the 2CH system, we follow the convention in \cite{vardisc} with a piecewise constant
interpolation $\rho_n(t,x)$ for the density, setting it equal to
$r_i(0)\D_+y_i(0)/\D_+y_i(t)$ for $x \in [y_i(t), y_{i+1}(t))$.
For initial data, unless otherwise specified, we follow the multipeakon method in choosing $y_i(0) = \xi_i$,
$U_i(0) = u_0(\xi_i)$, $(\rho_0)_i = \rho_0(\xi_i) = r_i(0)$, computing $h_i(0)$ according to $\eqref{eq:h_id}$,
and then compute $H_i(0)$ as \eqref{eq:H}.

\subsubsection{Finite difference schemes}
As they remain a standard method for solving PDEs numerically, it comes as no surprise that several finite
difference schemes have been proposed for the CH equation.
We will consider the convergent dissipative schemes for \eqref{eq:CH} presented in
\cite{Holden2006, Coclite2008}, and the energy-preserving scheme for \eqref{eq:CH} and \eqref{eq:2CH} studied
numerically in \cite{LiuPen2016}.
The schemes in \cite{Holden2006,LiuPen2016} are both based on the following reformulation of \eqref{eq:CH},
\begin{equation}\label{eq:CH_m}
m_t + (m u)_x + m u_x = 0, \qquad m = u - u_{xx}
\end{equation}
for $u = u(t,x)$ and $m = m(t,x)$. Here $u(t,\cdot) \in \Hone(\T)$ means that $m(t,\cdot)$ corresponds to a Radon measure
on $\T$.
Then, with \eqref{eq:CH_m} as starting point, and grid points $x_j = j\dx, \dx > 0$ we may apply finite 
differences, as defined in \eqref{eq:diff}, to obtain various semidiscretizations, specifically
\begin{equation}\label{eq:fdCH_HR}
\dot{m}_j = -\D_-(m_j u_j) - m_j \D_0u_j, \qquad m_j = u_j -\D_-\D_+u_j
\end{equation}
which is the discretization studied in \cite{Holden2006} under the assumption of $m$ initially being a \textit{positive} Radon
measure.
In the same paper they also briefly mention three alternative evolution equations for $m_j$,
\begin{subequations}
  \begin{align}
  \dot{m}_j &= -\D_-(m_j u_j) - m_j \D_+u_j, \\
  \dot{m}_j &= -\D_0(m_j u_j) - m_j \D_0u_j, \label{eq:fdCH_HR2} \\
  \dot{m}_j &= -\D_+(m_j u_j) - m_j \D_-u_j,
  \end{align}
\end{subequations}
which for different reasons were troublesome in practice when integrating in time using the explicit Euler method.
On the other hand, in \cite{LiuPen2016} they use the following discretization,
\begin{equation}\label{eq:fdCH_LP}
\dot{m}_j = -\D_0(m_j u_j) - m_j \D_0u_j, \qquad m_j = u_j -\D_0\D_0u_j
\end{equation}
which coincides with \eqref{eq:fdCH_HR2} except that they use a wider stencil when defining $m_j$.
A difference operator which approximates the $r$\textsuperscript{th} derivative using exactly $r+1$ consecutive grid points
is called compact, cf.\ \cite[Ch.\ 3]{NumEvol}. Clearly, $\D_{\pm}$ and $\D_-\D_+$ are compact difference operators,
while $\D_0$ and $\D_0\D_0$ are not.
As pointed out in \cite[Ch.\ 7]{NumEvol}, noncompact difference operators are notorious for producing spurious oscillations,
and this is exactly the problem reported in \cite{Holden2006} for \eqref{eq:fdCH_HR2}.
Similarly, in \cite{LiuPen2016} the authors remark that oscillations may appear when the solution of
\eqref{eq:CH} becomes less smooth.
In these cases they propose an adaptive strategy of adding numerical viscosity to the scheme with the drawback that the discrete
energy is no longer conserved.
We have not incorporated such a strategy here, as we would like an energy-preserving finite difference scheme to compare with
our energy-preserving variational discretizations.

An invariant-preserving discretization of the 2CH system \eqref{eq:2CH} is also presented in \cite{LiuPen2016},
and using the notation \eqref{eq:avg} we can write it as
\begin{align}\label{eq:fd2CH}
\begin{aligned}
\dot{m}_j &= -\D_0(m_j u_j) - m_j \D_0u_j - \left(\delta\rho_j + \delta\rho_{j-1} \right) \D_0\rho_j, \\
\dot{\rho}_j &= -\D_- (\bar{\rho}_j \bar{u}_j),
\end{aligned}
\end{align}
still with $m_j = u_j -\D_0\D_0u_j$.
The semidiscretizations \eqref{eq:fdCH_LP} and \eqref{eq:fd2CH} are conservative in the sense that both
preserve the invariants
\begin{equation}\label{eq:invariants}
\dx \sum_{i = 0}^{n-1} m_i, \qquad \frac12 \dx \sum_{i = 0}^{n-1} \left( u_i^2 + (\D_0u_i)^2 \right)
\end{equation}
which respectively correspond to the momentum and energy of the system, and have counterparts in
\eqref{eq:disc_mom} and \eqref{eq:disc_energy} for the system \eqref{eq:semidisc_sys}.
In addition, \eqref{eq:fd2CH} preserves the discrete mass
\begin{equation*}
\dx\sum_{i = 0}^{n-1} \rho_i.
\end{equation*}

A somewhat more refined spatial discretization is employed in \cite{Coclite2008}.
This method is based on yet another reformulation of \eqref{eq:CH},
\begin{equation}\label{eq:CH_P}
u_t + uu_x + P_x = 0, \qquad P - P_{xx} = u^2 + \frac12 u_x^2,
\end{equation}
and its discretization reads
\begin{subequations}\label{eq:fdCH_CKR}
  \begin{equation}
  \dot{u}_{j+1/2} + (u_{j+1/2} \vee 0) \D_-u_{j+1/2} + (u_{j+1/2} \wedge 0) \D_+u_{j+1/2} + \D_+P_j = 0 
  \end{equation}
  with
  \begin{equation}
  P_j - \D_-\D_+P_j = (u_{j+1/2} \vee 0)^2 + (u_{j+1/2} \wedge 0)^2 + \frac12 (\D_-u_{j+1/2})^2.
  \end{equation}
\end{subequations}
Not only do they use a staggered grid, but they also use $u \vee 0$ and $u \wedge 0$, the positive and negative parts of $u$
respectively, to obtain the proper upwinding required for dissipation only in the $\D_- u$-part and not the
$u$-part of the associated discrete energy
\begin{equation*}
\dx \sum_i \left( u_i^2 + (\D_-u_i)^2 \right).
\end{equation*}
Contrary to \cite{Holden2006}, this dissipative scheme allows for
initial data of any sign for $u$.
Moreover, this precarious choice of positive and negative parts of $u$ can be linked to
the traveling direction of the wave profile $u$: for $u_{j+1/2} > 0$ the solution moves to
the right, thus values to the right have a greater influence on the solution and it is
reasonable to use a forward difference.
For $u_{j+1/2} < 0$ the solution travels to the left and one analogously uses a backward difference.

For $u_n$ corresponding to \eqref{eq:fdCH_HR}, \eqref{eq:fdCH_LP}, \eqref{eq:fd2CH}, and
\eqref{eq:fdCH_CKR} we have made a piecewise interpolation between the grid points, meaning $(u_n)_x$ is
piecewise constant.
For initial data we define $u_i(0) = u_0(x_i)$ for $x_i = i \dx$, $n\dx = \period$, and apply the corresponding
discrete Helmholtz operator to produce $m_i(0)$ for the schemes \eqref{eq:fdCH_HR}, \eqref{eq:fdCH_LP}, and
\eqref{eq:fd2CH}.

\subsubsection*{A comment on the inversion of the discrete Helmholtz operator}
In both \cite{Holden2006} and \cite{Coclite2008} they compute the Green's function corresponding to the discrete Helmholtz
operator $\id - \D_-\D_+$ using a difference equation. As \cite{Holden2006} concerns the periodic CH equation, they periodize
this function to obtain the periodic Green's function which can be restated in our variables as
\begin{equation*}
g_j^{\text{p}} = \frac{1}{\sqrt{4 + \dx^2}}\frac{e^{-\kappa j} + e^{\kappa(j-n)} }{1 - e^{-\kappa n} } = \frac{1}{\sqrt{4 + \dx^2}} \frac{\cosh(\kappa \left( j - \frac{n}{2} \right))}{\sinh\left(\kappa \frac{n}{2}\right)}
\end{equation*}
for $j \in \{0,\dots,{n-1}\}$ with
\begin{equation*}
\kappa = \ln\left( 1 + \frac{\dx^2}{2} + \dx\sqrt{4 + \dx^2} \right), \qquad n \dx = \period.
\end{equation*}
Notice the resemblance to the periodized exponential \eqref{eq:ker_per} in the continuous case.
Defining the matrix $M$ corresponding to the discrete Helmholtz operator with periodic boundary conditions,
\begin{equation*}
M = I - \frac{1}{\dx^2} \begin{bmatrix}
-2 & 1 & 0 & \cdots & 1 \\
1 & -2 & \ddots & \ddots & \vdots \\
0 & \ddots & \ddots & \ddots & 0 \\
\vdots & \ddots & \ddots & -2 & 1 \\
1 & \cdots & 0 & 1 & -2
\end{bmatrix},
\end{equation*}
we have $(M g^{\text{p}})_j = \delta_{0,j}$ for the Kronecker delta $\delta_{i,j}$. Consequently, they compute $u$ from $m$ by a convolution
\begin{equation*}
u_j = \sum_{k=0}^{n-1}g^{\text{p}}_{j-k} m_k.
\end{equation*}
In the implementation they compute this convolution efficiently by employing the \emph{fast Fourier transform} (FFT) and the
convolution theorem for the discrete Fourier transform (DFT), namely $u = \Fcal_n^{-1}\left[ \Fcal_n[g^{\text{p}}] \cdot \Fcal_n[m] \right]$
with
\begin{equation*}
(\Fcal_n[f])_j = \sum_{k=0}^{n-1}f_k e^{-\iu \frac{2\pi}{n}k j }, \qquad (\Fcal_n^{-1}[\hat{f}])_j = \frac1n \sum_{k=0}^{n-1}\hat{f}_k e^{\iu \frac{2\pi}{n}k j }.
\end{equation*}
The discrete Fourier transform is more than an efficient
tool for evaluating the convolution in this case. Indeed, the matrix $M$ is circulant and thus diagonalizable using
the DFT matrix, cf.\ \cite[p.\ 379]{MatrixApplied}.
Defining the matrix $V$ through $V_{j,k} = \frac{1}{\sqrt{n}}e^{\iu \frac{2\pi}{n} j k}$ for
$j,k \in \{0,\dots,{n-1}\}$ we find that $V^* M V$ is indeed a diagonal matrix containing the eigenvalues of $M$,
\begin{equation*}
d_j = 1 + \left(\frac{2}{\dx} \sin\left(\frac{\pi j}{n}\right) \right)^2, \qquad j \in {0,\dots,{n-1}}.
\end{equation*}
A little rearrangement then shows $(\Fcal_n[g^\text{p}])_j =d_j^{-1}$, or equivalently $g^\text{p}_j = (\Fcal_n^{-1}[d^{-1}])_j$, revealing an alternative method for computing the periodic Green's function or computing $u$ through $u = \Fcal_n^{-1}\left[ d^{-1} \cdot \Fcal_n[m] \right]$.
This can also be seen by directly inserting $\Fcal_n^{-1}[\Fcal_n[g^\text{p}]]$ for $g^\text{p}$ in the difference equation
\begin{equation*}
-\frac1\dx g_{j+1}^\text{p} + \left(1+\frac2\dx\right) g_j^\text{p} - \frac1\dx g_{j-1}^\text{p} = \delta_{0,j}
\end{equation*}
considered in \cite{Holden2006,Coclite2008}, rearranging coefficients and applying the DFT.

The above method is also convenient for computing the inverse of the noncompact discrete Helmholtz matrix in \cite{LiuPen2016},
as its eigenvalues
\begin{equation*}
1 + \left(\frac{1}{\dx} \sin\left(\frac{2 \pi j}{n}\right) \right)^2, \qquad j \in {0,\dots,{n-1}}
\end{equation*}
can be computed from circulant matrix theory.
Hence, this is how we have implemented the computation of $u$ in \eqref{eq:fdCH_HR}, \eqref{eq:fdCH_LP}, and of $P$ in \eqref{eq:fdCH_CKR}.
In our examples, when compared to computing $g^{\text{p}}$ by solving the nearly tridiagonal system with $M$,
the FFT method was consistently faster.

\subsubsection{Pseudospectral (Fourier collocation) scheme}
Let us consider the so-called pseudospectral scheme used in the study of traveling waves for the CH equation
in \cite{Kalisch2005}, see also \cite{SpectralMatlab} for an introduction to the general idea.
This method is based on applying Fourier series to \eqref{eq:CH} and solving the resulting evolution equation in the
frequency domain.
Introducing the scaling factor $a \coloneqq \period/2\pi$ and assuming the discretization parameter $n$ to be even, the
pseudospectral method for \eqref{eq:CH} with period $\period$ can be written
\begin{equation}\label{eq:PS}
\dot{V}_n(t,k) = -\frac{\iu k}{2(a^2 + k^2)} \left[ (3 a^2 + k^2)  \Fcal_n \! \left[(\Fcal_n^{-1}[V])^2\right]\!\!(t,k) + \Fcal_n \!\left[(\Fcal_n^{-1}[\iu k V])^2\right]\!\!(t,k) \right]
\end{equation}
with $V(0,k) = \Fcal_n[u_0](k)$ for $k \in \{-\frac{n}{2},\dots,\frac{n}{2}-1 \}$.
Here $\iu k V$ means the pointwise product of the vectors $\iu [-\frac{n}{2},\dots,\frac{n}{2}-1]$ and $[V(-\frac{n}{2}),\dots,V(\frac{n}{2}-1)]$, while $\Fcal_n$ and $\Fcal_n^{-1}$ are the discrete Fourier transform and its inverse, defined as
\begin{align*}
\Fcal_n[v](k) &= \sum_{j = 0}^{n-1} v(x_{j}) e^{-i k x_{j}}, \\
\Fcal_n^{-1}[V](x_j) &= \frac{1}{n}\sum_{k = -n/2}^{n/2-1} V(k) e^{i k x_{j}}
\end{align*}
for $x_j = 2\pi j /n$, $j \in \{0,\dots,n-1\}$.
The right-hand side of \eqref{eq:PS} can be efficiently computed by applying FFT, and to evaluate the resulting $u$
and $u_x$ on the grid $x_j$ one computes $u_n(x_j) = \Fcal_n^{-1}[V](x_j)$ and
$(u_n)_x(x_j) = \Fcal_n^{-1}[\iu k V](x_j)$.
Unfortunately, this scheme is prone to \textit{aliasing}, and for this reason the authors of \cite{Kalisch2006} propose a
modified version of the scheme which employs the \textit{Orszag 2/3-rule}.
For more details on aliasing and the 2/3 rule we refer to \cite[Ch.~11]{ChebSpec}.
This 2/3 rule amounts to removing the Fourier coefficients $V(k)$ corresponding to one third of the frequencies $k$ before
applying the inverse transform in \eqref{eq:PS}, specifically those frequencies of largest absolute value.
In our setting this means that we keep $V(k)$ as is for $k \in \{-\frac{n}{2},\dots,\frac{n}{2}-1\}$ satisfying
$\abs{k} \le \frac{n}{3}$, while setting $V(k) = 0$ for the rest.
For this \textit{dealiased} scheme, which they call a Fourier collocation method, the authors in \cite{Kalisch2006}
prove convergence in $\Ltwo$-norm for sufficiently regular solutions of \eqref{eq:CH}.

When interpolating $u_n$ on a denser grid containing $x_j$ we must use the corresponding real Fourier basis function
for each frequency $k$ to obtain the correct representation of the pseudospectral solution, which will always be
smooth. For this we use the routine \texttt{interpft} which interpolates using exactly Fourier basis functions.

\subsection{Example 1: Smooth traveling waves}
To the best of our knowledge, there are no explicit formulae for smooth traveling wave solutions of either \eqref{eq:CH} or \eqref{eq:2CH}, and to obtain such solutions we make use of numerical integration in the spirit of \cite{Cohen2014} and \cite{Cohen2008}.

\subsubsection{Computing reference solution for the CH equation}
Traveling waves are solutions of the form $u(t,x) = \varphi(x-ct)$.
Inserting this ansatz in \eqref{eq:CH} yields
\begin{equation*}
-c\varphi' + c\varphi''' + 3\varphi\varphi' - 2\varphi'\varphi'' - \varphi\varphi''' = 0.
\end{equation*}
Assuming $\varphi \neq c$ and multiplying with $(c-\varphi)$ the above equation can be rearranged to  yield
\begin{equation*}
\left( (c-\varphi)^2 (\varphi''-\varphi) \right)' = 0,
\end{equation*}
which after integration gives
\begin{equation*}
\varphi''(z) = \varphi(z) + \frac{B}{(c-\varphi(z))^2}
\end{equation*}
for some constant $B \in \R$.
For the right choice of $B$, this ordinary differential equation can be integrated numerically to give periodic
smooth solutions.
Inspired by \cite{Cohen2008} we have chosen $c = -B = 3$ and initial conditions $\varphi(0) = 1$, $\varphi'(0) = 0$.
To integrate we used \texttt{ode45} with very strict tolerances, namely $\texttt{AbsTol} = \texttt{eps}$ and
$\texttt{RelTol} = 100\:\texttt{eps}$, where $\texttt{eps} = 2^{-52}$ is the distance from 1.0 to the next double
precision floating point number representable in \textsc{Matlab}.
After integration we found the solution to have period $p = 6.4695469424989$, where the first ten decimal digits
agree with the period found in the experiments section of \cite{Cohen2008}.

\subsubsection{Numerical results for the CH equation}
Figures \ref{fig:smoothCH_N16} and \ref{fig:smoothCH_results} display numerical results for the smooth reference
solution above after moving one period $\period$ to the right.
As the traveling wave has velocity $c = 3$, this corresponds to integrating over a time period $\period/3$.
To integrate in time we have applied \texttt{ode45} with parameters $\texttt{AbsTol} = \texttt{RelTol} = 10^{-10}$.

To highlight the different properties of each scheme, Figure \ref{fig:smoothCH_N16} displays $u_n$ and $(u_n)_x$
for the various schemes for the low number $n = 2^4$ and interpolated on a reference grid with step size
$2^{-10}\period$.
It is apparent how the dissipative nature of the schemes \eqref{eq:fdCH_HR} and \eqref{eq:fdCH_CKR}
reduces the height of the traveling wave such that it lags behind the true solution.
This effect is particularly severe for \eqref{eq:fdCH_HR}, which probably explains why its error
displayed in Figure \ref{fig:smoothCH_results_rates} is consistently the largest.
The perhaps most obvious feature in Figure \ref{fig:smoothCH_N16_ux} is the large-amplitude deviations introduced
by the discontinuities for the multipeakon scheme.
As indicated by its decreasing $\Hone$-error in Figure \ref{fig:smoothCH_results_rates},
the amplitudes of these discrepancies reduce as $n$ increases.

\begin{figure}
  \begin{subfigure}{0.495\textwidth}
    \includegraphics[width=1\textwidth]{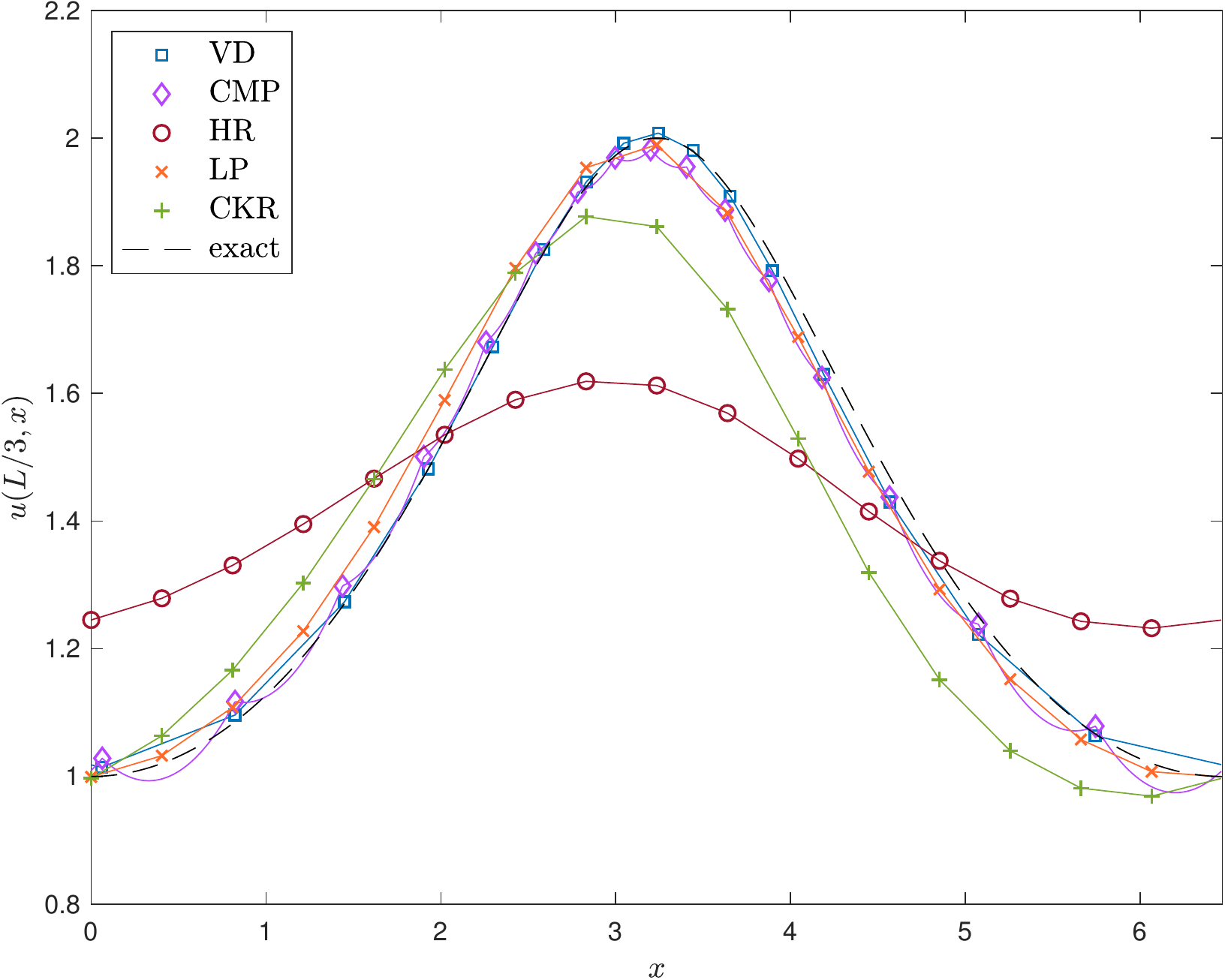}
    \caption{} \label{fig:smoothCH_N16_u}
  \end{subfigure}
  \begin{subfigure}{0.495\textwidth}
    \includegraphics[width=1\textwidth]{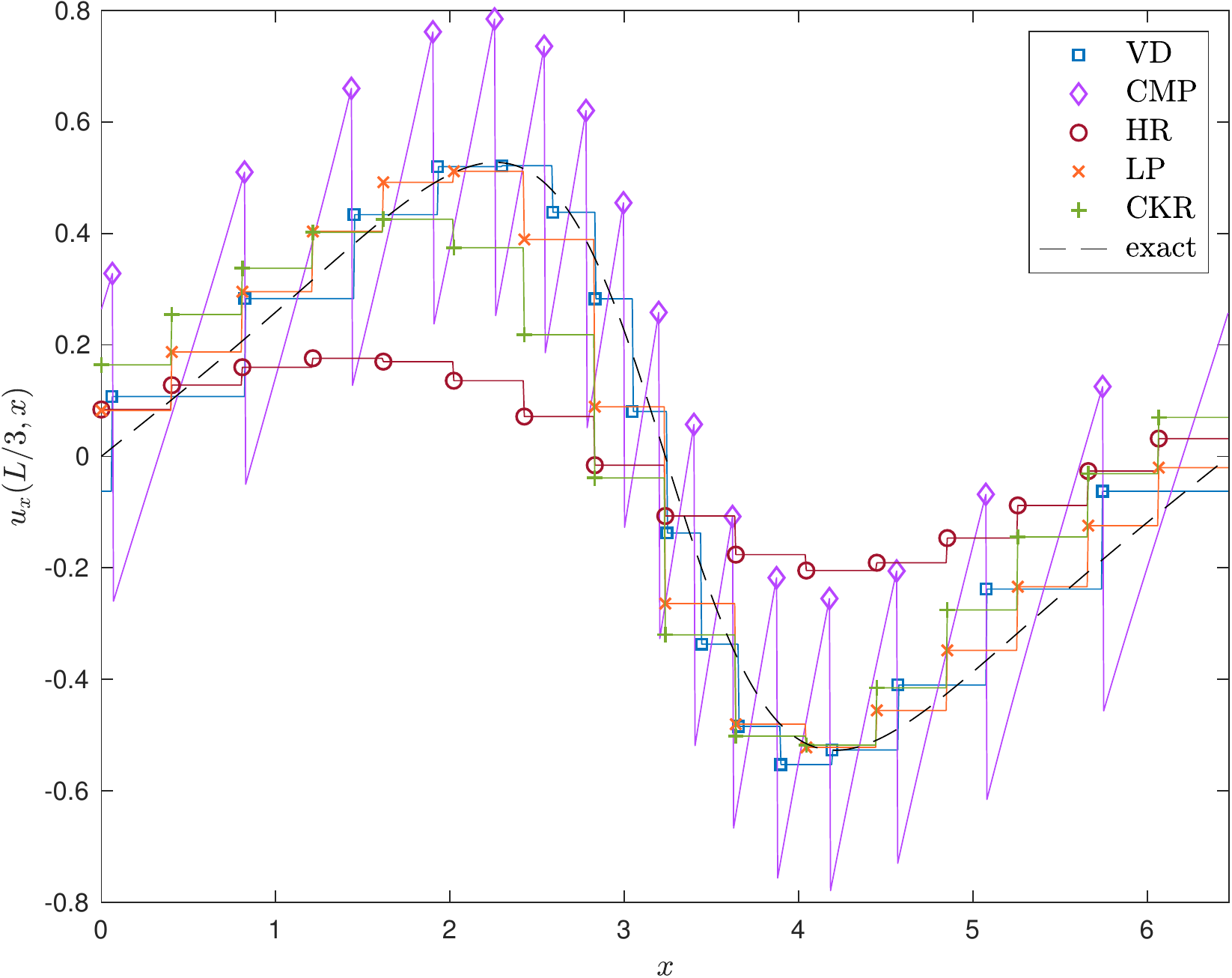}
    \caption{} \label{fig:smoothCH_N16_ux}
  \end{subfigure}
  \caption{Smooth traveling wave for the CH equation. Plot of the interpolated numerical solutions $u_n$ (\textsc{a}) and $(u_n)_x$ (\textsc{b}) at time $t = \period/3$ for $n = 2^4$ and evaluated on a reference grid with step size $2^{-10}\period$. The schemes considered are VD \eqref{eq:VD_H}, CMP \eqref{eq:cmp_per}, HR \eqref{eq:fdCH_HR}, LP \eqref{eq:fdCH_LP}, and CKR \eqref{eq:fdCH_CKR}. }
  \label{fig:smoothCH_N16}
\end{figure}

\begin{figure}
  \begin{subfigure}{0.495\textwidth}
    \includegraphics[width=1\textwidth]{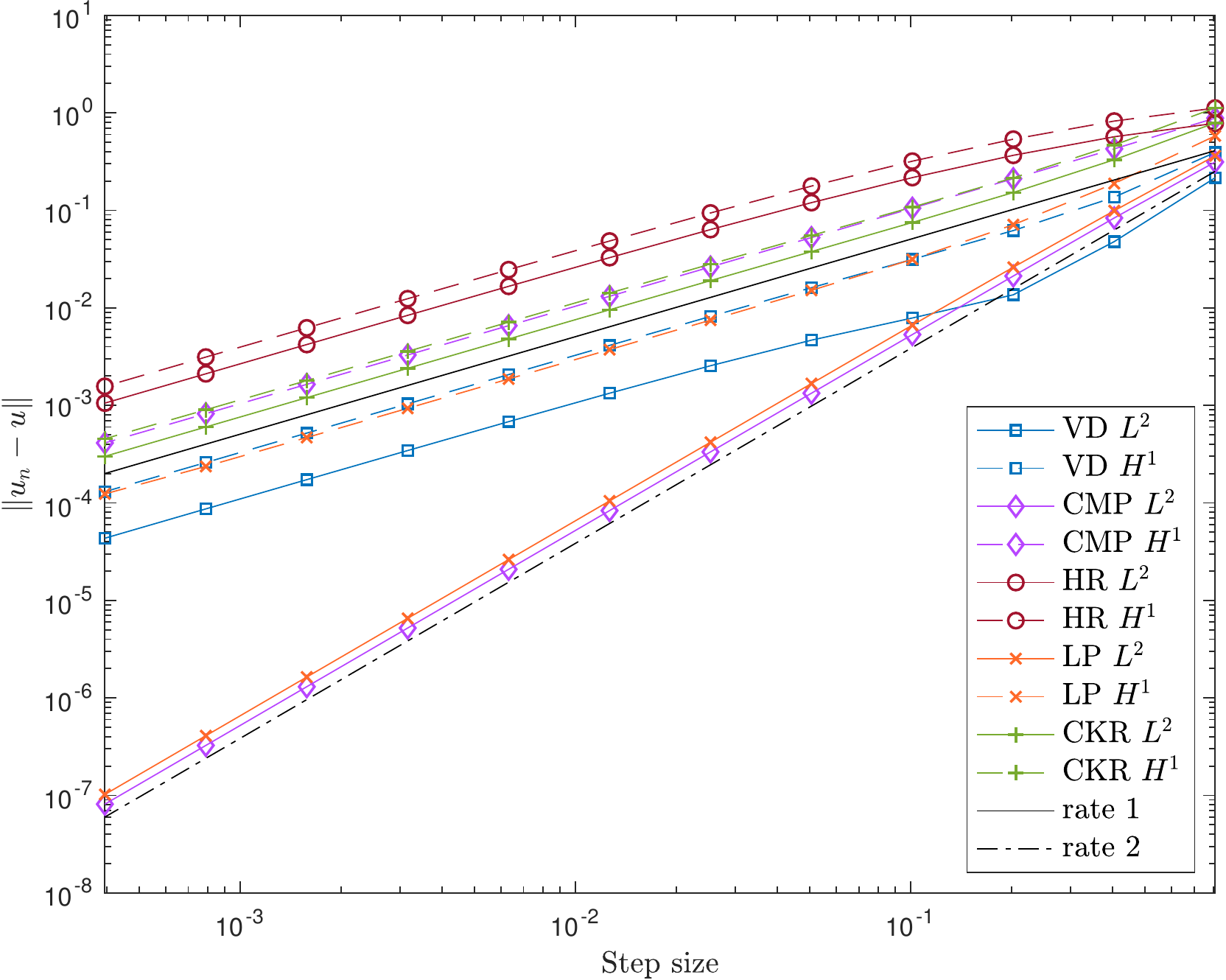}
    \caption{} \label{fig:smoothCH_results_rates}
  \end{subfigure}
  \begin{subfigure}{0.495\textwidth}
    \includegraphics[width=1\textwidth]{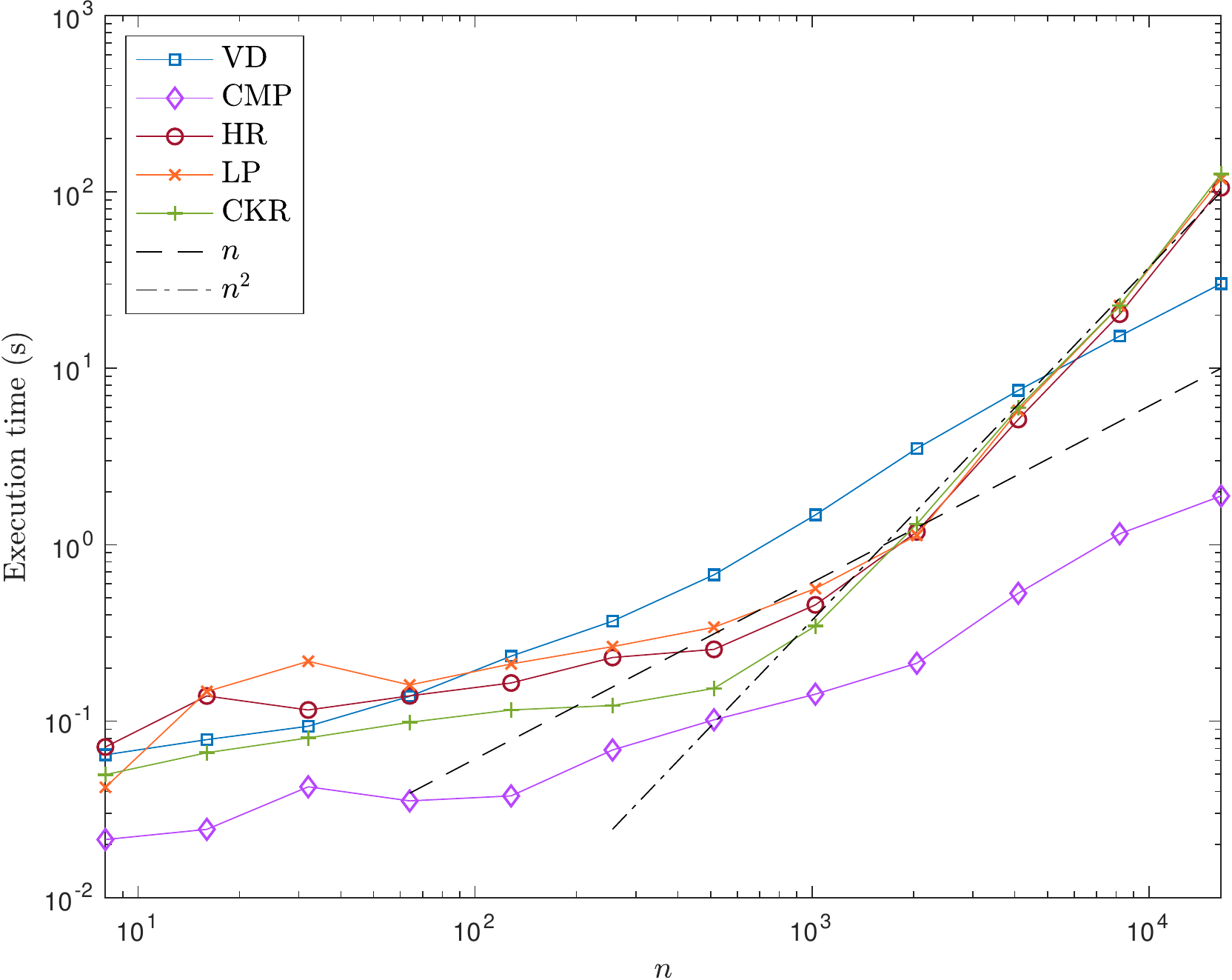}
    \caption{} \label{fig:smoothCH_results_times}
  \end{subfigure}
  \caption{Smooth traveling wave for the CH equation. Errors in $\Ltwo$- and $\Hone$-norms after one period (\textsc{a}), and execution times for \texttt{ode45} in seconds (\textsc{b}). The schemes tested are VD \eqref{eq:VD_H}, CMP \eqref{eq:cmp_per}, HR \eqref{eq:fdCH_HR}, LP \eqref{eq:fdCH_LP} and CKR \eqref{eq:fdCH_CKR} for step sizes $2^{-k} \period$ where $3 \le k \le 14$,
    and evaluated on a reference grid with step size $2^{-16}\period$.}
  \label{fig:smoothCH_results}
\end{figure}

Figure \ref{fig:smoothCH_results_rates} contains $\Ltwo$- and $\Hone$-errors of the interpolated solutions for
$n = 2^k$ with $k \in \{3,\dots,14\}$, evaluated on a reference grid with $k_0 = 16$.
Before commenting on these convergence results, we underline that the pseudospectral method \eqref{eq:PS}
has not been included in the figure, as its superior performance for this example would make it
hard to differ between the plots for the remaining methods.
Indeed, this scheme displays so-called spectral convergence in both $\Ltwo$- and $\Hone$-norm, and exhibits
an $\Ltwo$-error close to rounding error already for $n = 2^6$.

For the remaining methods it is perhaps not surprising that the finite difference scheme \eqref{eq:fdCH_LP} based
on central differences in general has the smallest error in $\Hone$-norm for this smooth reference solution,
exhibiting convergence orders of 2 and 1 for $\Ltwo$ and $\Hone$ respectively.
However, we observe that the $\Ltwo$-error of the multipeakon scheme is consistently the lowest, but its
convergence in $\Hone$ is impeded by its irregular derivative.
Moreover, it appears that for small $n$, i.e., $n \le 2^5$ in this setting, the variational scheme
\eqref{eq:VD_H} performs better.

An observation regarding the execution time of \texttt{ode45} is that the finite difference schemes
\eqref{eq:fdCH_HR}, \eqref{eq:fdCH_LP} and \eqref{eq:fdCH_CKR} seem to experience some tipping point around
$k = 11$ where their running times tend to be of complexity $\Ocal(n^2)$ rather than $\Ocal(n)$, see
Figure \ref{fig:smoothCH_results_times}.
A closer look at the statistics for the time integrator in these cases reveals that from $k = 11$ and onwards,
the solver starts experiencing failed attempts at satisfying the specified error tolerances, thus increasing
the execution time.
This does not occur for the schemes in Lagrangian coordinates, which exhibit execution times aligning well with
the $\Ocal(n)$-reference line.
A possible explanation for this could be that the semidiscrete schemes based in Lagrangian coordinates are easier
to handle for time integrator.
Indeed, the almost semilinear structure of the ODE system corresponding to \eqref{eq:VD_H} in \cite{vardisc} is
key to its existence and uniqueness proofs, and perhaps this structure is advantageous also for the ODE solvers.

It is however important to impose sufficiently strict error tolerances for the temporal integration error to be
negligible compared to the spatial discretization for the smallest step sizes.
Furthermore, in the case of \eqref{eq:fdCH_LP} it appears important to not have too large error tolerances
irrespective of $n$ to avoid oscillations in $m_i$.
Since the convolution with the periodic Green's function to compute $u_i$ from $m_i$ is a regularizing process,
we observe from experiments that seemingly well-behaved $u_i$ may hide extremely oscillatory $m_i$.
This is not surprising, as $m_i$ is a discretization of what in general may be a measure, and thus much less regular
than $u_i$.

It should be emphasized that the multipeakon scheme \eqref{eq:cmp_per} is considerably faster than the
other schemes, which likely comes from it being the only scheme where no matrix equations are solved.
Thus, the fast summation algorithm appears to benefit the multipeakon scheme in this direction.

\subsubsection{Computing reference solution for the 2CH system}
\label{sss:smooth2CH}
Here we essentially follow the steps for the CH equation, with some slight modifications.
We make the ansatz $u(t,x) = \varphi(x-ct)$ and $\rho(t,x) = \psi(x-ct)$.
Plugging these into \eqref{eq:2CH} leads to the system
\begin{subequations}
  \begin{align}
  -c\varphi' + c\varphi''' + 3\varphi\varphi' - 2\varphi'\varphi'' - \varphi\varphi''' + \psi \psi' &= 0 \label{eq:smooth2CH_u}\\
  -c\psi' + (\psi \varphi)' &= 0. \label{eq:smooth2CH_rho}
  \end{align}
\end{subequations}
Integration of \eqref{eq:smooth2CH_rho} yields the relation
\begin{equation}\label{eq:smooth2CH_rho2}
\psi = \frac{A}{c-\varphi},
\end{equation}
for some constant $A \in \R$.
This expression makes sense as long as $\varphi \neq c$, which we will assume from now on.
Using the above relation we can replace $\varphi$ by \eqref{eq:smooth2CH_rho2} in \eqref{eq:smooth2CH_u}. Rearranging we get in a similar manner as for the CH equation,
\begin{equation*}
\left( (c-\varphi)^2 (\varphi''-\varphi) +\frac{A^2}{c-\varphi}\right)' = 0.
\end{equation*}
Integration gives
\begin{equation*}
\varphi''(t) = \varphi(t) - \frac{A^2}{(c-\varphi(t))^3} + \frac{B}{(c-\varphi(t))^2}
\end{equation*}
for some constant $B \in \R$.
We follow \cite{Cohen2014} in choosing $c = A = -B = 2$ and initial conditions $\varphi(0) = 0.5$, $\varphi'(0) = 0$.
Proceeding as in the case of the CH equation we obtain the period $p = 5.1475159326651$
where the four first decimal digits agree with the four decimal places provided in \cite{Cohen2014}.

\subsubsection{Numerical results for the 2CH system}
For this example we have only compared the variational scheme \eqref{eq:VD_h} to \eqref{eq:fd2CH}, as these are
the only methods presented in Section \ref{ss:review} applicable to the 2CH system.
As in the experiment for the CH equation we want to measure the error after the wave has moved a distance $\L$ to
the right, corresponding to one period.
Since the velocity of the solution now is $c = 2$, this corresponds to integrating from $t = 0$ to $t = \frac{\period}{2}$, which we did using \texttt{ode45} with $\texttt{AbsTol} = \texttt{RelTol} = 10^{-8}$.
Figure \ref{fig:smooth2CH_plot} shows the interpolants $u_n$ and $\rho_n$ for $n = 2^4$, together with the
exact solutions.
\begin{figure}
  \begin{subfigure}{0.495\textwidth}
    \includegraphics[width=1\textwidth]{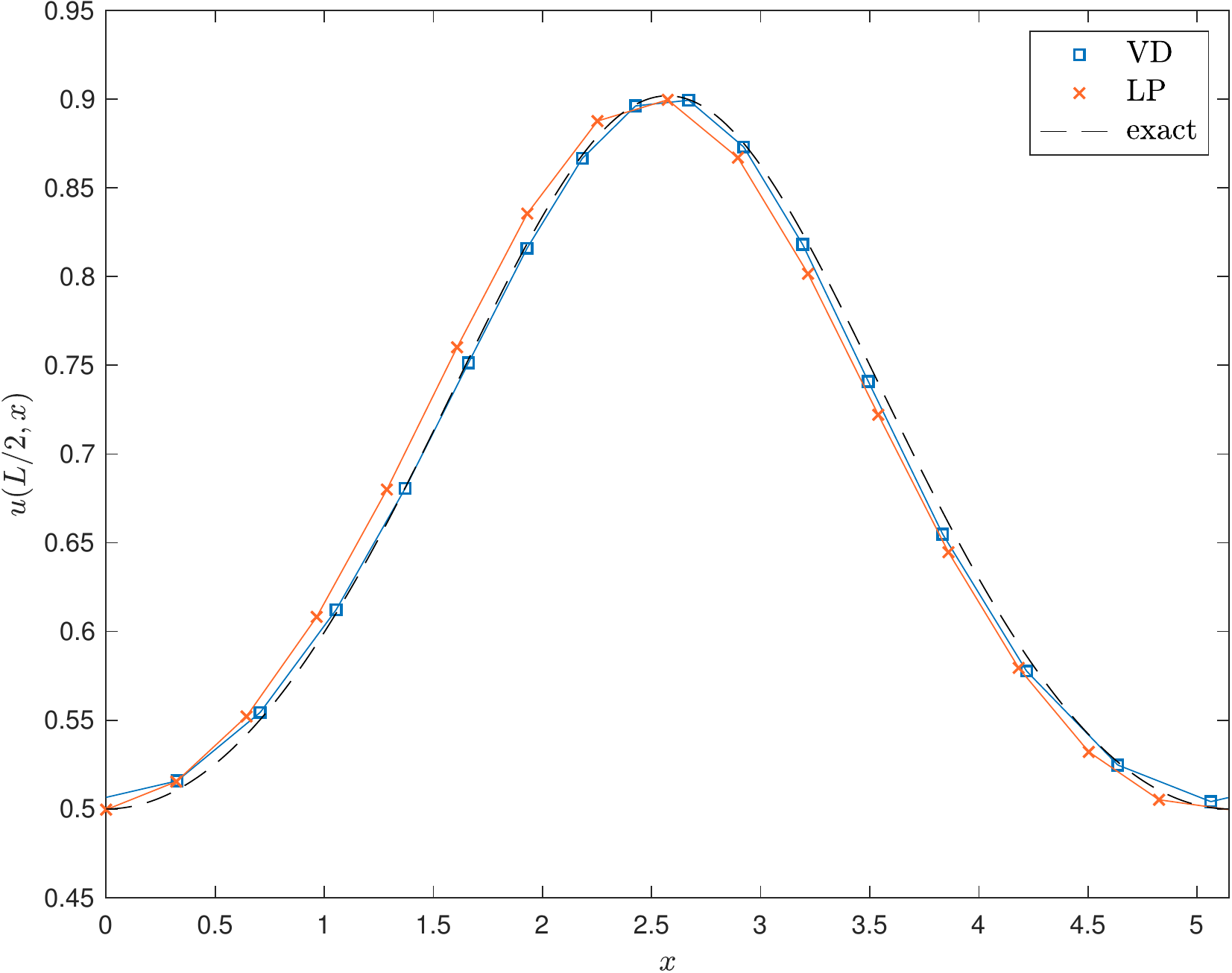}
    \caption{} \label{fig:smooth2CH_plot_u}
  \end{subfigure}
  \begin{subfigure}{0.495\textwidth}
    \includegraphics[width=1\textwidth]{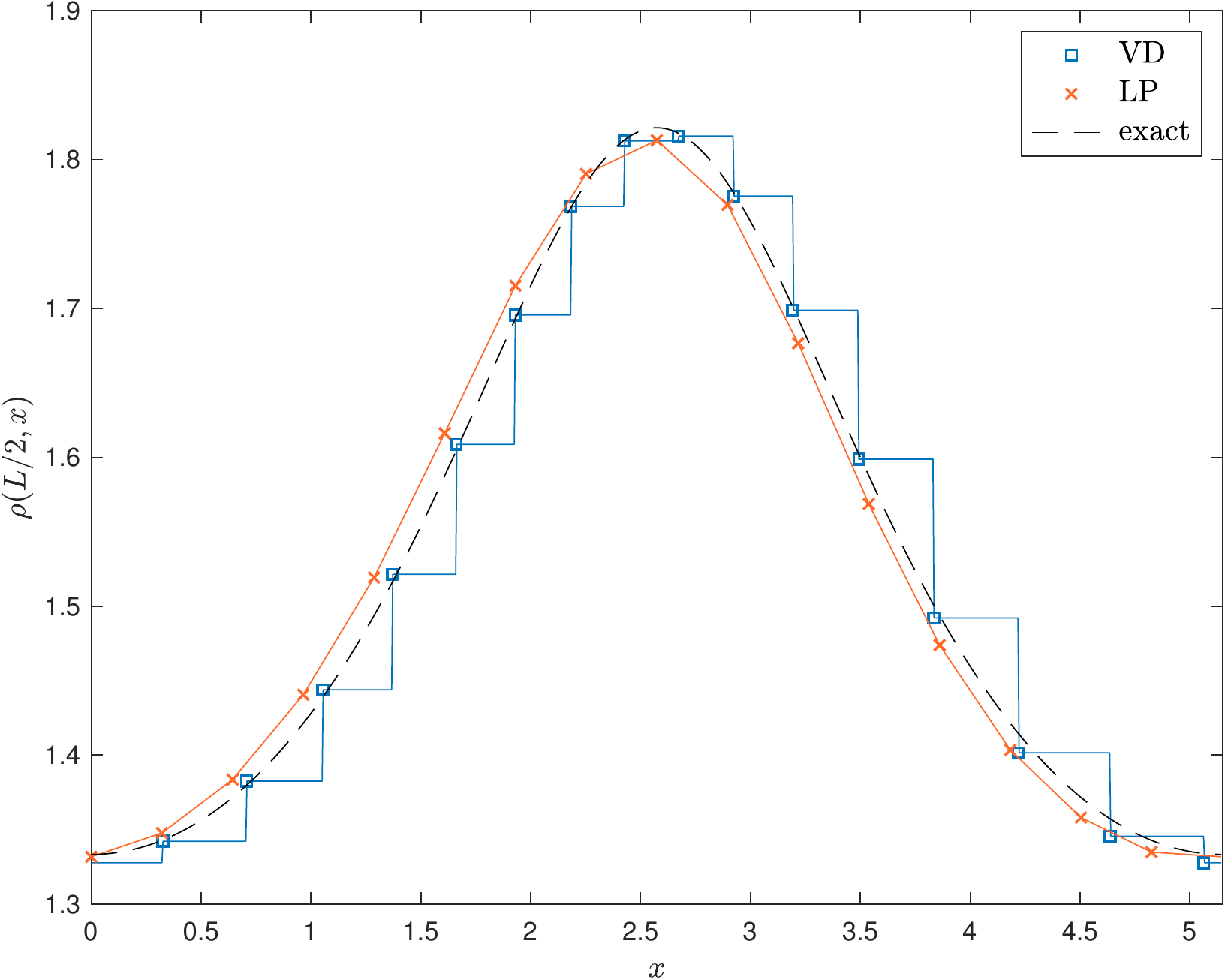}
    \caption{} \label{fig:smooth2CH_plot_rho}
  \end{subfigure}
  \caption{Smooth traveling wave for the 2CH system. Plot of the interpolated numerical solutions $u_n$ (\textsc{a}) and $\rho_n$ (\textsc{b}) at time $t = \period/2$ for $n = 2^4$ and evaluated on a reference grid with step size $2^{-10}\period$. The schemes considered are VD \eqref{eq:VD_H} and LP \eqref{eq:fd2CH}.}
  \label{fig:smooth2CH_plot}
\end{figure}
The reader may wonder why we have interpolated $\rho_n$ differently for the two methods, as it is clear from
Figure \ref{fig:smooth2CH_plot_rho} that the variational scheme would be much closer to the reference solution
if one had used a piecewise linear interpolation of $\rho_n$.
In fact, when looking at the convergence rates of $\rho_n$ in this case, the variational scheme exhibits rate 1
convergence for both piecewise linear and piecewise constant interpolations.
Since the discrete density is computed using $\D_+y$ which is piecewise constant when $y$ is piecewise linear,
we follow \cite{vardisc} in using a piecewise constant $\rho_n$.
For the scheme \eqref{eq:fd2CH} on the other hand, the convergence rate actually depends on the interpolation, as
piecewise linear $\rho_n$ has convergence rate 2, while piecewise constant interpolation gives approximate rate 1.
It is then only reasonable to use the interpolation which performs better.
These convergence rates are shown in Figure \ref{fig:smooth2CH_results} together with execution times for
\texttt{ode45}, where $n = 2^k$, $k \in \{3,\dots,14\}$, and the reference grid has $k_0 = 16$.
\begin{figure}
  \begin{subfigure}{0.495\textwidth}
    \includegraphics[width=1\textwidth]{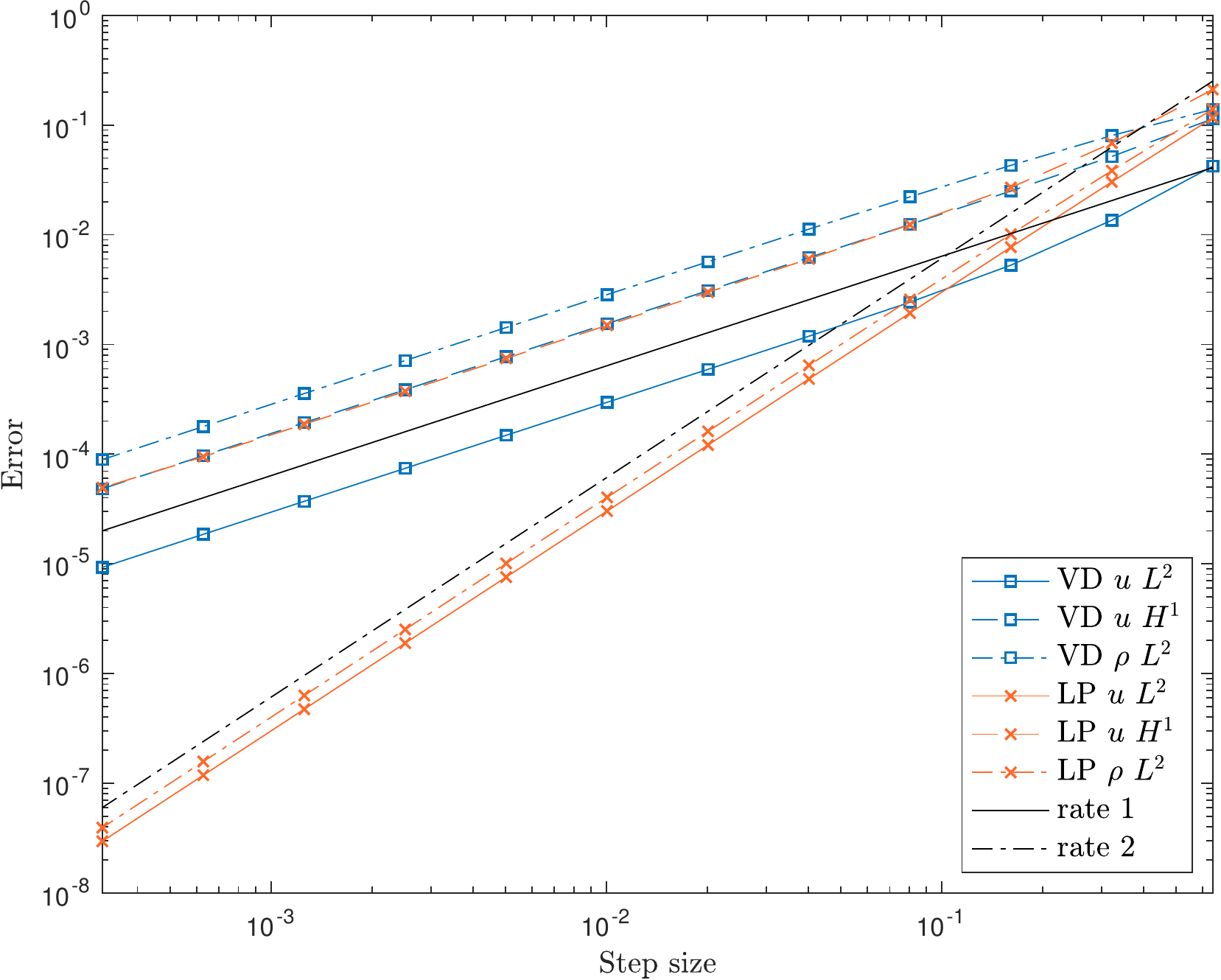}
    \caption{}
  \end{subfigure}
  \begin{subfigure}{0.495\textwidth}
    \includegraphics[width=1\textwidth]{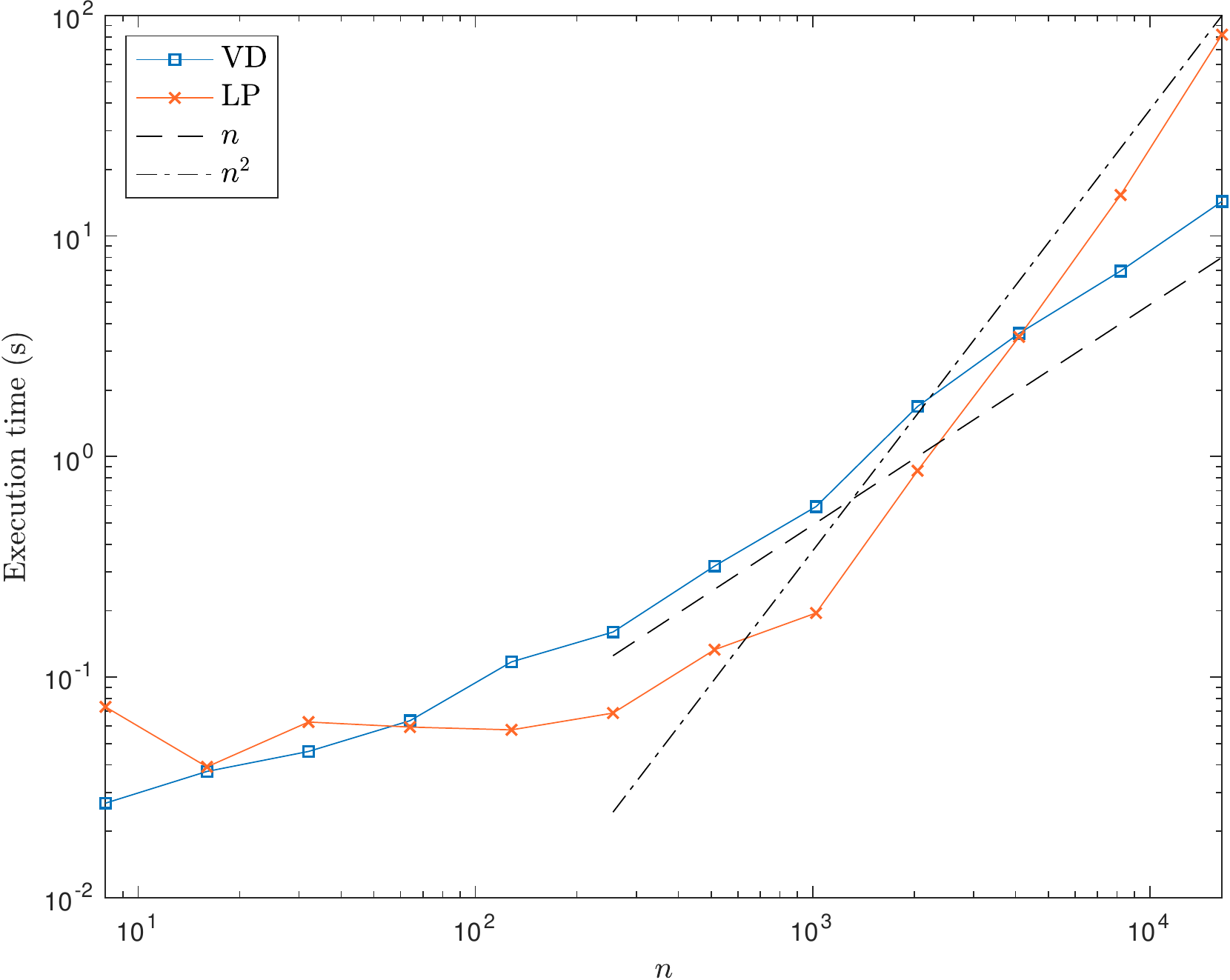}
    \caption{}
  \end{subfigure}
  \caption{Smooth traveling wave for the 2CH system. Errors in $\Ltwo$- and $\Hone$-norms for $u_n$ and $\rho_n$ after one period (\textsc{a}), and execution times for \texttt{ode45} in seconds (\textsc{b}). The schemes tested are VD \eqref{eq:VD_h} and HR \eqref{eq:fd2CH} for step sizes $2^{-k} \period$ where $3 \le k \le 14$,
    and evaluated on a reference grid with step size $2^{-16}\period$.}
  \label{fig:smooth2CH_results}
\end{figure}
We observe that both schemes exhibit highly consistent rates, but the difference scheme \eqref{eq:fd2CH} 
outperforms \eqref{eq:VD_H} convergence-wise by having smallest errors overall and rate 2 convergence in $\Ltwo$-norm.
For the run times we see that for the largest $n = 2^k$, namely $k > 12$, the difference scheme is considerably
slower than the variational scheme for these tolerances.

\subsection{Example 2: Periodic peakon}
It is now well known, cf.\ \cite{Camassa1993}, that a single peakon
\begin{equation*}
u(t,x) = c e^{-\abs{x-x_0-ct}}
\end{equation*}
is a weak solution of \eqref{eq:CH} with its peak at $x = x_0 + c t$.
The periodic counterpart of this solution is
\begin{equation*}
u(t,x) = c \frac{\cosh\left(\abs{x-x_0 -ct}-\frac{\period}{2}\right)}{\cosh\left(\frac{\period}{2}\right)}
\end{equation*}
valid for $\abs{x-x_0-ct} \le \period$, and periodically extended outside this interval.
This formula for the periodic peakon can in fact be deduced from \eqref{eq:cmp_per} for $n = 1$, or found in, e.g., \cite[Eq.\ (8.5)]{Lenells2005}.
Setting $x_0 = \frac12$, $c = \period = 1$, and $t = 0$ we use this function as initial datum on $[0,1]$
for a numerical example.
As the periodic multipeakon scheme reduces to exactly this peakon for $n = 1$, $y_1(0) = \frac12$ and
$u_1(0) = 1$, we have chosen to omit this scheme for the experiment, and rather compare how well the other
schemes approximate a peakon solution.

As one could expect, the schemes generally performed worse for this problem compared to the smooth traveling
wave, and so we could reduce the
tolerances for the time integrator to $\texttt{AbsTol} = \texttt{RelTol} = 10^{-8}$ with no change in
leading digits for the errors.
However, as the finite difference schemes, and especially the noncompact scheme \eqref{eq:fdCH_LP}, were quite slow
when using \texttt{ode45} for large values of $n$, we instead used the solver \texttt{ode113} which proved to be
somewhat faster in this case.
Moreover, when computing the approximate $\Hone$-error \eqref{eq:H1_approx} in this case, we encounter the problem
of the reference solution derivative not being defined at the peak.
To circumvent this issue, we measure the error at time $t = \period$ on a shifted reference grid.
That is, we evaluate \eqref{eq:H1_approx} on $x_i = (i + \tfrac12) 2^{-k_0}$ instead of $x_i = i 2^{-k_0}$ for
$i \in \{0,\dots,2^{k_0}-1\}$ to ensure $x_i \neq \tfrac12$.
We plot the solutions again for the relatively small $n = 2^4$ to highlight differences between the schemes
in Figure \ref{fig:per_peakCH_plot}.
\begin{figure}
  \begin{subfigure}{0.495\textwidth}
    \includegraphics[width=1\textwidth]{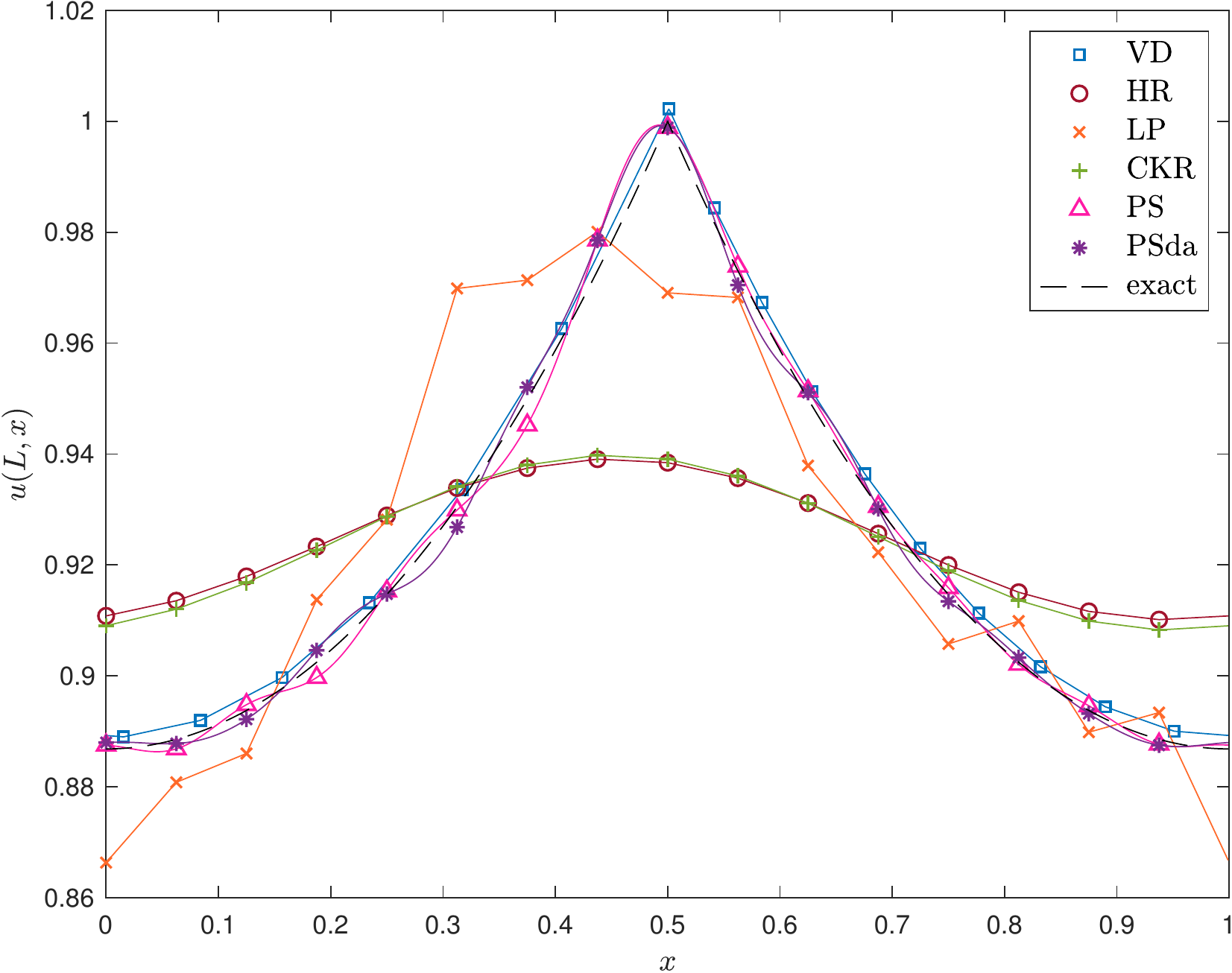}
  \end{subfigure}
  \begin{subfigure}{0.495\textwidth}
    \includegraphics[width=1\textwidth]{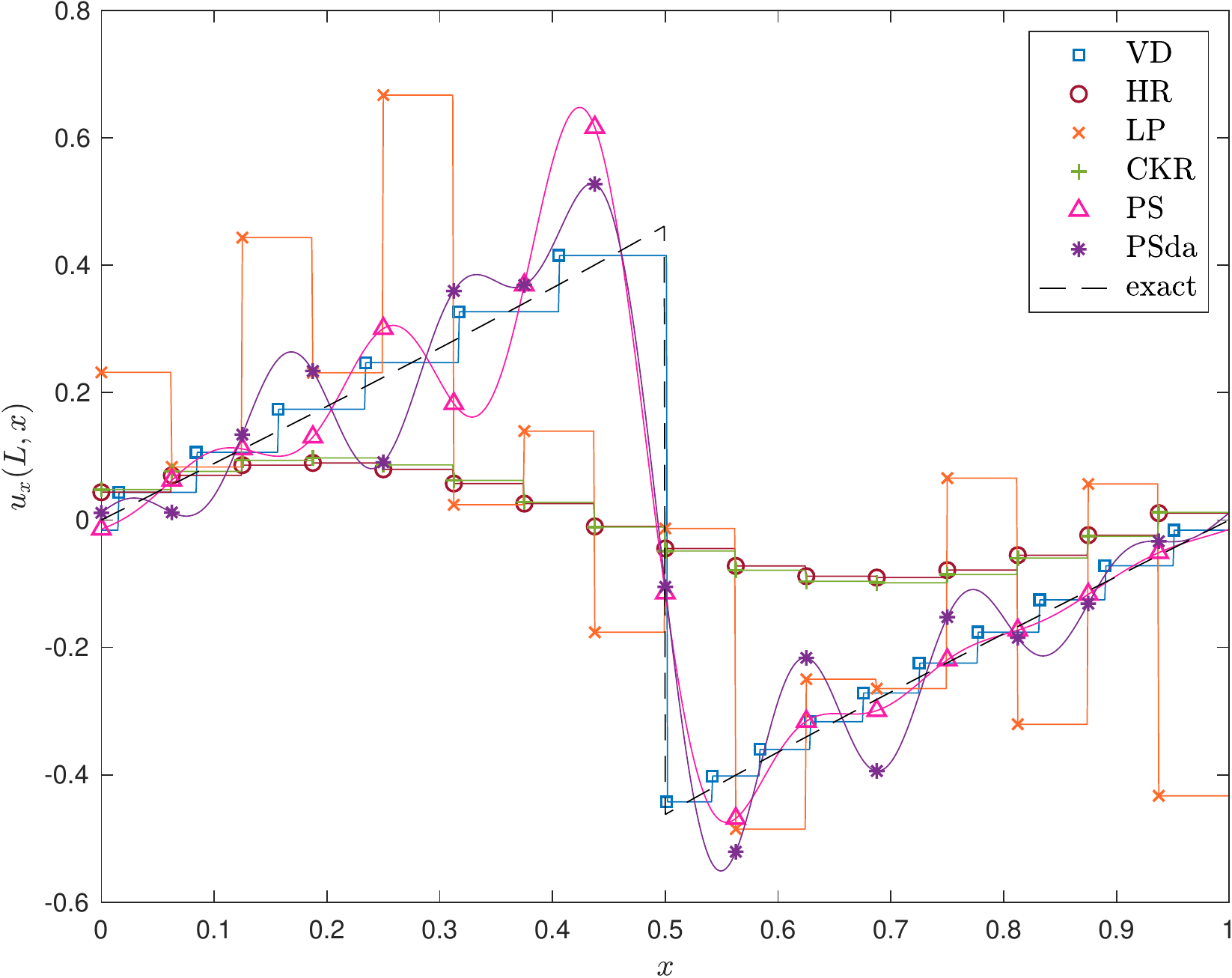}
  \end{subfigure}
  \caption{Periodic peakon. Plot of the interpolated solutions $u_n$ and $(u_n)_x$ at time $t = \period = 1$ for $n = 2^4$ and interpolated on a reference grid with $2^{10}$ grid points. The schemes considered are VD \eqref{eq:VD_H}, HR \eqref{eq:fdCH_HR}, LP \eqref{eq:fdCH_LP}, CKR \eqref{eq:fdCH_CKR}, while PS and PSda are \eqref{eq:PS} without and with dealiasing respectively.}
  \label{fig:per_peakCH_plot}
\end{figure}

Once more we observe that the dissipativity of the schemes \eqref{eq:fdCH_HR} and \eqref{eq:fdCH_LP} is quite
severe for this step size and they fail to capture the shape of the peakon.
The energy preserving difference scheme \eqref{eq:fdCH_LP} is closer to the
shape of the peakon, but exhibits oscillations which are particularly
prominent in the derivative.
On the other hand, the variational scheme manages to capture the shape of
the peakon very well, and manages far better than the other schemes to
capture the derivative of the reference solution after one period.
\begin{figure}
  \begin{subfigure}{0.495\textwidth}
    \includegraphics[width=1\textwidth]{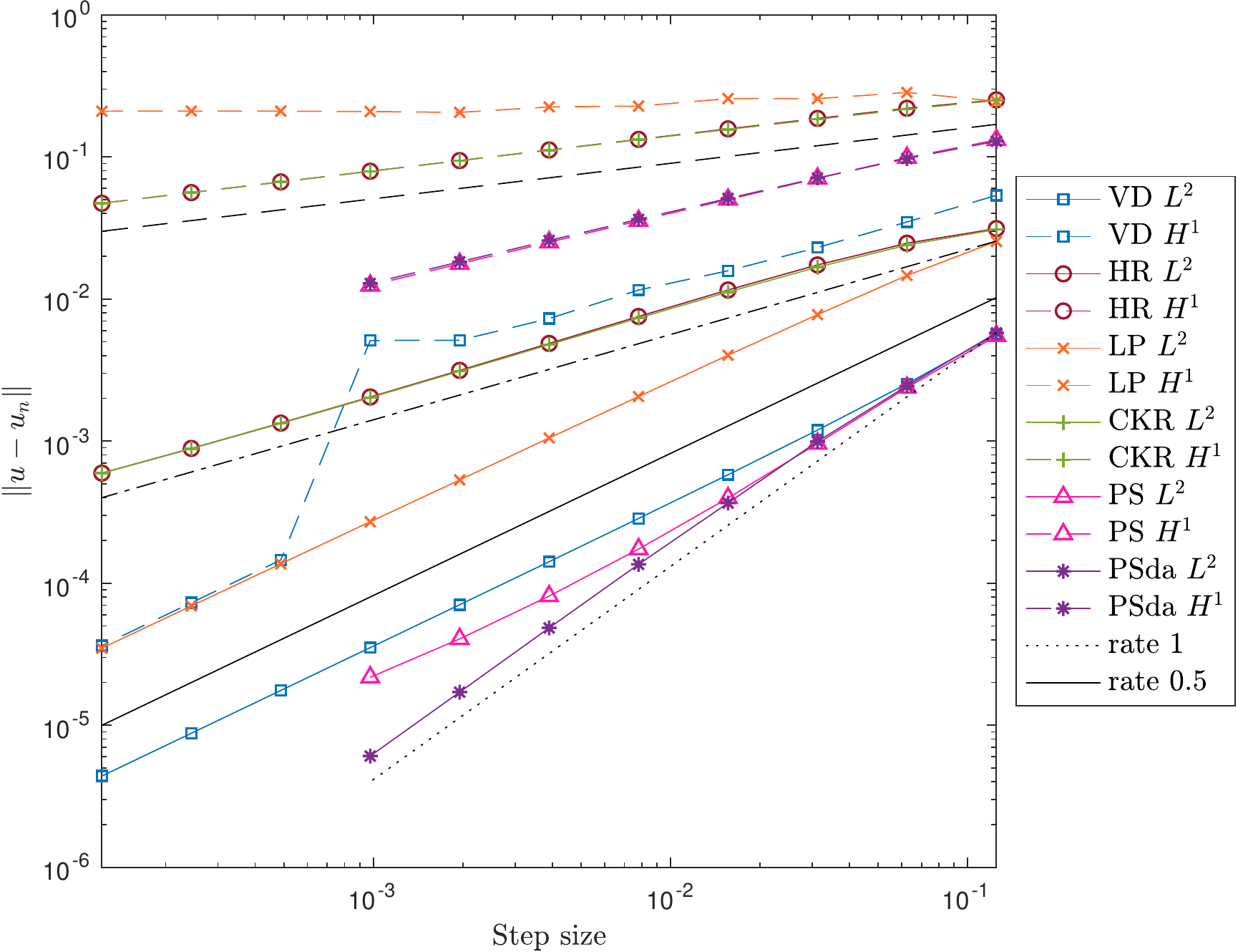}
    \caption{}\label{fig:per_peakCH_results_rates}
  \end{subfigure}
  \begin{subfigure}{0.495\textwidth}
    \includegraphics[width=1\textwidth]{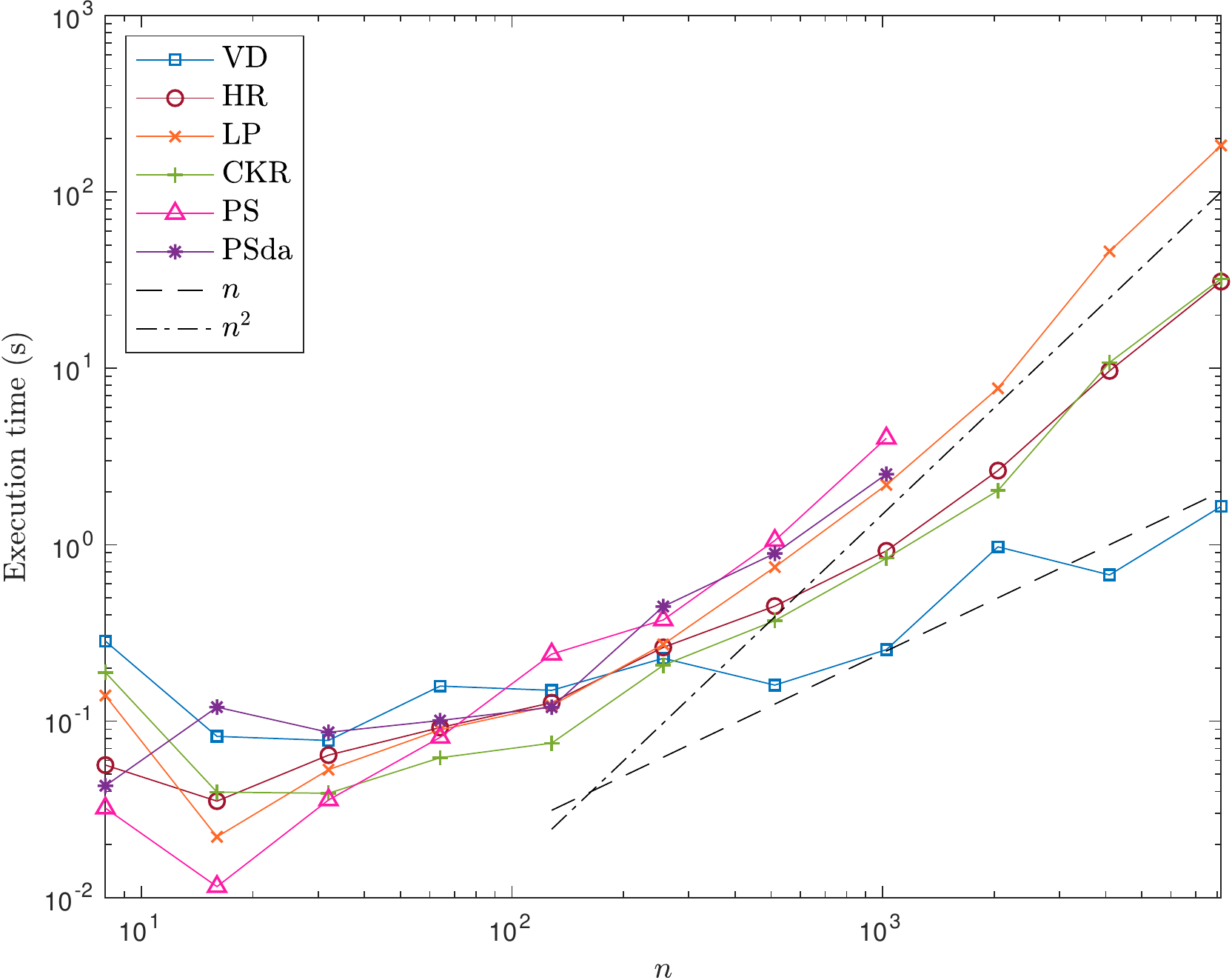}
    \caption{}\label{fig:per_peakCH_results_times}
  \end{subfigure}
  \caption{Periodic peakon. Errors in $\Ltwo$- and $\Hone$-norms after one period (\textsc{a}) and execution times for \texttt{ode113} in seconds (\textsc{b}). The schemes tested are VD \eqref{eq:VD_h}, HR \eqref{eq:fdCH_HR}, LP \eqref{eq:fdCH_LP} and CKR \eqref{eq:fdCH_CKR} for step sizes $2^{-k}$ where $3 \le k \le 13$, while \eqref{eq:PS} without and with dealiasing, PS and PSda, have $3 \le k \le 10$. All are evaluated on a reference grid with step size $2^{-15}$.}
  \label{fig:per_peakCH_results}
\end{figure}

The above observations are reflected in Figure \ref{fig:per_peakCH_results_rates}
which shows the rate of convergence.
The errors for the dissipative schemes decrease, which is expected since
both have been proven to converge in $\Hone$.
However, this convergence is quite slow, with approximate rates of 0.6 and
0.25 for the $\Ltwo$- and $\Hone$-norms respectively.

The energy-preserving difference scheme \eqref{eq:fdCH_LP} exhibits order 1 convergence in
$\Ltwo$-norm, but the oscillations in the derivative put an end to any hope of
$\Hone$-convergence.
Indeed, the $\Hone$-seminorm of the error is larger than 0.2 irrespective of the step
size.
The oscillations are of course even more severe for $m_i$, the
discrete version of $u-u_{xx}$ which is actually solved for in the ODE.

The variational scheme \eqref{eq:VD_H} performs quite well, with convergence
rate 1 in $\Ltwo$ and $\Hone$-rates generally between 0.45 and 1.
The exception is the transition from $n = 2^9$ to $n = 2^{10}$  where there
was barely any decrease in the error, followed by a large decrease
corresponding to a rate of 5 in $n = 2^{11}$, and from here the
$\Hone$-error is comparable in magnitude to the $\Ltwo$-error of
\eqref{eq:fdCH_LP}.
This jump is possibly connected to the discontinuity of the reference solution.

The pseudospectral scheme performs quite well for the $\Ltwo$-norm, and for larger $n$ it has the smallest error of all the
methods, with the dealiased scheme showing a better convergence rate which approaches 1.5.
However, in $\Hone$-norm the scheme performs worse than the variational scheme, owing to the major oscillations
close to the discontinuity in the reference solution.
Note that we have only run this scheme for $k \in \{3,\dots,10\}$, as opposed to $k \in \{3,\dots,13\}$ for
the other methods.
The reason for this is that for larger $n$ this scheme needs a finer reference grid to have consistent
convergence rates, as opposed to the other schemes, and in addition the run times are very long for larger $n$.

\begin{rem}\label{rem:cluster}
  Here we have only run the schemes over one period for the traveling wave, but an additional issue for the
  schemes in Lagrangian coordinates becomes apparent if they are run for a long time with initial data
  containing a derivative discontinuity, such as the traveling peakon.
  Then one typically observes a clustering of characteristics, or particles, at the front of the
  traveling discontinuity, leaving less particles to resolve the rest of the wave profile.
  Indeed, this is also reported in the numerical results of \cite{CamHuaLee2005} for their particle method,
  and the authors suggest that a redistribution algorithm may be applied when particles come too close.
  Such redistribution algorithms would be useful for \eqref{eq:cmp_per} and \eqref{eq:VD_H} when running them
  for long times, but development of such tools fall outside the scope of this paper.
  It is however important to be aware of this phenomenon, as the clustering can lead to artificial numerical
  collisions of the characteristics when they become too close for the computer to distinguish them.
  In worst case this can lead to a breakdown of the initial ordering of the characteristics $y_i$,
  which again ruins the structure of the ODE system, leading to wrong solutions or breakdown of the method.
\end{rem}

\subsection{Example 3: Peakon-antipeakon example}
In this example we consider the interval $[0,\period]$ and peakon-antipeakon initial datum
\begin{equation*}
u_0(x) = \frac{c}{\sinh\left(\frac{\period}{4}\right)} \begin{cases}
\sinh(x), & 0 \le x < \frac{\period}{4}, \\
\sinh\left(\frac{\period}{2} - x\right), & \frac{\period}{4} \le x < \frac{3\period}{4}, \\
\sinh(x - \period), & \frac{3\period}{4} \le x < \period,
\end{cases}
\end{equation*}
for $c = 1$ and $\period = 2\pi$.
We want to evaluate the numerical solutions at $t = 4.5$, which is approximately the time when the two peaks
have returned to their initial positions $x = \pi/2$ and $x = 3\pi/2$ after colliding once.
Since this is a multipeakon solution, we may use the conservative multipeakon scheme \eqref{eq:cmp_per} to
provide a reference solution.
Setting $n = 2$, choosing $y_1 = \pi/2$, $y_2 = 3\pi/2$, $u_1 = -u_2 = 1$, and computing $H_i$ for $i = 1,2$
according to \eqref{eq:delH}, we integrated in time using \texttt{ode113} with the very stringent tolerances
$\texttt{AbsTol} = \texttt{eps}$ and $\texttt{RelTol} = 100\:\texttt{eps}$.
For the schemes in the comparison we used the same solver with $\texttt{AbsTol} = \texttt{RelTol} = 10^{-9}$, and a reference
grid with $2^{16}$ equispaced points.

For this example we have omitted the dissipative schemes \eqref{eq:fdCH_HR} and \eqref{eq:fdCH_CKR}, as the former
cannot handle initial data of this type, while the latter would produce an approximation of the dissipative
solution which is identically zero after the collision.
Figure \ref{fig:per_peakCH_plot} shows $u_n$ for the variational scheme \eqref{eq:VD_H}, the finite
difference scheme \eqref{eq:fdCH_LP}, and the pseudospectral scheme \eqref{eq:PS} with and without dealiasing, all
for $n = 2^6$.
\begin{figure}
  \includegraphics[width=0.8\textwidth]{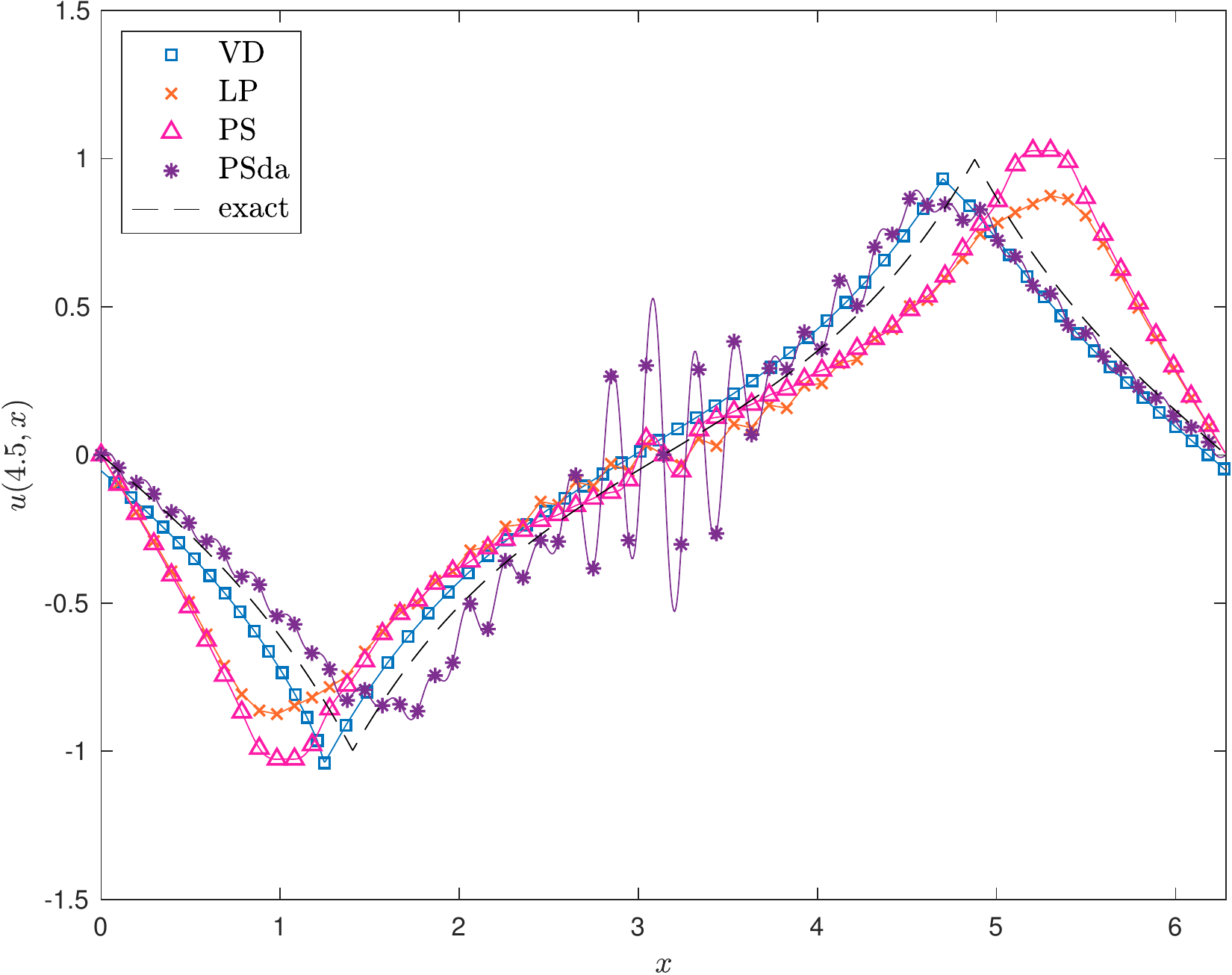}
  \caption{Peakon-antipeakon example. Interpolants $u_n$ for the schemes VD \eqref{eq:VD_H}, LP \eqref{eq:fdCH_LP}, PS and PSda which are respectively \eqref{eq:PS} without and with dealiasing. Here $n = 2^6$ with $2^{10}$ reference grid points, and the reference solution is computed using \eqref{eq:cmp_per}.}
  \label{fig:peak_antiCH_plot_N64}
\end{figure}

This is a numerical example which is especially ill-suited for the pseudospectral method \eqref{eq:PS}, and
illustrates the necessity of the dealiasing to have any form of convergence.
From Figure \ref{fig:peak_antiCH_plot_N64} one could get the impression that the dealiased scheme will perform
worse because of its large-amplitude oscillations near the point of collision.
A possible explanation for this is that the peakon-antipeakon interaction is very localized at collision time,
and in removing the high frequency components of the pseudospectral approximation one loses the only basis
functions which are able to resolve these localized details, and we are left with oscillations caused by the
remaining basis functions.
However, the amplitude of these spurious oscillations will decrease as $n$ increases, since we have more basis
functions with high frequencies.
On the other hand, the phase error in \eqref{eq:PS} without dealiasing does not decrease as $n$ increases.
Indeed, after collision this approximation always attains larger amplitudes than the reference solution, and thus
travels further, irrespective of $n$.
This becomes apparent in Figure \ref{fig:peak_antiCH_results_rates} where we see that there is no convergence at
all for the method without dealiasing.
\begin{figure}
  \begin{subfigure}{0.495\textwidth}
    \includegraphics[width=1\textwidth]{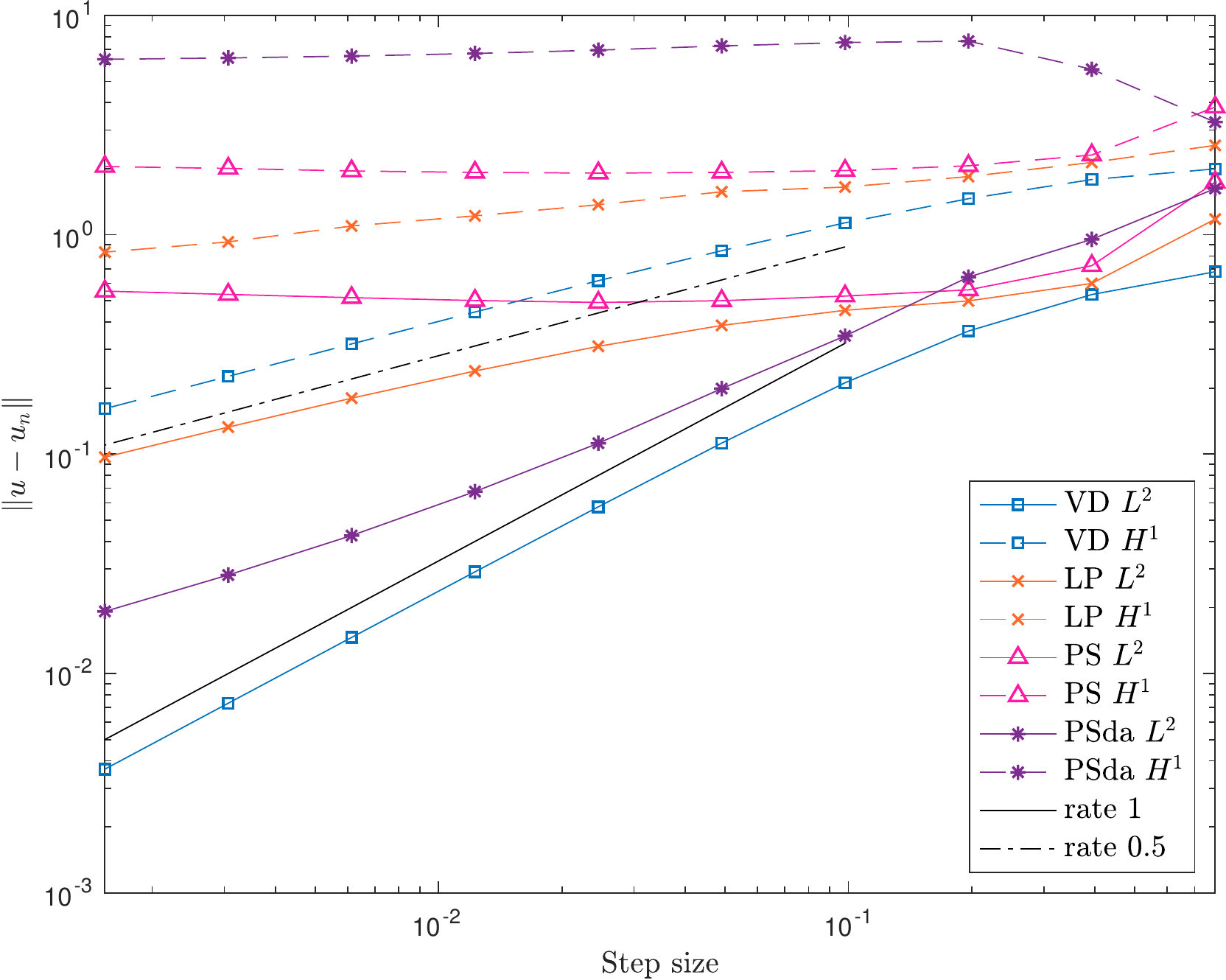}
    \caption{} \label{fig:peak_antiCH_results_rates}
  \end{subfigure}
  \begin{subfigure}{0.495\textwidth}
    \includegraphics[width=1\textwidth]{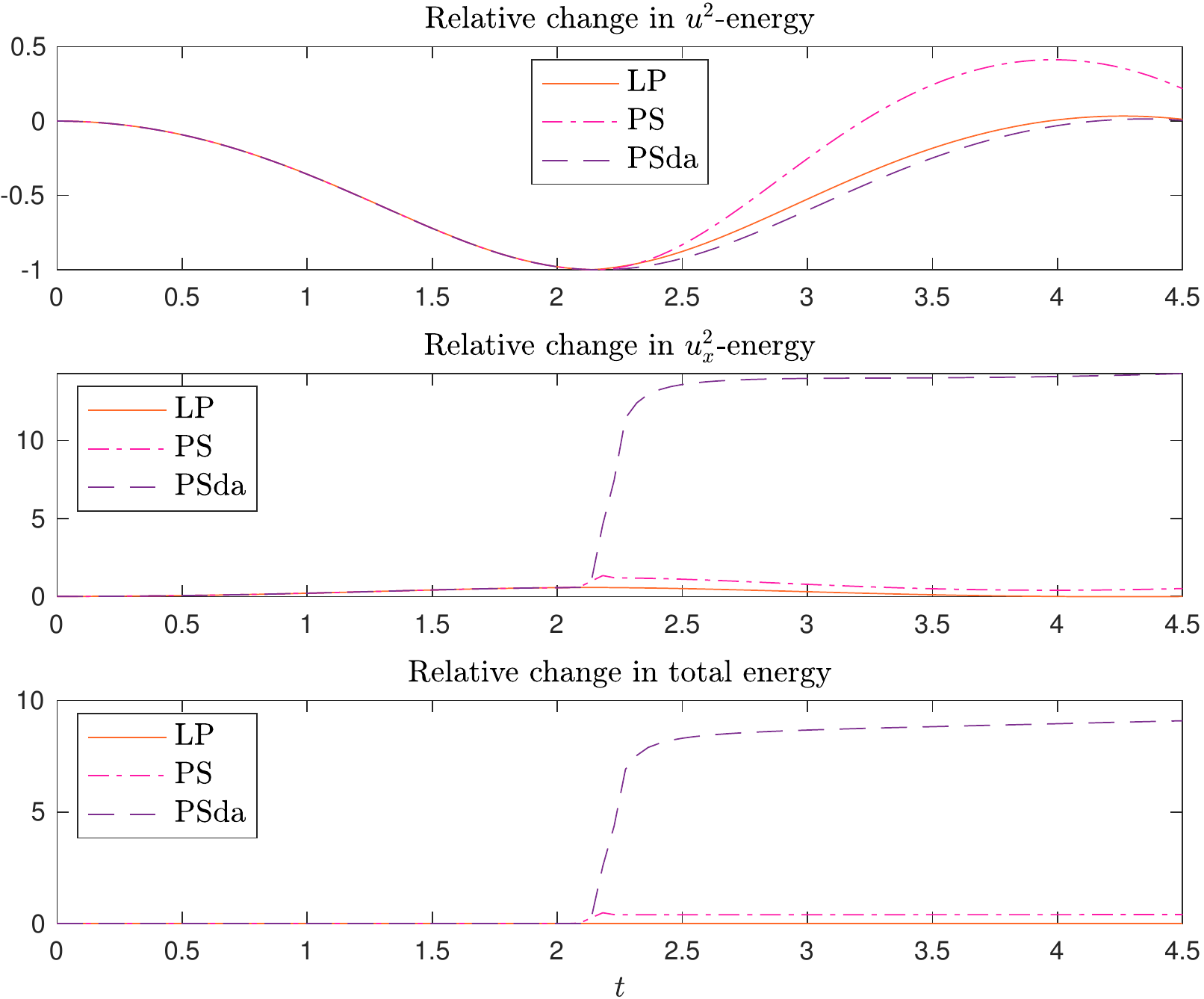}
    \caption{} \label{fig:peak_antiCH_results_energy}
  \end{subfigure}
  \caption{Peakon-antipeakon example. Errors in $\Ltwo$- and $\Hone$-norms (\textsc{a}) for the schemes displayed in Figure \ref{fig:peak_antiCH_plot_N64} with $n = 2^k$, $ 3 \le k \le 12$, and relative change in energies (\textsc{b}) for the schemes \eqref{eq:fdCH_LP} and \eqref{eq:PS} for $n = 2^{12}$.}
  \label{fig:peak_antiCH_results}
\end{figure}
To try and understand this behavior we turned to the energy of the solution, and even though there is no
defined discrete energy in the derivation of this method, based on the discrete energy in \eqref{eq:invariants}
we used the following expression as an indicator of the energy,
\begin{equation*}
\frac12 \dx \sum_{i = 0}^{n-1} \left( (u_n(x_i))^2 + ((u_n)_x(x_i))^2 \right).
\end{equation*}
In Figure \ref{fig:peak_antiCH_results_energy} we have plotted the relative change $(E(t) - E(0))/E(0)$ in the
discrete energies of \eqref{eq:fdCH_LP} and \eqref{eq:PS} over time for $n = 2^{12}$.
Note that we have also added plots of the relative change in the $u^2$- and $u^2_x$-parts of the energies.
From these plots we observe that all three schemes behave similarly until collision time, after which the total
energy of both version of the pseudospectral scheme increases, while the relative change in total
energy for \eqref{eq:fdCH_LP} remains of order $10^{-11}$.
We observe that the $u_x^2$-energy, and thus also the total energy, increases far more for the dealiased scheme,
which is probably caused by the oscillations near the collision point discussed earlier.
On the other hand, when considering the $u^2$-energy, the scheme without dealiasing has a larger increase than
the dealiased one which has a $u^2$-energy closer to that of the finite difference scheme \eqref{eq:fdCH_LP}.
This suggests that unless one applies dealiasing to the pseudospectral scheme \eqref{eq:PS},
the peakon-antipeakon collision introduces an artificial increase in energy which ruins $\Ltwo$-convergence
through a phase error, see Figure \ref{fig:peak_antiCH_results_rates}.
For the $\Hone$-norm one cannot expect convergence from any version of \eqref{eq:PS}, as the collision introduces severe
oscillations in the pseudospectral derivative.

Oscillatory behavior also explains the very slow decrease in $\Hone$-error for the invariant-preserving
difference scheme \eqref{eq:fdCH_LP}, displaying a rate fluctuating around 0.15.
For the $\Ltwo$-norm it exhibits a rate which approaches 0.5.

Meanwhile, the variational scheme \eqref{eq:VD_H} performs rather well for this example, having the smallest errors
in both $\Ltwo$- and $\Hone$-norm, and displaying consistent rates of respectively 1 and 0.5.

\subsection{Example 4: Collision-time initial datum}
An interesting feature discussed in \cite[Section 5.2]{vardisc} is that the variational discretization allows for
irregular initial data.
That is, pairs $(u, \mu)$, where $\mu$ may be a positive finite Radon measure, provide a complete description of the initial
data and the corresponding solution of \eqref{eq:CH} in Eulerian coordinates.
In particular, for the absolutely continuous part of $\mu$ one has
\begin{equation*}
\mu_\text{ac}((-\infty,x)) = \frac12 \int_{-\infty}^{x}\left(u^2(x') + u_x^2(x') \right) dx',
\end{equation*}
and the cumulative energy $\mu((-\infty,x))$ can be a step function, which is connected to the well-studied peakon-antipeakon
dynamics.

For example, at collision time, $u$ may be identically zero and all energy is concentrated in the point of collision as a
delta distribution, meaning the cumulative energy will be a step function centered at the collision.
To be able to accurately represent the solution between the two peakons emerging from a collision, we have to
``pack'' sufficiently many characteristics into the collision point.
To this end we introduce the initial characteristics $y_0(\xi) = y(0,\xi)$ and the initial cumulative energy
$H_0(\xi) = H(0,\xi)$ similar to \cite[Eq.\ (3.20)]{HolRay2008},
\begin{equation}\label{eq:yH_para}
y_0(\xi) \coloneqq \sup\left\{ y \:|\: y + F_\mu(y) < (1 + E/\period)\xi \right\}, \quad
H_0(\xi) = (1 + E/\period)\xi - y_0(\xi),
\end{equation}
where $F_\mu(x) = \mu([0,x))$ for $x \in [0,\period]$ and $E = \mu([0,\period))$ is the total energy of the system. 

This feature inspired the following variation of peakon-antipeakon initial data,
where we initially have a system with period $\period = 8$ and total energy $E = 6$ equally concentrated in the points $x = 2$
and $x = 6$ on the interval $[0,8]$.
In Eulerian variables this reads
\begin{equation*}
u_0(x) \equiv 0, \qquad \mu_0([0,x)) = \begin{cases}
0, & 0 \le x \le 2, \\ 3, & 2 < x \le 6, \\ 6, & 6 < x \le 8.
\end{cases}
\end{equation*}
On the other hand, for the Lagrangian description we use \eqref{eq:yH_para} to compute
\begin{equation*}
y_0(\xi) = \begin{cases}
\frac74 \xi, & 0 \le \xi < \frac87, \\
2, & \frac87 \le \xi < \frac{20}{7}, \\
\frac74 \xi - 3, & \frac{20}{7} \le \xi < \frac{36}{7}, \\
6, & \frac{36}{7} \le \xi < \frac{48}{7}, \\
\frac74 \xi - 6, & \frac{48}{7} \le \xi < 8,
\end{cases} \quad
H_0(\xi) = \begin{cases}
0, & 0 \le \xi < \frac87, \\
\frac74 \xi - 2, & \frac87 \le \xi < \frac{20}{7}, \\
3, & \frac{20}{7} \le \xi < \frac{36}{7}, \\
\frac74 \xi - 6, & \frac{36}{7} \le \xi < \frac{48}{7}, \\
6, & \frac{48}{7} \le \xi < 8,
\end{cases}
\end{equation*}
together with $U_0(\xi) \equiv 0$.
From this we define the discrete initial data for \eqref{eq:VD_H} by $(y_0)_i = y_0(\xi_i)$,
$(H_0)_i = H_0(\xi_i)$, and $(U_0)_i = 0$ in the grid points $\xi_i = i 2^{-k} \period$ for
$k \in \{3,\dots,14\}$.
Using \texttt{ode45} with $\texttt{AbsTol} = \texttt{RelTol} = 10^{-8}$ we integrate from $t = 0$ to $t = 4$.

As the conservative multipeakon method describes exactly the interaction of peakons,
we may once more use it as reference solution.
Setting $n = 4$, $\period = 8$ in \eqref{eq:cmp_per} we define initial data
\begin{equation*}
\mathbf{y}_0 = \begin{bmatrix} 2 & 2 & 6 & 6 \end{bmatrix}^\tp, \quad
\mathbf{U}_0 = \begin{bmatrix} 0 & 0 & 0 & 0 \end{bmatrix}^\tp, \quad
\mathbf{H}_0 = \begin{bmatrix} 0 & 6 & 6 & 12 \end{bmatrix}^\tp
\end{equation*}
corresponding to two pairs of peakons respectively placed at $x = 2$ and $x = 6$ with energy $6$ contained between
the peakons in each pair.
Note that the energy is double that of the energy prescribed for the variational scheme \eqref{eq:VD_H},
since the factor $\frac12$ is not present in the definition of the energy for the multipeakon scheme.
This was then integrated using \texttt{ode45} with the same tolerances as for the reference solution in the previous example,
$\texttt{AbsTol} = \texttt{eps}$ and $\texttt{RelTol} = 100\: \texttt{eps}$.

Then we measured the errors using \eqref{eq:H1_approx} on the reference grid $x_i = 2^{-16} \period$, and the
results are displayed in Figure \ref{fig:step_initCH_rates}.
We found the decrease in error to be remarkably consistent, rate 1 in the $\Ltwo$-norm and approximately 0.5 for
the $\Hone$-norm, and this is true for both the time $t = 2$ before the collision and time $t = 4$
after the collision.
Figure \ref{fig:step_initCH_chars} displays the characteristics for the solution with $n = 2^6$ together with the four
trajectories of the peaks of the exact solution. Observe how the characteristics, initially clustered
at the collision points $x = 2$ and $x = 6$ in accordance with \eqref{eq:yH_para}, spread out between the pairs of peaks in the reference solution.

In Figure \ref{fig:step_initCH_u} we have plotted the solution for $n = 2^6$ and interpolated on a reference grid with
$2^{10}$ grid points. We observe that the interpolants match the shape of the exact solution quite well, even for the derivative, and have only a slight phase error. 

\begin{figure}
  \begin{subfigure}{0.495\textwidth}
    \includegraphics[width=1\textwidth]{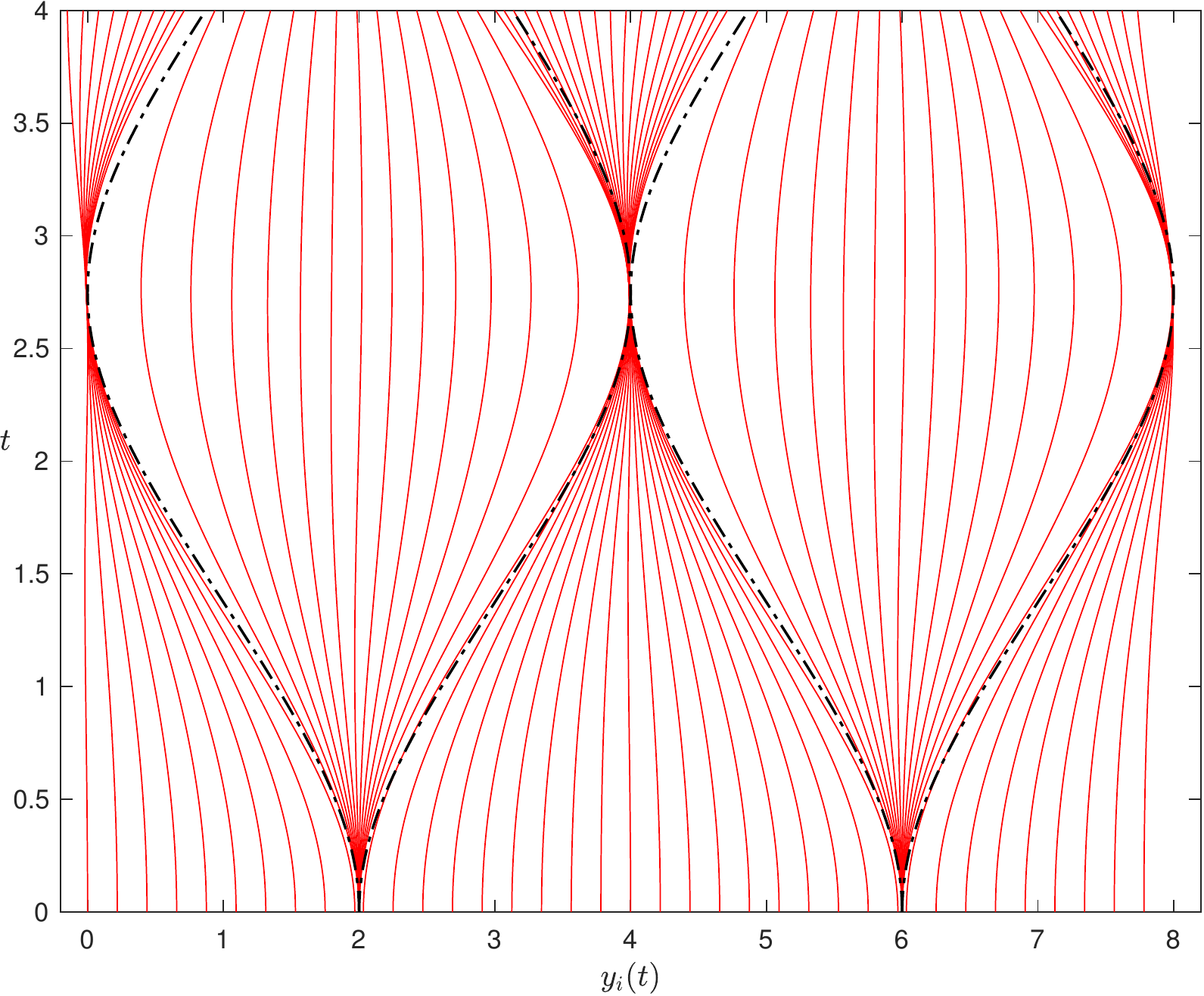}
    \caption{}\label{fig:step_initCH_chars}
  \end{subfigure}
  \begin{subfigure}{0.495\textwidth}
    \includegraphics[width=1\textwidth]{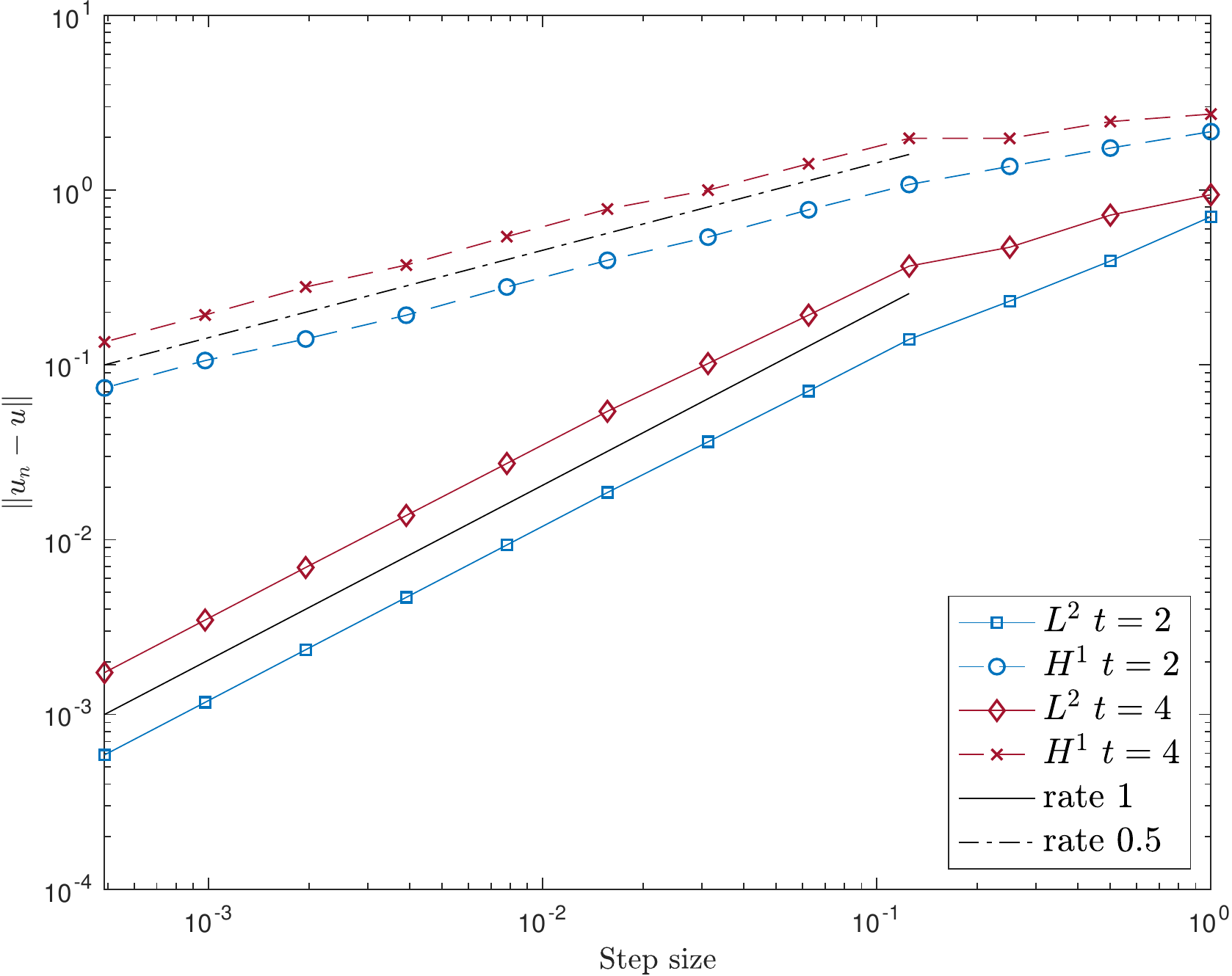}
    \caption{}\label{fig:step_initCH_rates}
  \end{subfigure}
  \caption{Collision-time initial datum. (\textsc{a}): The $n = 2^6$ characteristics for the variational scheme (solid red) and for the four reference peakons (dash-dotted black). (\textsc{b}): Error rates of the variational scheme at times $t = 2$ and $t = 4$ for $n = 2^k$ and $3 \le k \le 14$ evaluated on a $2^{16}$ point reference grid. }
\end{figure}

\begin{figure}
  \begin{subfigure}{0.495\textwidth}
    \includegraphics[width=\textwidth]{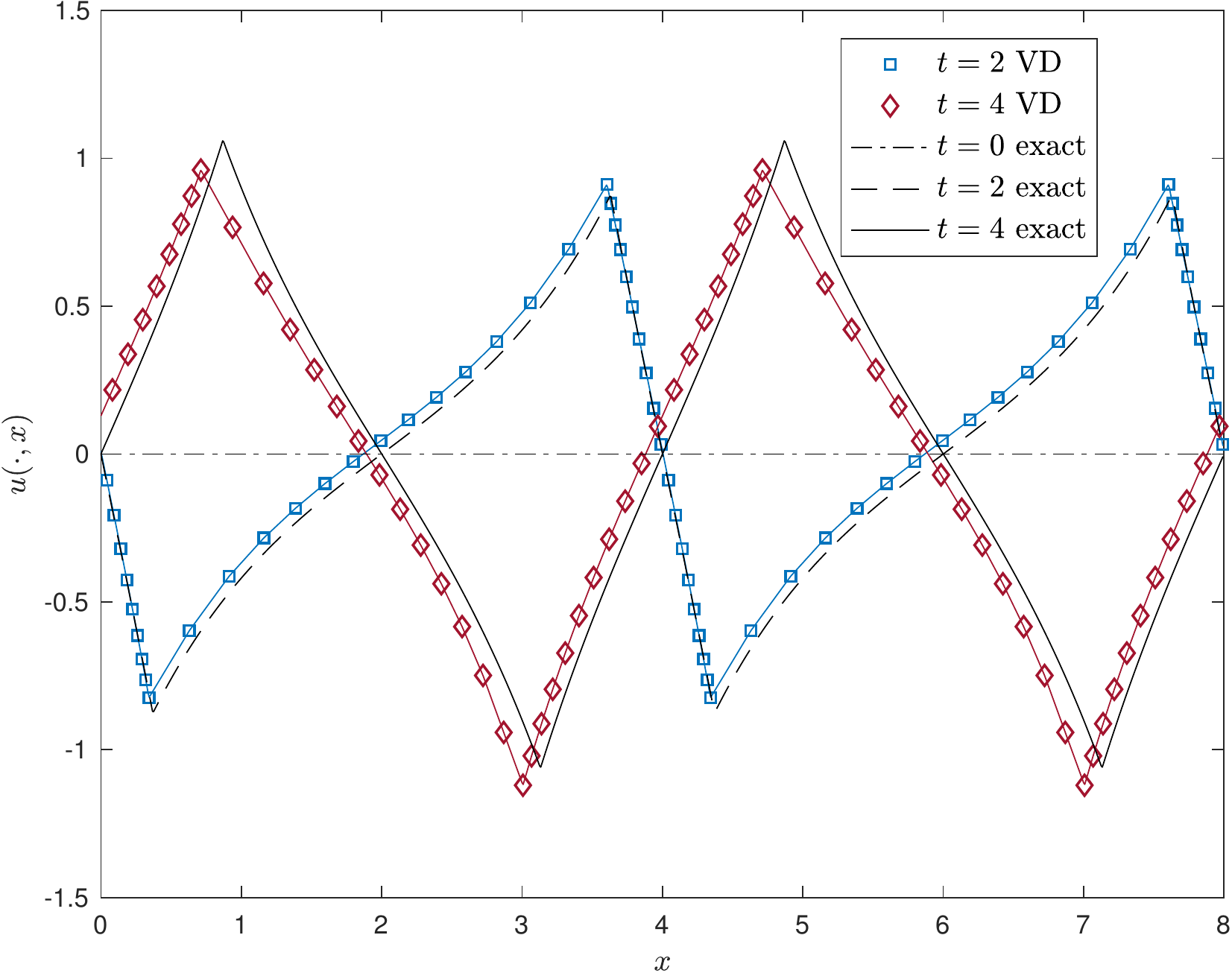}
  \end{subfigure}
  \begin{subfigure}{0.495\textwidth}
    \includegraphics[width=\textwidth]{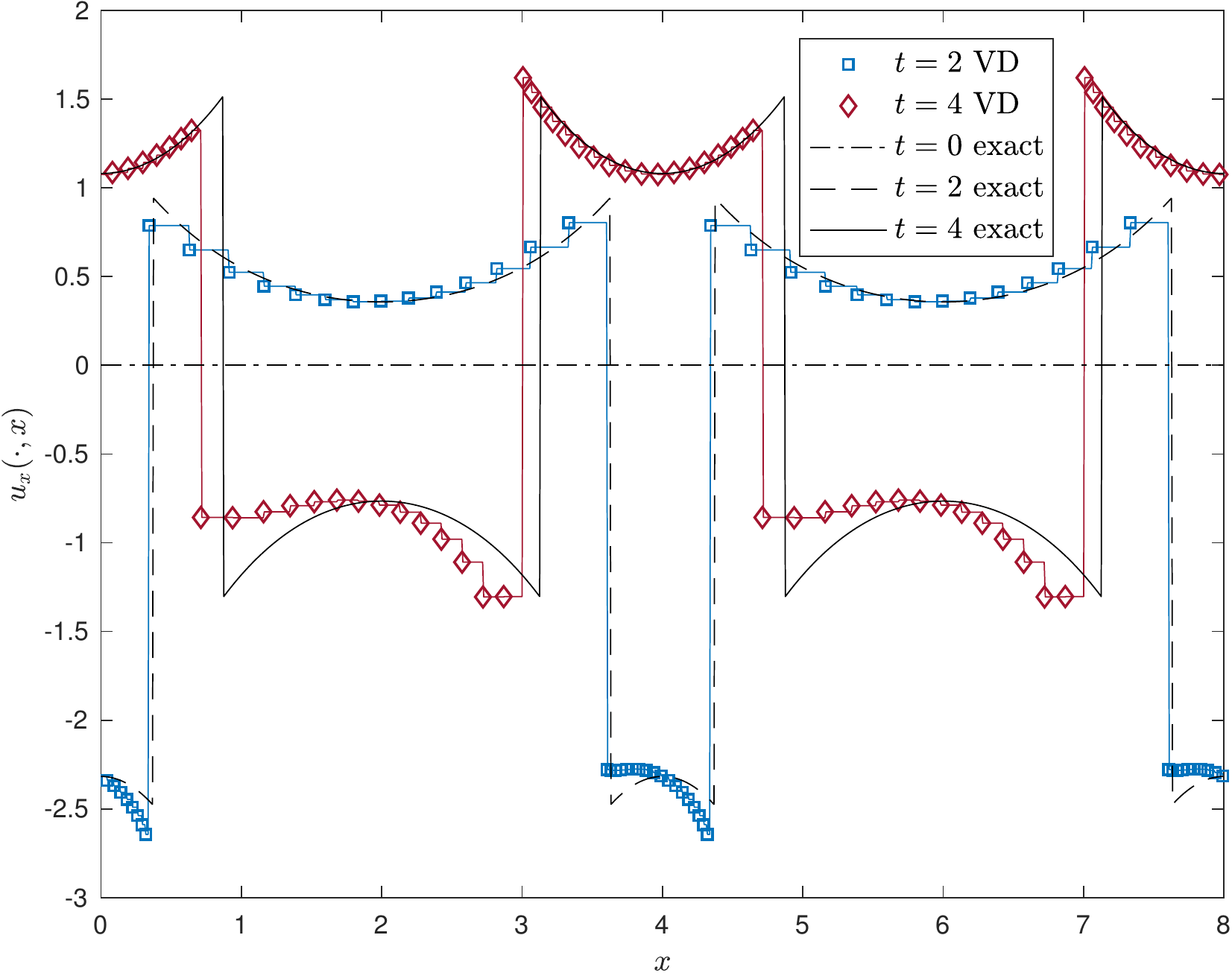}
  \end{subfigure}
  \caption{Plot of the interpolated solutions $u_n$ and $(u_n)_x$ for collision-time initial datum at
    times $t = 2$ and $t = 4$ with  $n = 2^6$ interpolated on a reference grid with $2^{10}$ grid points. }
  \label{fig:step_initCH_u}
\end{figure}

\subsection{Example 5: Sine initial datum for the CH equation}
In the following example we will qualitatively compare how the variational scheme \eqref{eq:VD_h} and the
conservative multipeakon scheme \eqref{eq:cmp_per} handle smooth initial data which leads to wave breaking.
We have chosen to consider $u_0(x) = \sin(x)$ for $x \in [0,2\pi]$, since this is a simple, smooth periodic
function which leads to singularity formation.
Furthermore, it is antisymmetric around the point $x = \pi$, which will highlight another difference between the
methods.
Since we do not have a reference solution in this case, the comparison will be of a more qualitative nature than
in the preceding examples.

As usual we chose $y_i(0) = \xi_i = i 2 \pi/n$ and $U_i(0) = u_0(\xi_i)$ for both schemes, and computed their
corresponding initial cumulative energies $H_i(0)$ in their own respective ways.
Then we have integrated from $t = 0$ to $t = 6\pi$ using \texttt{ode45} with
$\texttt{AbsTol} = \texttt{RelTol} = 10^{-10}$, and evaluated the interpolated functions on a finer grid with step
size $\dx = 2^{-10}\period$.

\begin{figure}
  \begin{subfigure}{0.495\textwidth}
    \includegraphics[width=1\textwidth]{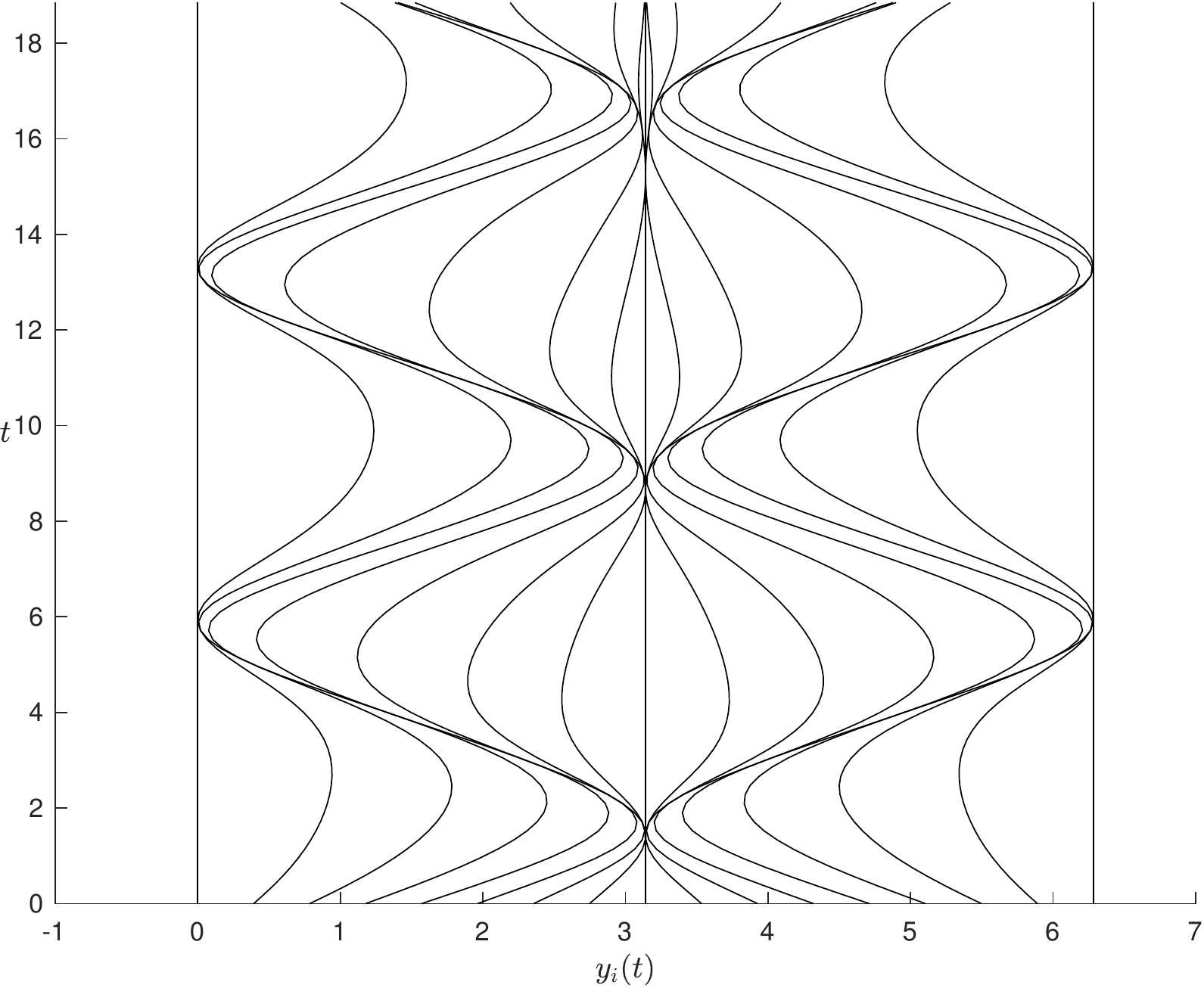}
    \subcaption{} \label{fig:sineCH_CMP16}
  \end{subfigure}
  \begin{subfigure}{0.495\textwidth}
    \includegraphics[width=1\textwidth]{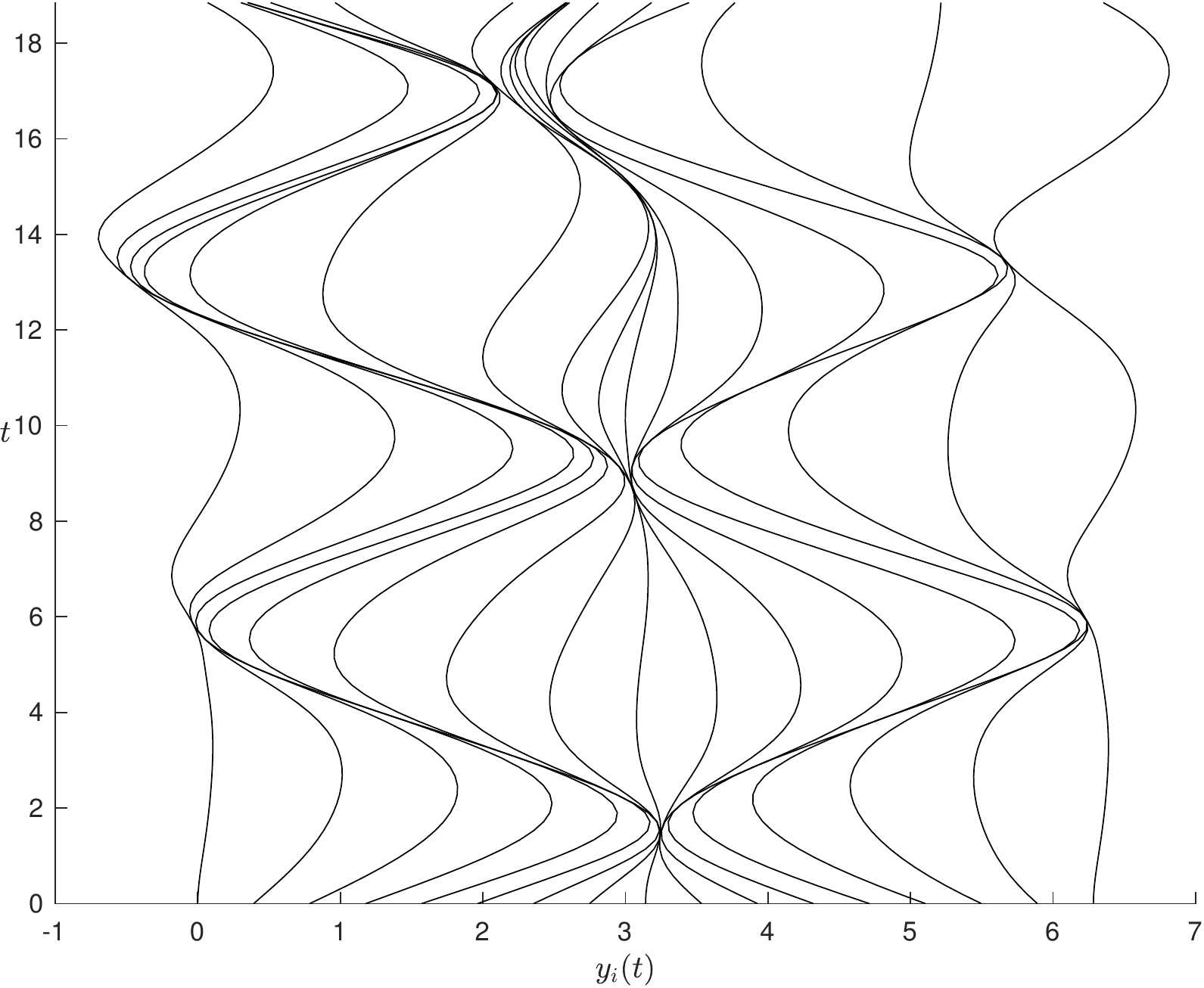}
    \subcaption{} \label{fig:sineCH_VD16}
  \end{subfigure}
  \begin{subfigure}{0.495\textwidth}
    \includegraphics[width=1\textwidth]{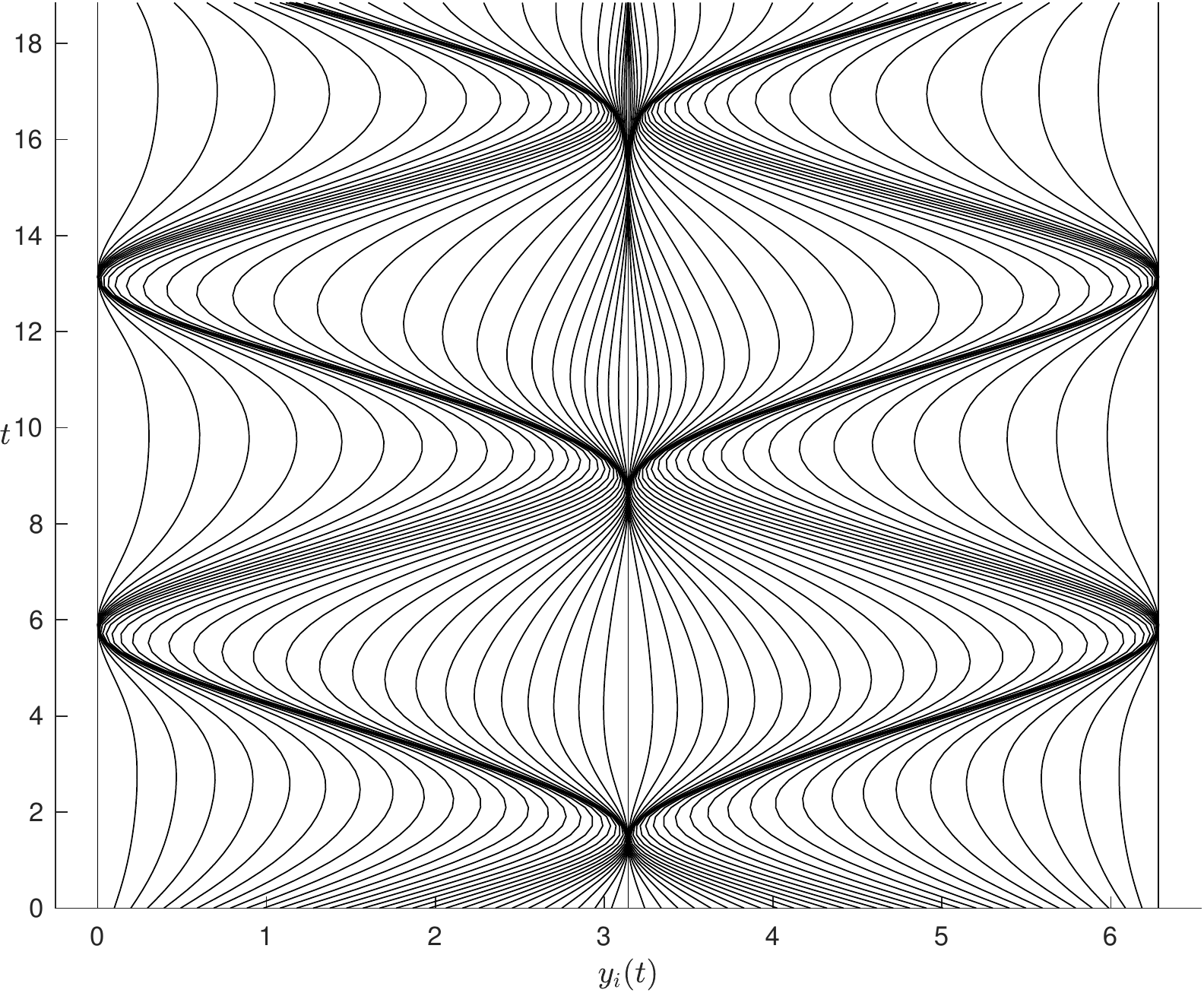}
    \subcaption{} \label{fig:sineCH_CMP64}
  \end{subfigure}
  \begin{subfigure}{0.495\textwidth}
    \includegraphics[width=1\textwidth]{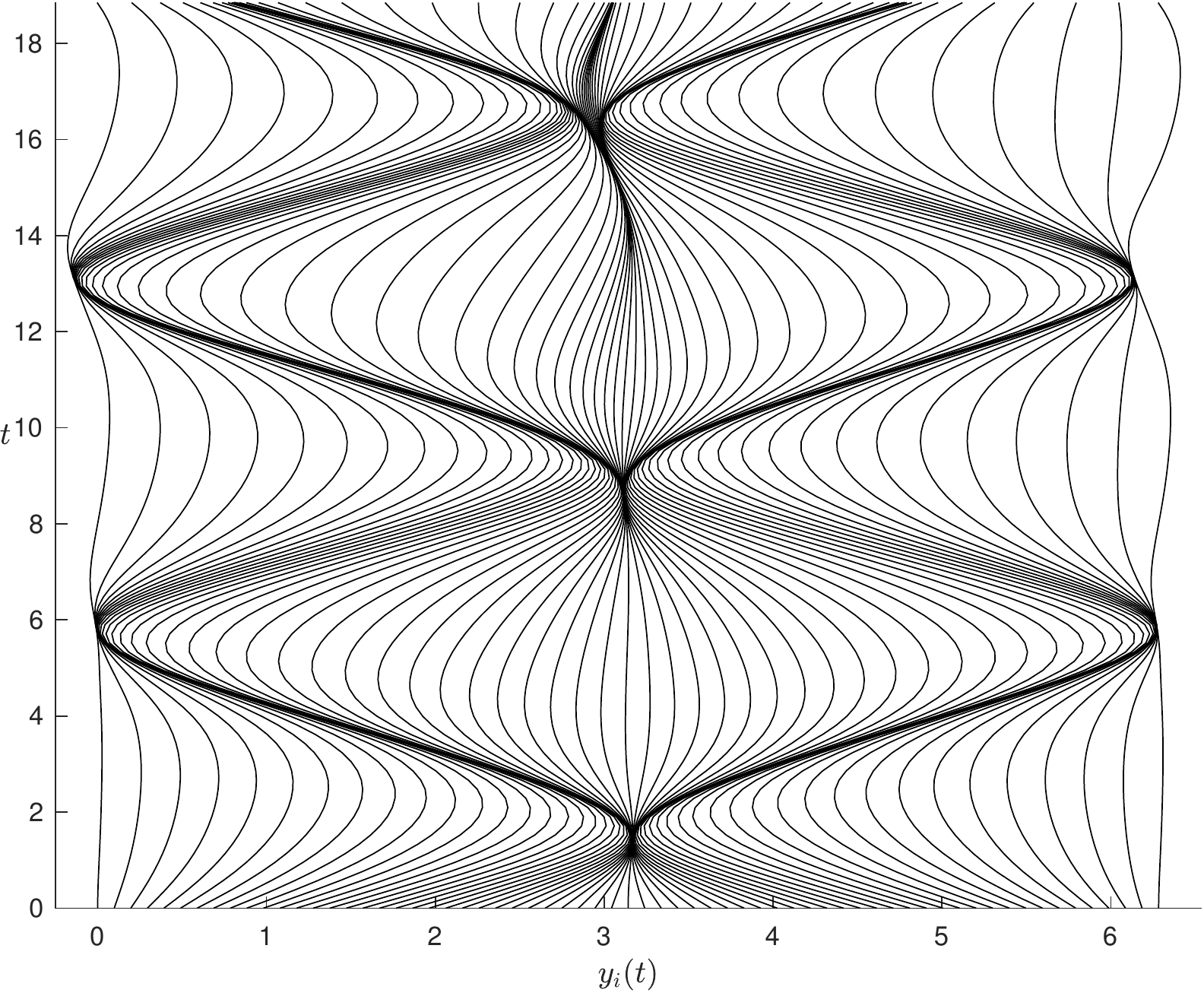}
    \subcaption{} \label{fig:sineCH_VD64}
  \end{subfigure}
  \caption{The characteristics for $u_0(x) = \sin(x)$ for the conservative multipeakon scheme with $n = 2^4$ (\textsc{a}) and $n = 2^6$ (\textsc{c}), and for the variational scheme for $n = 2^4$ (\textsc{b}) and $n = 2^6$ (\textsc{d}). Note that we have also plotted $y_n(t) = y_0(t) + \period$ to highlight the periodicity. }
  \label{fig:sineCH_chars}
\end{figure}

A striking difference in the methods is immediately seen from their characteristics for $n = 2^4$ and $n = 2^6$
displayed in Figure \ref{fig:sineCH_chars}.
Indeed, we find that the multipeakon method preserves the
symmetry of the characteristics, and this we have observed for all values of $n$ that we tested for whenever
$y_i(0)$ were equally spaced for $i \in \{0,\dots,n-1\}$.
In particular, the characteristics starting at $\xi_0 = 0$ and $\xi_{n/2} = \pi$ remain in the same position for
all $t$.
Indeed, this is a consequence of the fact that antisymmetry is preserved by \eqref{eq:CH}, cf.\ \cite{Camassa1993},
\cite[Rem.\ 4.2]{constantin1998global}.
On the other hand, for the variational scheme the characteristics have a slight drift which becomes more
pronounced over time.
This is especially clear for small $n$; compare for instance Figures \ref{fig:sineCH_VD16} and \ref{fig:sineCH_VD64}
where in the former the drift is apparent for $y_{n/2}$ from the start, while in the latter the characteristics
look more similar to those of Figure \ref{fig:sineCH_CMP64} for a longer time.
This introduces a small phase error at $t = 6\pi$ for $n = 2^6$. The clustered characteristics indicating the peaks
in Figure \ref{fig:sineCH_VD64} lie slightly to the left for the corresponding peaks in Figure
\ref{fig:sineCH_CMP64}.

To visualize the solution corresponding to Figure \ref{fig:sineCH_VD64}, we have plotted the interpolant $u_n$
at 30 equally spaced times in Figure \ref{fig:sineCH_VD}.
Here we see how the initial smooth profile breaks after about two seconds, leading to two peaked waves traveling in
opposite directions until they collide and reflect at the boundary.
Comparing with Figure \ref{fig:sineCH_VD64} we see how the emerging peaks and antipeaks
correspond to clustering of characteristics with very high density, while the locations
of smoother ridges and troughs coincide with less dense clusters.

\begin{figure}
  \includegraphics[width=0.8\textwidth]{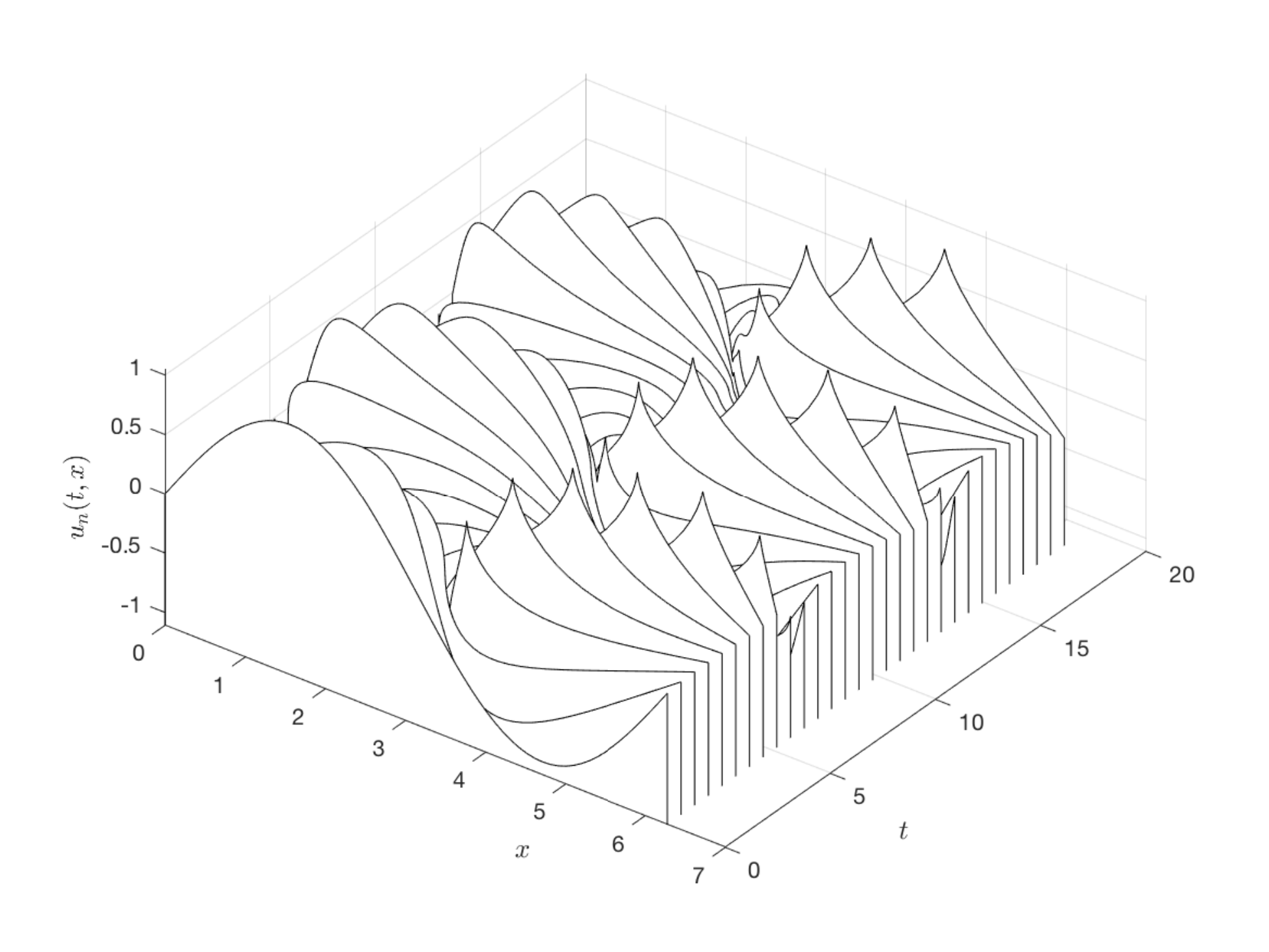}
  \caption{The interpolant $u_n$ for $u_0(x) = \sin(x)$ with $n = 2^6$ and 30 equally spaced times between $t = 0$ and $t = 6\pi$.}
  \label{fig:sineCH_VD}
\end{figure}

\begin{rem}\label{rem:symm}
  A natural question arising from this example is whether one could have chosen a different discrete energy as a
  starting point for the variational discretization in order to obtain a scheme which respects the preservation of
  antisymmetry.
  One could for instance try to use symmetric differences such as the central difference from \eqref{eq:diff}.
  However, there is the potential drawback of the oscillatory solutions associated with noncompact
  difference operators, cf.\ the discussion in \cite[p.\ 1929]{CohRay2011diff}.
  In fact, an early prototype of the scheme \eqref{eq:VD_H}, comprising only of \eqref{eq:semidisc_sys} solved
  as an ODE system with solution dependent mass matrix, exhibited severe oscillations in front of the peak when
  applied to the periodic peakon example after replacing $\D_+$ by $\D_0$.
  This indicates that some care has to be exercised when choosing the defining energy.
\end{rem}

\subsection{Example 6: Sine initial datum for the 2CH system}
Our final example illustrates how the solution of the previous example changes when we instead solve the
two-component Camassa--Holm system initial value problem \eqref{eq:2CH_IVP} by adding a positive density $\rho_0$ initially.
In this case we cannot compare with the multipeakon method \eqref{eq:cmp_per}, since multipeakons will not be
a solution of the 2CH system.
Moreover, this time we will use \eqref{eq:VD_h} instead of \eqref{eq:VD_H} to illustrate how the total energy is
preserved when $h_i$ rather than $H_i$ is used to track the energy of the system.

Thus, we choose $u_0(x) = \sin(x)$ and $\rho_0(x) \equiv 2$, and define initial data for \eqref{eq:VD_h}
through $y_i(0) = \xi_i = i 2\pi/n$, $U_i(0) = u_0(\xi_i)$, $r_i(0) \equiv 2$, and
$2h_i(0) = (U_i(0))^2 + (\D_+U_i(0))^2 + (r_i(0))^2$ for $i \in \{0, \dots, n-1\}$ and $n = 2^6$.
Once more we integrate from $t = 0$ to $t = 6\pi$ with \texttt{ode45} and
$\texttt{AbsTol} = \texttt{RelTol} = 10^{-10}$. The results are displayed in Figure \ref{fig:sine2CH_VD}.

\begin{figure}
  \begin{subfigure}[b]{0.4\textwidth}
    \includegraphics[width=1\textwidth]{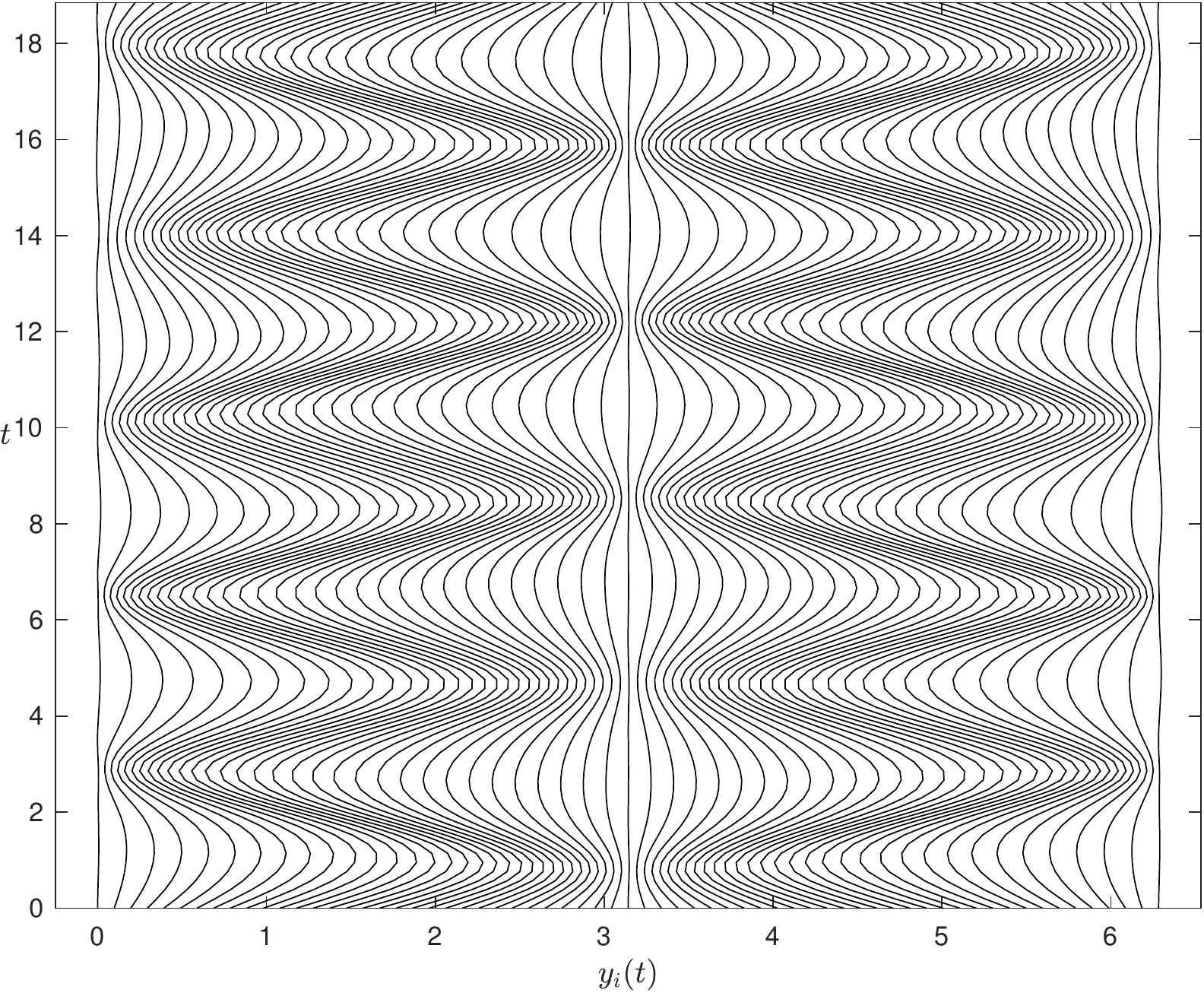}
    \subcaption{} \label{fig:sine2CH_chars}
  \end{subfigure}
  \begin{subfigure}[b]{0.59\textwidth}
    \includegraphics[width=1\textwidth]{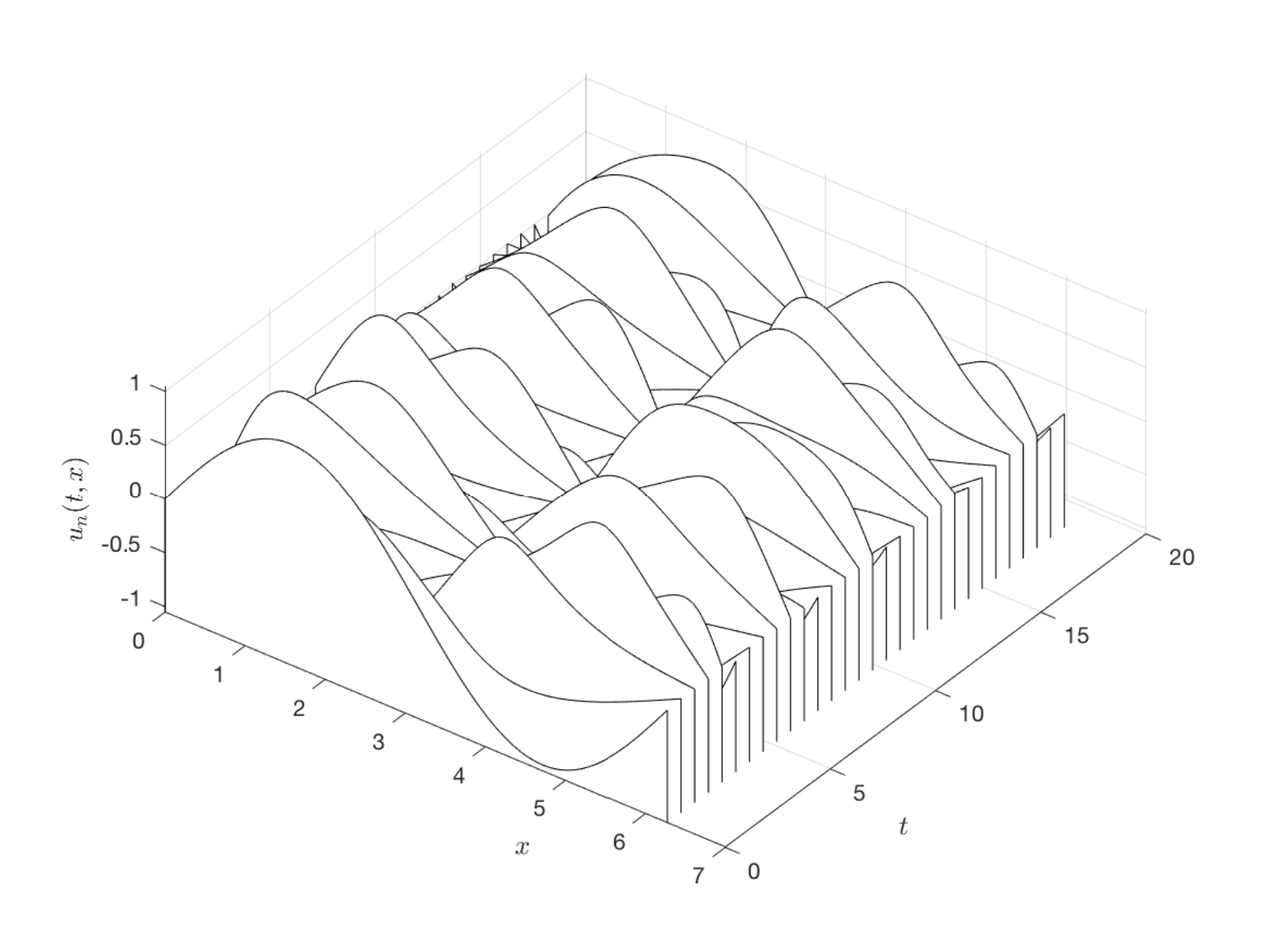}
    \subcaption{} \label{fig:sine2CH_VD_wf}
  \end{subfigure}
  \begin{subfigure}[b]{0.4\textwidth}
    \includegraphics[width=1\textwidth]{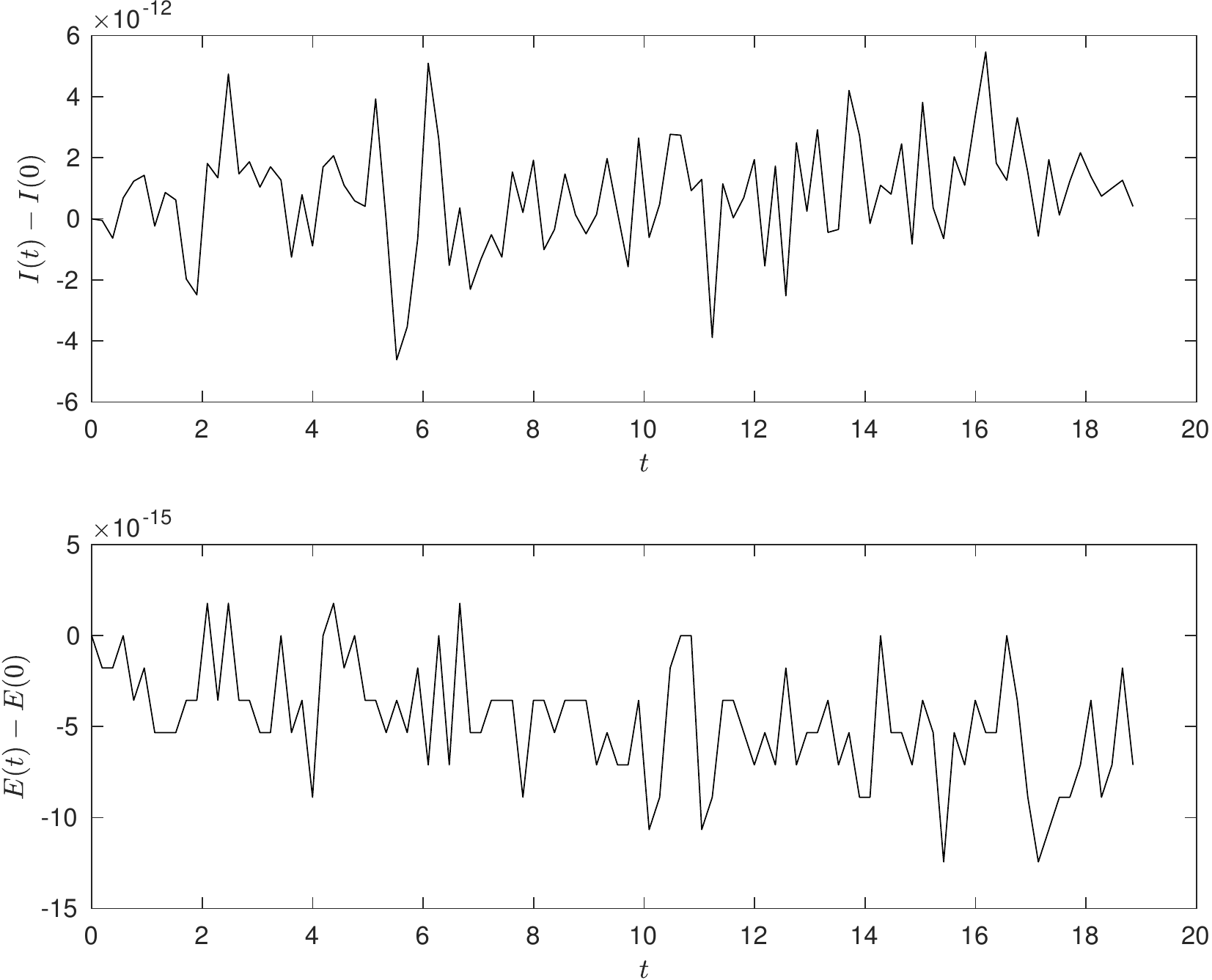}
    \subcaption{} \label{fig:sine2CH_VD_dev}
  \end{subfigure}
  \begin{subfigure}[b]{0.59\textwidth}
    \includegraphics[width=1\textwidth]{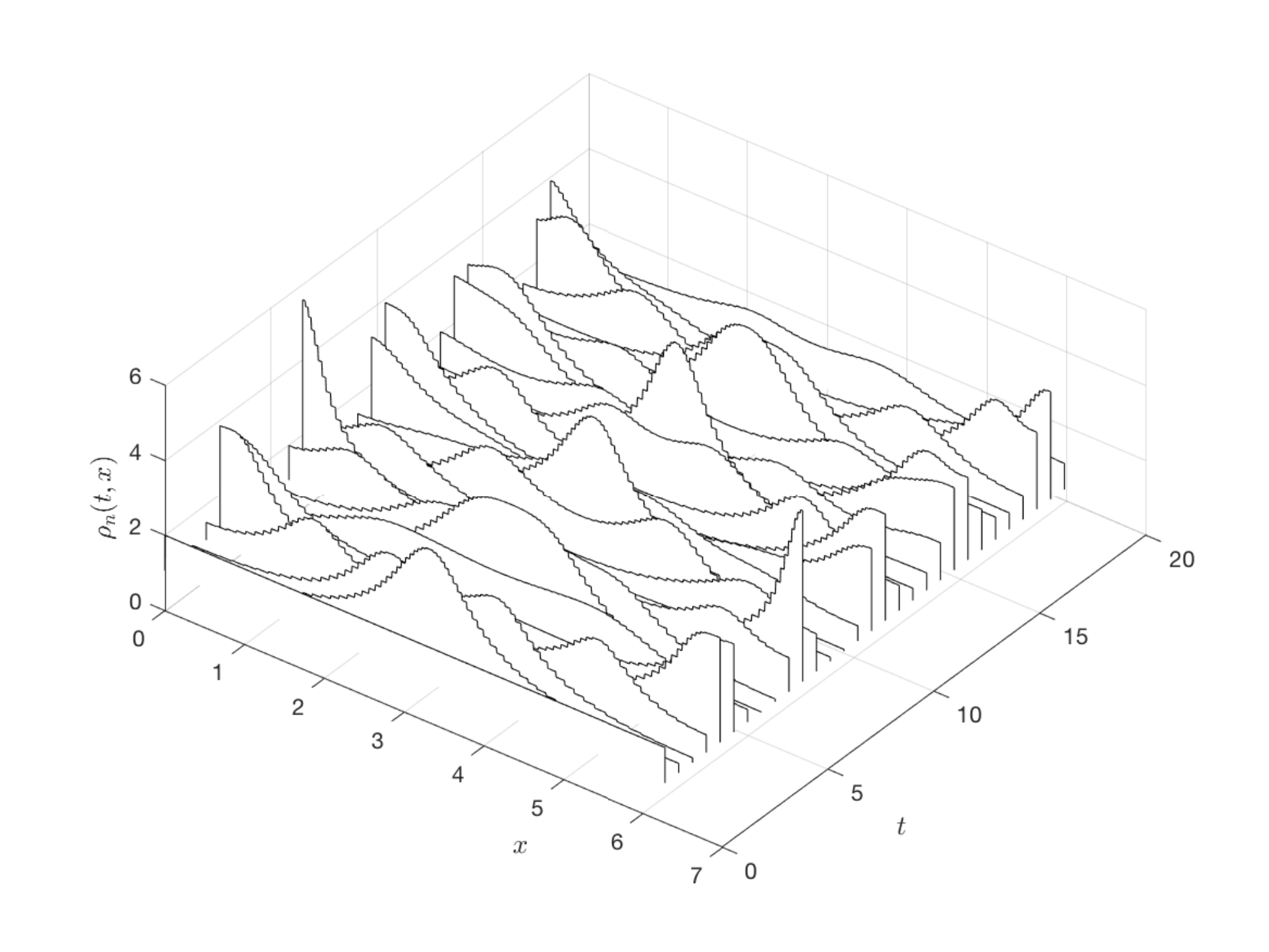}
    \subcaption{} \label{fig:sine2CH_VD_rho}
  \end{subfigure}
  \caption{Initial data $u_0(x) = \sin(x)$ and $\rho_0(x) \equiv 2$ for variational scheme with $n = 2^6$. Characteristics (\textsc{a}) and deviations $I(t) - I(0)$ and $E(t) - E(0)$ in momentum and energy respectively (\textsc{c}) for 100 equally spaced times between $t = 0$ and $t = 6\pi$. Parts (\textsc{b}) and (\textsc{d}) show the interpolants $u_n$ and $\rho_n$ respectively for 30 equally spaced times.}
  \label{fig:sine2CH_VD}
\end{figure}

Figure \ref{fig:sine2CH_chars} shows the characteristics which appear to have less drift in this case compared to
Figure \ref{fig:sineCH_VD64}.
Moreover, as expected from theory we see that there is no collision of characteristics, there is always a
positive distance between them.
The lack of singularity formation expressed in the plot of the interpolant $u_n$ in Figure \ref{fig:sine2CH_VD_wf}, where the wave profile remains smooth. There are no sharp peaks or antipeaks, only ridges
and troughs where the characteristics are dense.
In Figure \ref{fig:sine2CH_VD_rho} we see how the energy is transferred to the density $\rho_n$ when the crests and
troughs of $u_n$ meet, as opposed to the CH equation in Figure \ref{fig:sineCH_VD} where the energy is
concentrated in the point of collision.

Finally, Figure	\ref{fig:sine2CH_VD_dev} shows the deviation in momentum $I(t) - I(0)$ and in energy
$E(t) - E(0)$, where $I(t)$ is defined in \eqref{eq:disc_mom} and $E(t) = H_n(t)$ as defined by the sum in
\eqref{eq:H}.
In the other examples, $H_n$ was one of the solution variables for the variational scheme, just as for
the multipeakon scheme.
For both schemes, the corresponding evolution equation is $\dot{H}_n = 0$ and the energy is conserved by default.
Here the energy is a sum scaled by $\dxi$, hence it is a linear invariant of the ODE system.
Since linear invariants are preserved by any Runge--Kutta method, cf.\ \cite[Thm.~1.5]{GeoNumInt}, we expect
this deviation to be of the order of rounding error in our example.
Indeed, this is what we observe: for any tolerance we set for \texttt{ode45} we found the energy deviation to be of
order $10^{-15}$.
On the other hand, this is not the case for the total momentum $I(t)$, which is a sum of products of $U_i$ and
$\D_+y_i$ and thus a quadratic invariant.
In our results, the momentum deviation always scales with the tolerances of the solver. This is also the case in
Figure \ref{fig:sine2CH_VD_dev} where it is near the tolerance $10^{-10}$.
Speaking of these invariants, we point out the role reversal of the total momentum and energy for the
variational scheme and the invariant-preserving finite difference scheme:
the total energy is a linear invariant for \eqref{eq:VD_h} and a quadratic invariant for \eqref{eq:fd2CH},
while the total momentum is a quadratic invariant for \eqref{eq:VD_h} and a linear invariant for \eqref{eq:fd2CH}.

\section*{Summary}
We have applied the novel variational Lagrangian scheme \eqref{eq:VD_H} to several numerical examples.
In general it performed well and displayed consistent convergence rates.
In particular we saw rate 1 in both $\Ltwo$- and $\Hone$-norm for smooth reference solutions, while for the
more irregular peakon reference solutions we observed rate 1 in $\Ltwo$-norm and rate 0.5 in $\Hone$-norm.
Due to its rather simple discretization of the energy, it comes as no surprise
that other higher-order methods outperform \eqref{eq:VD_H} for smooth reference solutions.
However, it is for the more irregular examples involving wave breaking that this scheme stands apart, exhibiting
consistent convergence even in $\Hone$-norm where other methods may struggle with oscillations.

When it comes to extensions of this work there are several possible paths, and we mention those most apparent.
An obvious question is whether the scheme could be improved by choosing a more refined discrete energy for the
variational derivation, and if there are choices other than the multipeakons which lead to an integrable
discrete system.
Another extension is to make the method fully discrete, in the sense that one introduces a tailored method to
integrate in time, preferably one that respects the conserved quantities of the system.
Finally, one could consider developing a specific redistribution algorithm which can handle the potential
clustering of characteristics and prevent artificial numerical collisions when such a Lagrangian method is run
over long time intervals.


\subsection*{Acknowledgments:}
The authors would like to express their utmost gratitude to Xavier Raynaud for ideas and feedback which greatly improved this
paper.
The first author would also like to thank Elena Celledoni and S\o lve Eidnes for helpful discussions.


\begin{thebibliography}{10}
  
  \bibitem{AnDoMi2019}
  D.~C. Antonopoulos, V.~A. Dougalis, and D.~E. Mitsotakis.
  \newblock Error estimates for {G}alerkin finite element methods for the
  {C}amassa-{H}olm equation.
  \newblock {\em Numer. Math.}, 142(4):833--862, 2019.
  
  \bibitem{arnold1999topological}
  V.~I. Arnold and B.~A. Khesin.
  \newblock {\em Topological methods in hydrodynamics}, volume 125 of {\em
    Applied Mathematical Sciences}.
  \newblock Springer-Verlag, New York, 1998.
  
  \bibitem{NumSim_ArtSchr2006}
  R.~Artebrant and H.~J. Schroll.
  \newblock Numerical simulation of {C}amassa-{H}olm peakons by adaptive
  upwinding.
  \newblock {\em Appl. Numer. Math.}, 56(5):695--711, 2006.
  
  \bibitem{NumEvol}
  U.~M. Ascher.
  \newblock {\em Numerical methods for evolutionary differential equations},
  volume~5 of {\em Computational Science \& Engineering}.
  \newblock Society for Industrial and Applied Mathematics (SIAM), Philadelphia,
  PA, 2008.
  
  \bibitem{ChebSpec}
  J.~P. Boyd.
  \newblock {\em Chebyshev and {F}ourier spectral methods}.
  \newblock Dover Publications, Inc., Mineola, NY, second edition, 2001.
  
  \bibitem{bressan2007global}
  A.~Bressan and A.~Constantin.
  \newblock Global conservative solutions of the {C}amassa-{H}olm equation.
  \newblock {\em Arch. Ration. Mech. Anal.}, 183(2):215--239, 2007.
  
  \bibitem{Camassa2003}
  R.~Camassa.
  \newblock Characteristics and the initial value problem of a completely
  integrable shallow water equation.
  \newblock {\em Discrete Contin. Dyn. Syst. Ser. B}, 3(1):115--139, 2003.
  
  \bibitem{Camassa1993}
  R.~Camassa and D.~D. Holm.
  \newblock An integrable shallow water equation with peaked solitons.
  \newblock {\em Phys. Rev. Lett.}, 71(11):1661--1664, 1993.
  
  \bibitem{CamHuaLee2005}
  R.~Camassa, J.~Huang, and L.~Lee.
  \newblock On a completely integrable numerical scheme for a nonlinear
  shallow-water wave equation.
  \newblock {\em J. Nonlinear Math. Phys.}, 12(suppl. 1):146--162, 2005.
  
  \bibitem{Camassa2008}
  R.~Camassa and L.~Lee.
  \newblock Complete integrable particle methods and the recurrence of initial
  states for a nonlinear shallow-water wave equation.
  \newblock {\em J. Comput. Phys.}, 227(15):7206--7221, 2008.
  
  \bibitem{CheLiuPen2012}
  A.~Chertock, J.-G. Liu, and T.~Pendleton.
  \newblock Convergence of a particle method and global weak solutions of a
  family of evolutionary {PDE}s.
  \newblock {\em SIAM J. Numer. Anal.}, 50(1):1--21, 2012.
  
  \bibitem{CheLiuPen2015}
  A.~Chertock, J.-G. Liu, and T.~Pendleton.
  \newblock Elastic collisions among peakon solutions for the {C}amassa-{H}olm
  equation.
  \newblock {\em Appl. Numer. Math.}, 93:30--46, 2015.
  
  \bibitem{MR2192287}
  G.~M. Coclite, H.~Holden, and K.~H. Karlsen.
  \newblock Global weak solutions to a generalized hyperelastic-rod wave
  equation.
  \newblock {\em SIAM J. Math. Anal.}, 37(4):1044--1069, 2005.
  
  \bibitem{Coclite2008}
  G.~M. Coclite, K.~H. Karlsen, and N.~H. Risebro.
  \newblock A convergent finite difference scheme for the {C}amassa-{H}olm
  equation with general {$H^1$} initial data.
  \newblock {\em SIAM J. Numer. Anal.}, 46(3):1554--1579, 2008.
  
  \bibitem{Cohen2014}
  D.~Cohen, T.~Matsuo, and X.~Raynaud.
  \newblock A multi-symplectic numerical integrator for the two-component
  {C}amassa--{H}olm equation.
  \newblock {\em J. Nonlinear Math. Phys.}, 21(3):442--453, 2014.
  
  \bibitem{Cohen2008}
  D.~Cohen, B.~Owren, and X.~Raynaud.
  \newblock Multi-symplectic integration of the {C}amassa--{H}olm equation.
  \newblock {\em J. Comput. Phys.}, 227(11):5492--5512, 2008.
  
  \bibitem{CohRay2011diff}
  D.~Cohen and X.~Raynaud.
  \newblock Geometric finite difference schemes for the generalized
  hyperelastic-rod wave equation.
  \newblock {\em J. Comput. Appl. Math.}, 235(8):1925--1940, 2011.
  
  \bibitem{CohRay2012num}
  D.~Cohen and X.~Raynaud.
  \newblock Convergent numerical schemes for the compressible hyperelastic rod
  wave equation.
  \newblock {\em Numer. Math.}, 122(1):1--59, 2012.
  
  \bibitem{constantin1998global}
  A.~Constantin and J.~Escher.
  \newblock Global existence and blow-up for a shallow water equation.
  \newblock {\em Ann. Scuola Norm. Sup. Pisa Cl. Sci. (4)}, 26(2):303--328, 1998.
  
  \bibitem{constantin1998wave}
  A.~Constantin and J.~Escher.
  \newblock Wave breaking for nonlinear nonlocal shallow water equations.
  \newblock {\em Acta Math.}, 181(2):229--243, 1998.
  
  \bibitem{ConsIvan2008}
  A.~Constantin and R.~I. Ivanov.
  \newblock On an integrable two-component {C}amassa-{H}olm shallow water system.
  \newblock {\em Phys. Lett. A}, 372(48):7129--7132, 2008.
  
  \bibitem{cons:01b}
  A.~Constantin and B.~Kolev.
  \newblock Least action principle for an integrable shallow water equation.
  \newblock {\em J. Nonlinear Math. Phys.}, 8(4):471--474, 2001.
  
  \bibitem{MR2016696}
  A.~Constantin and B.~Kolev.
  \newblock Geodesic flow on the diffeomorphism group of the circle.
  \newblock {\em Comment. Math. Helv.}, 78(4):787--804, 2003.
  
  \bibitem{constantin2000global}
  A.~Constantin and L.~Molinet.
  \newblock Global weak solutions for a shallow water equation.
  \newblock {\em Comm. Math. Phys.}, 211(1):45--61, 2000.
  
  \bibitem{MR1811323}
  H.-H. Dai and Y.~Huo.
  \newblock Solitary shock waves and other travelling waves in a general
  compressible hyperelastic rod.
  \newblock {\em R. Soc. Lond. Proc. Ser. A Math. Phys. Eng. Sci.},
  456(1994):331--363, 2000.
  
  \bibitem{EckKos2014}
  J.~Eckhardt and A.~Kostenko.
  \newblock An isospectral problem for global conservative multi-peakon solutions
  of the {C}amassa-{H}olm equation.
  \newblock {\em Comm. Math. Phys.}, 329(3):893--918, 2014.
  
  \bibitem{EckKos2018}
  J.~Eckhardt and A.~Kostenko.
  \newblock The inverse spectral problem for periodic conservative multi-peakon
  solutions of the {C}amassa-{H}olm equation.
  \newblock {\em Int.~Mat.~Res.~Notices}, 07 2018.
  \newblock \url{https://doi.org/10.1093/imrn/rny176}.
  
  \bibitem{Fokas1981}
  B.~Fuchssteiner and A.~S. Fokas.
  \newblock Symplectic structures, their {B}\"{a}cklund transformations and
  hereditary symmetries.
  \newblock {\em Phys. D}, 4(1):47--66, 1981.
  
  \bibitem{vardisc}
  S.~T. Galtung and X.~Raynaud.
  \newblock Well-posedness and convergence of a variational discretization of the
  {C}amassa--{H}olm equation, 2020.
  \newblock \arxiv{2003.03114}.
  
  \bibitem{Grunert2015}
  K.~Grunert.
  \newblock Blow-up for the two-component {C}amassa-{H}olm system.
  \newblock {\em Discrete Contin. Dyn. Syst.}, 35(5):2041--2051, 2015.
  
  \bibitem{GruHol2016}
  K.~Grunert and H.~Holden.
  \newblock The general peakon-antipeakon solution for the {C}amassa-{H}olm
  equation.
  \newblock {\em J. Hyperbolic Differ. Equ.}, 13(2):353--380, 2016.
  
  \bibitem{Grunert2012}
  K.~Grunert, H.~Holden, and X.~Raynaud.
  \newblock Global solutions for the two-component {C}amassa--{H}olm system.
  \newblock {\em Comm. Partial Differential Equations}, 37(12):2245--2271, 2012.
  
  \bibitem{Grunert2013}
  K.~Grunert, H.~Holden, and X.~Raynaud.
  \newblock Periodic conservative solutions for the two-component
  {C}amassa--{H}olm system.
  \newblock In {\em Spectral analysis, differential equations and mathematical
    physics: a festschrift in honor of {F}ritz {G}esztesy's 60th birthday},
  volume~87 of {\em Proc. Sympos. Pure Math.}, pages 165--182. Amer. Math.
  Soc., Providence, RI, 2013.
  
  \bibitem{grunert2020numerical}
  K.~Grunert, A.~Nordli, and S.~Solem.
  \newblock Numerical conservative solutions of the {H}unter--{S}axton equation,
  2020.
  \newblock \arxiv{2005.03882}.
  
  \bibitem{GeoNumInt}
  E.~Hairer, C.~Lubich, and G.~Wanner.
  \newblock {\em Geometric numerical integration}, volume~31 of {\em Springer
    Series in Computational Mathematics}.
  \newblock Springer-Verlag, Berlin, second edition, 2006.
  \newblock Structure-preserving algorithms for ordinary differential equations.
  
  \bibitem{Holden2006}
  H.~Holden and X.~Raynaud.
  \newblock Convergence of a finite difference scheme for the {C}amassa-{H}olm
  equation.
  \newblock {\em SIAM J. Numer. Anal.}, 44(4):1655--1680, 2006.
  
  \bibitem{Raynaud2006}
  H.~Holden and X.~Raynaud.
  \newblock A convergent numerical scheme for the {C}amassa-{H}olm equation based
  on multipeakons.
  \newblock {\em Discrete Contin. Dyn. Syst.}, 14(3):505--523, 2006.
  
  \bibitem{holden2007globalmp}
  H.~Holden and X.~Raynaud.
  \newblock Global conservative multipeakon solutions of the {C}amassa-{H}olm
  equation.
  \newblock {\em J. Hyperbolic Differ. Equ.}, 4(1):39--64, 2007.
  
  \bibitem{holden2007global}
  H.~Holden and X.~Raynaud.
  \newblock Global conservative solutions of the {C}amassa-{H}olm equation---a
  {L}agrangian point of view.
  \newblock {\em Comm. Partial Differential Equations}, 32(10-12):1511--1549,
  2007.
  
  \bibitem{MR2292515}
  H.~Holden and X.~Raynaud.
  \newblock Global conservative solutions of the generalized hyperelastic-rod
  wave equation.
  \newblock {\em J. Differential Equations}, 233(2):448--484, 2007.
  
  \bibitem{NumMP_HolRay2008}
  H.~Holden and X.~Raynaud.
  \newblock A numerical scheme based on multipeakons for conservative solutions
  of the {C}amassa-{H}olm equation.
  \newblock In {\em Hyperbolic problems: theory, numerics, applications}, pages
  873--881. Springer, Berlin, 2008.
  
  \bibitem{HolRay2008}
  H.~Holden and X.~Raynaud.
  \newblock Periodic conservative solutions of the {C}amassa-{H}olm equation.
  \newblock {\em Ann. Inst. Fourier (Grenoble)}, 58(3):945--988, 2008.
  
  \bibitem{Kalisch2005}
  H.~Kalisch and J.~Lenells.
  \newblock Numerical study of traveling-wave solutions for the {C}amassa--{H}olm
  equation.
  \newblock {\em Chaos Solitons Fractals}, 25(2):287--298, 2005.
  
  \bibitem{Kalisch2006}
  H.~Kalisch and X.~Raynaud.
  \newblock Convergence of a spectral projection of the {C}amassa-{H}olm
  equation.
  \newblock {\em Numer. Methods Partial Differential Equations},
  22(5):1197--1215, 2006.
  
  \bibitem{MR1674267}
  S.~Kouranbaeva.
  \newblock The {C}amassa-{H}olm equation as a geodesic flow on the
  diffeomorphism group.
  \newblock {\em J. Math. Phys.}, 40(2):857--868, 1999.
  
  \bibitem{MR2125239}
  J.~Lenells.
  \newblock Conservation laws of the {C}amassa-{H}olm equation.
  \newblock {\em J. Phys. A}, 38(4):869--880, 2005.
  
  \bibitem{Lenells2005}
  J.~Lenells.
  \newblock Traveling wave solutions of the {C}amassa--{H}olm equation.
  \newblock {\em J. Differential Equations}, 217(2):393--430, 2005.
  
  \bibitem{LiuPen2016}
  H.~Liu and T.~Pendleton.
  \newblock On invariant-preserving finite difference schemes for the
  {C}amassa-{H}olm equation and the two-component {C}amassa-{H}olm system.
  \newblock {\em Commun. Comput. Phys.}, 19(4):1015--1041, 2016.
  
  \bibitem{Matsuo2010}
  T.~Matsuo.
  \newblock A {H}amiltonian-conserving {G}alerkin scheme for the {C}amassa-{H}olm
  equation.
  \newblock {\em J. Comput. Appl. Math.}, 234(4):1258--1266, 2010.
  
  \bibitem{MatrixApplied}
  C.~Meyer.
  \newblock {\em Matrix analysis and applied linear algebra}.
  \newblock Society for Industrial and Applied Mathematics (SIAM), Philadelphia,
  PA, 2000.
  
  \bibitem{Molinari2008}
  L.~G. Molinari.
  \newblock Determinants of block tridiagonal matrices.
  \newblock {\em Linear Algebra Appl.}, 429(8-9):2221--2226, 2008.
  
  \bibitem{Olver1996}
  P.~J. Olver and P.~Rosenau.
  \newblock Tri-{H}amiltonian duality between solitons and solitary-wave
  solutions having compact support.
  \newblock {\em Phys. Rev. E (3)}, 53(2):1900--1906, 1996.
  
  \bibitem{JacobiOperator}
  G.~Teschl.
  \newblock {\em Jacobi operators and completely integrable nonlinear lattices},
  volume~72 of {\em Mathematical Surveys and Monographs}.
  \newblock American Mathematical Society, Providence, RI, 2000.
  
  \bibitem{SpectralMatlab}
  L.~N. Trefethen.
  \newblock {\em Spectral methods in {MATLAB}}, volume~10 of {\em Software,
    Environments, and Tools}.
  \newblock Society for Industrial and Applied Mathematics (SIAM), Philadelphia,
  PA, 2000.
  
  \bibitem{Wahlen2006}
  E.~Wahl\'{e}n.
  \newblock The interaction of peakons and antipeakons.
  \newblock {\em Dyn. Contin. Discrete Impuls. Syst. Ser. A Math. Anal.},
  13(3-4):465--472, 2006.
  
  \bibitem{XuShu2008}
  Y.~Xu and C.-W. Shu.
  \newblock A local discontinuous {G}alerkin method for the {C}amassa-{H}olm
  equation.
  \newblock {\em SIAM J. Numer. Anal.}, 46(4):1998--2021, 2008.
  
\end{thebibliography}
\end{document}